\documentclass[a4paper, 10pt, notitlepage]{article}

\usepackage{amsthm, amsmath, amssymb, latexsym}
\usepackage{mathrsfs,cite}
\usepackage[multiple]{footmisc}
\usepackage{graphicx}

\usepackage[a4paper]{geometry}	

\theoremstyle{plain}
\newtheorem{theorem}{Theorem}[section]	% The counter starts from section number ([section]).
\newtheorem{lemma}[theorem]{Lemma}	% The same counter as in theorem ([theorem]).
\newtheorem{corollary}[theorem]{Corollary}
\newtheorem{proposition}[theorem]{Proposition}

\theoremstyle{definition}
\newtheorem{definition}[theorem]{Definition}
\newtheorem{example}[theorem]{Example}

\theoremstyle{remark}
\newtheorem{remark}[theorem]{Remark}

\DeclareMathOperator{\co}{co}
\DeclareMathOperator{\cl}{cl}
\DeclareMathOperator{\cone}{cone}
\DeclareMathOperator{\interior}{int}
\DeclareMathOperator{\relint}{ri}
\DeclareMathOperator{\dimens}{dim}
\DeclareMathOperator{\linhull}{span}
\DeclareMathOperator{\sign}{sign}
\DeclareMathOperator{\argmin}{arg\,min}
\DeclareMathOperator{\determ}{det}
\DeclareMathOperator{\rank}{rank}
\DeclareMathOperator{\diag}{diag}
\DeclareMathOperator{\trace}{Tr}
\DeclareMathOperator{\support}{supp}
\DeclareMathOperator{\dist}{dist}
\DeclareMathOperator{\affine}{aff}

\author{Dolgopolik M.V.\footnote{Institute for Problems in Mechanical Engineering, Saint Petersburg, Russia}
\footnote{The results presented in this article were supported by the President of Russian Federation grant for the
support of young Russian scientists (grant number MK-3621.2019.1).}}
\title{A Unified Study of Necessary and Sufficient Optimality Conditions for Minimax and Chebyshev Problems with Cone
Constraints}

\begin{document}

\maketitle

\begin{abstract}
We present a unified study of first and second order necessary and sufficient optimality conditions for minimax and
Chebyshev optimisation problems with cone constraints. First order optimality conditions for such problems can be
formulated in several different forms: in terms of a linearised problem, in terms of Lagrange multipliers (KKT-points),
in terms of subdifferentials and normal cones, in terms of a nonsmooth penalty function, in terms of cadres with
positive cadre multipliers, and in an alternance form. We describe interconnections between all these forms of necessary
and sufficient optimality conditions and prove that seemingly different conditions are in fact equivalent. We also
demonstrate how first order optimality conditions can be reformulated in a more convenient form for particular classes
of cone constrained optimisation problems and extend classical second order optimality condition for smooth cone
constrained problems to the case of minimax and Chebyshev problems with cone constraints. The optimality conditions
obtained in this article open a way for a development of new efficient structure-exploiting methods for solving cone
constrained minimax and Chebyshev problems.
\end{abstract}

\section{Introduction}

It is well-known that discrete minimax problems and discrete Chebyshev problems (problems of best
$\ell_{\infty}$-approximation) can be reduced to equivalent nonlinear programming problems. Many methods for solving
minimax problems are based on application of nonlinear programming algorithms to these equivalent reformulations of
minimax problems (see such methods based on, e.g. sequential quadratic programming methods
\cite{RustemNguyen98,YuGao2002,JianQuanZhang2007,HuChenLi2009}, sequential quadratically constrained quadratic
programming methods \cite{ChaoWangLiagnHu,JianChao2010,JianMoQiu2014}, interior point methods
\cite{RustemZakovicParpas,ObasanTzalRustem}, augmented Lagrangian methods \cite{HeZhou2011,HeNie2013,HeLiuWang2016}, 
etc.). On the other hand, efficient, superlinearly or even quadratically convergent methods for solving minimax problems
can be also based on a convenient characterisation of an optimal solution of a minimax problem, that is, on optimality
conditions that are specific for minimax or Chebyshev problems (cf. such methods for discrete minimax problems
\cite{ConnLi92}, problems of rational $\ell_{\infty}$-approximation \cite{BarrodalePowellRoberts}, and synthesis of a
rational filter \cite{MalozemovTamasyan}). To extend such methods to the case of minimax and Chebyshev problems with
cone constraints (e.g. problems with semidefinite or semi-infinite constraints), first and second order optimality
conditions for such problems are needed.

Optimality conditions for general smooth optimisation problems with cone constraints and their particular classes were
studied in detail in multiple papers and monographs
\cite{Kawasaki,Cominetti,Shapiro97,BonComShap98,BonComShap99,BonnansShapiro,BonnansRamirez,Shapiro2009}. In the
nonsmooth case, much less attention has been paid to this subject. Optimality conditions for general nonsmooth
optimisation problems with cone constraints were studied in \cite{MordukhovichNghia}. Optimality conditions for
nonsmooth semidefinite programming problems were obtained in \cite{ZhaoGao2006,GolestaniNobakhtian2015,Tung}, while in
the case of nonsmooth semi-infinite programming problems they were analysed in 
\cite{ZhenYang2007,KanziNobakhtian,Kanzi2011,CaristiFerrara,Gadhi}. However, to the best of the author's knowledge
optimality conditions for minimax problems and Chebyshev problems (problems of best $\ell_{\infty}$-approximation) with
cone constraints have not been thoroughly analysed in the literature.

In the case of unconstrained problems, optimality conditions for minimax problems can be formulated in many seemingly
non-equivalent forms some of which are not very well-known to researchers and relatively unusual in the context of
nonsmooth optimisation. In particular, optimality conditions for minimax problems can be formulated in terms of
so-called \textit{cadres} of minimax problems \cite{Descloux,ConnLi92} or in an \textit{alternance} form 
\cite{MalozemovPevnyi,DaugavetMalozemov75,Daugavet,DaugavetMalozemov81,DaugavetMalozemov79,Malozemov77}, which is often
used within approximation theory \cite{Rice,Cheney}. In \cite{DemyanovMalozemov_Alternance,DemyanovMalozemov_Collect} it
was shown that the classical optimality condition $0 \in \partial f(x)$, where $\partial f(x)$ is some convex
subdifferential, can be rewritten in an alternance form. However, interconnections between various types of optimality
conditions for minimax and Chebyshev problems (particularly, sufficient optimality conditions and optimality conditions
for constrained minimax problems) have not been analysed before.

The main goal of this paper is to present a unified study of various types of optimality conditions for minimax and
Chebyshev problems with cone constraints scattered in the literature. Namely, we study six different forms of first
order necessary and sufficient optimality conditions for such problems (conditions involving a linearised problem,
Lagrange multipliers, subdifferentials and normal cones, $\ell_1$ penalty function, cadres, and alternance conditions)
and show that all these conditions are equivalent. We also demonstrate how they can be refined for particular types of
cone constraints, namely, for problems with equality and inequality constraints, problems with second order cone
constraints, as well as problems with semidefinite and semi-infinite constraints. Finally, we show how well-known 
necessary and sufficient second order optimality conditions for cone constrained optimisation problems can be extended
to the case of minimax and Chebyshev problems and present several examples illustrating theoretical results.

It should be noted that although some results presented in this paper are straightforward generalisations of
corresponding results for smooth cone constrained optimisation problems to the minimax setting (e.g. optimality
conditions in terms of a linearised problem and Lagrange multipliers, Section~\ref{subsect:LagrangeMultipliers}, or
second order optimality conditions, Section~\ref{sect:SecondOrderOptCond}), many other results are completely new. In
particular, to the best of the author's knowledge interconnections between various forms of sufficient optimality
conditions for minimax problems and complete alternance (Thrms.~\ref{thrm:EquivOptCond_Subdiff} and
\ref{thrm:EquivOptCond_PenaltyFunc} and Section~\ref{subsect:Alternance_Cadre}), as well as alternance optimality
conditions for particular classes of minimax problems with cone constraints (Section~\ref{subsect:Examples}), have not
been studied before. 

The paper is organised as follows. In Section~\ref{sect:FirstOrderOptCond}, we study various forms of first order 
necessary and sufficient optimality conditions for cone constrained minimax problems.
Section~\ref{subsect:LagrangeMultipliers} is devoted to optimality conditions in terms of a linearised problem and
Lagrange multipliers. Optimality conditions involving subdifferentials, normal cones and a nonsmooth penalty function
are contained in Section~\ref{subsect:Subdifferentials_ExactPenaltyFunc}, while optimality conditions in terms of cadres
and in an alternance form are studied in Section~\ref{subsect:Alternance_Cadre}. A more detailed analysis of first order
optimality conditions for particular classes of cone constrained minimax problems is given in
Section~\ref{subsect:Examples}. Finally, Section~\ref{sect:SecondOrderOptCond} is devoted to second order necessary and
sufficient optimality conditions, while optimality conditions for Chebyshev (uniform approximation) problems are
discussed in Section~\ref{sect:ChebyshevProblems}.

\section{First order optimality conditions for cone constrained minimax problems}
\label{sect:FirstOrderOptCond}

Let $A \subseteq \mathbb{R}^d$ be a nonempty closed convex set, $Y$ be a Banach space, and $K \subset Y$ be a nonempty
closed convex cone. Denote by $Y^*$ the topological dual of $Y$, and by $\langle \cdot, \cdot \rangle$ either 
the canonical duality pairing between $Y$ and its dual or the inner product in $\mathbb{R}^s$, $s \in \mathbb{N}$,
depending on the context.

Let $W$ be a compact Hausdorff topological space, and $f \colon \mathbb{R}^d \times W \to \mathbb{R}$ and 
$G \colon \mathbb{R}^d \to Y$ be given functions. Throughout this article we suppose that the function 
$f = f(x, \omega)$ is differentiable in $x$ for any $\omega \in W$, and the functions $f$ and $\nabla_x f$ are
continuous jointly in $x$ and $\omega$, while $G$ is continuously Fr\'echet differentiable. However, for the main
results below to hold true it is sufficient to suppose that $f(x, \omega)$ is continuous and continuously differentiable
in $x$ only on $\mathcal{O}(x_*) \times W$, and $G$ is continuously Fr\'echet differentiable on $\mathcal{O}(x_*)$,
where $\mathcal{O}(x_*)$ is a neighbourhood of a given point $x_*$.

Denote $F(x) = \max_{\omega \in W} f(x, \omega)$ for any $x \in \mathbb{R}^d$. Hereinafter we study the following cone
constrained minimax problem:
$$
  \min F(x) \quad \text{subject to} \quad G(x) \in K, \quad x \in A.	\eqno{(\mathcal{P})}
$$
Our aim is obtain several different forms of first order necessary and sufficient optimality conditions for the problem
$(\mathcal{P})$ and analyse how they relate to each other.

\subsection{Lagrange multipliers and first order growth condition}
\label{subsect:LagrangeMultipliers}

Let us start with an analysis of necessary and sufficient optimality conditions for the problem $(\mathcal{P})$
involving Lagrange multipliers. The main results of this subsection are a straightforward extension of the first order
necessary optimality conditions for cone constrained optimisation problems from \cite[Sect.~3.1]{BonnansShapiro} to the
case of cone constrained \textit{minimax} problems.

Firstly, we apply a standard linearisation procedure to the problem $(\mathcal{P})$ in order to reduce an analysis of
optimality conditions to the convex case. Then with the use of the linearised convex problem we obtain optimality 
conditions involving Lagrange multipliers. To this end, we utilise the well-known \textit{Robinson's constraint
qualification} (RCQ) (see~\cite{Robinson75,Robinson76}). 

Recall that RCQ is said to hold at a feasible point $x_*$ of the problem $(\mathcal{P})$, if
\begin{equation} \label{eq:RCQ}
  0 \in \interior\Big\{ G(x_*) + D G(x_*)\big( A - x_* \big) - K \Big\},
\end{equation}
where $D G(x_*)$ is the Fr\'echet derivative of $G$ at $x_*$ and $\interior C$ stands for the topological interior of 
a set $C$. RCQ allows one to easily compute the contingent (Bouligand tangent) cone to the feasible set of the problem
$(\mathcal{P})$.

Recal that \textit{the contingent cone} to a subset $C$ of a normed space $X$ at a point $x_* \in C$, denoted by
$T_C(x_*)$, consists of all those vectors $h \in X$ for which one can find sequences 
$\{ \alpha_n \} \subset (0, + \infty)$ and $\{ h_n \} \subset X$ such that $\alpha_n \to 0$ and $h_n \to h$ as 
$n \to \infty$, and $x_* + \alpha_n h_n \in C$ for all $n \in \mathbb{N}$.

Denote by $\Omega = \{ x \in A \mid G(x) \in K \}$ the feasible region of the problem $(\mathcal{P})$. The following
lemma on the contingent cone to the set $\Omega$ is well-known. Nevertheless, we present its proof for the sake of
completeness.

\begin{lemma} \label{lem:ContingConeToFeasibleSet}
Let RCQ hold true at a feasible point $x_*$ of the problem $(\mathcal{P})$. Then
\begin{equation} \label{eq:ContingConeToFeasibleSet}
  T_{\Omega}(x_*) = \{ h \in T_A(x_*) \colon D G(x_*) h \in T_K(G(x_*)) \}.
\end{equation}
\end{lemma}

\begin{proof}
Introduce a function $\Phi \colon \mathbb{R}^d \to \mathbb{R}^d \times Y$ by setting $\Phi(x) = (x, G(x))$ for any 
$x \in \mathbb{R}^d$. Clearly, $\Omega = \{ x \in \mathbb{R}^d \mid \Phi(x) \in A \times K \}$. By
\cite[Lemma~2.100]{BonnansShapiro} RCQ implies that
$$
  0 \in \interior\big\{ \Phi(x_*) + D \Phi(x_*)\big( \mathbb{R}^d \big) - A \times K \big\}.
$$
Hence with the use of \cite[Corollary~2.91]{BonnansShapiro} one obtains that
\begin{equation} \label{eq:ContingentConeViaDerivative}
  T_{\Omega}(x_*) = \big\{ h \in \mathbb{R}^d \bigm| D \Phi(x_*) h \in T_{A \times K} (\Phi(x_*)) \big\}.
\end{equation}
One can easily check that $T_{A \times K}(\Phi(x_*)) \subseteq T_A(x_*) \times T_K(G(x_*))$. On the other hand, if
$h \in T_A(x_*)$, then there exist sequences $\{ \alpha_n \} \subset (0, + \infty)$ and $\{ h_n \} \subset \mathbb{R}^d$
such that $\alpha_n \to 0$ and $h_n \to h$ as $n \to \infty$, and $x_* + \alpha_n h_n \in A$ for all $n \in \mathbb{N}$.
Consequently, for all $n \in \mathbb{N}$ one has $(x_* + \alpha_n h_n, G(x_*)) \in A \times K$ and 
$(h, 0) \in T_{A \times K}(\Phi(x_*))$. Similarly, for any $w \in T_K(G(x_*))$ one has 
$(0, w) \in T_{A \times K}(\Phi(x_*))$. Since $A \times K$ is a convex set, the contingent cone 
$T_{A \times K}(\Phi(x_*))$ is convex. Therefore for all $h \in T_A(x_*)$ and $w \in T_K(G(x_*))$ one has
$(h, w) = (h, 0) + (w, 0) \in T_{A \times K}(\Phi(x_*))$, which implies that
$T_{A \times K}(\Phi(x_*)) = T_A(x_*) \times T_K(G(x_*))$. Hence bearing in mind
\eqref{eq:ContingentConeViaDerivative} and the fact that $D \Phi(x_*) h = (h, D G(x_*) h)$ one obtains that equality
\eqref{eq:ContingConeToFeasibleSet} holds true.
\end{proof}

Let $K^* = \{ y^* \in Y^* \mid \langle y^*, y \rangle \le 0 \: \forall y \in K \}$ \textit{the polar cone} of
$K$ and $L(x, \lambda) = F(x) + \langle \lambda, G(x) \rangle$ be the Lagrangian for the problem $(\mathcal{P})$.
Recall that a vector $\lambda_* \in Y^*$ is called \textit{a Lagrange multiplier} of $(\mathcal{P})$ at a feasible point
$x_*$, if $\lambda_* \in K^*$, $\langle \lambda_*, G(x_*) \rangle = 0$, and 
$[L(\cdot, \lambda_*)]' (x_*, h) \ge 0$ for all $h \in T_A(x_*)$, where $[L(\cdot, \lambda_*)]' (x_*, h)$ is the
directional derivative of the function $L(\cdot, \lambda_*)$ at $x_*$ in the direction $h$. Finally, if $\lambda_*$ is a
Lagrange multiplier of $(\mathcal{P})$ at a feasible point $x_*$, then the pair $(x_*, \lambda_*)$ is called 
\textit{a KKT-pair} of the problem $(\mathcal{P})$.

\begin{theorem} \label{thrm:NessOptCond}
Let $x_*$ be a locally optimal solution of the problem $(\mathcal{P})$ and RCQ hold at $x_*$. Then: 
\begin{enumerate}
\item{$h = 0$ is a globally optimal solution of the linearised problem
\begin{equation} \label{probl:LinearisedProblem}
  \min_{h \in \mathbb{R}^d} \max_{\omega \in W(x_*)} \langle \nabla_x f(x_*, \omega), h \rangle \quad
  \text{subject to} \quad \enspace D G(x_*) h \in T_K\big( G(x_*) \big), \quad h \in T_A(x_*),
\end{equation}
where $W(x_*) = \{ \omega \in W \mid f(x_*, \omega) = F(x_*) \}$;
\label{stat:NessOpt_LinearisedProblem}}

\item{the set of Lagrange multipliers at $x_*$ is a nonempty, convex, bounded, and weak${}^*$ compact subset of $Y^*$.
\label{stat:NessOpt_LagrangeMult}}
\end{enumerate}
\end{theorem}

\begin{proof}
\textbf{Part~\ref{stat:NessOpt_LinearisedProblem}.} Fix an arbitrary $h \in T_{\Omega}(x_*)$. By definition there exist
sequences $\{ \alpha_n \} \subset (0, + \infty)$ and $\{ h_n \} \subset \mathbb{R}^d$ such that $\alpha_n \to 0$ and
$h_n \to h$ as $n \to \infty$, and $x_* + \alpha_n h_n \in \Omega$ for all $n \in \mathbb{N}$.

As is well-known (see, e.g. \cite[Thrm.~4.4.3]{IoffeTihomirov}), from the fact that the function $f(x, \omega)$ is
differentiable in $x$, and the gradient $\nabla_x f(x, \omega)$ is continuous jointly in $x$ and $\omega$ it follows
that the function $F(x) = \max_{\omega \in W} f(x, \omega)$ is Hadamard directionally differentiable at $x_*$ and for
any $h \in \mathbb{R}^d$ its Hadamard directional derivative at $x_*$ has the from
\begin{equation} \label{eq:DirectDerivOfMaxFunc}
  F'(x_*, h) = \lim_{[h', \alpha] \to [h, +0]} \frac{F(x_* + \alpha h') - F(x_*)}{\alpha} =
  \max_{\omega \in W(x_*)} \langle \nabla_x f(x_*, \omega), h \rangle.
\end{equation}
Recall that $x_*$ is a locally optimal solution of the problem $(\mathcal{P})$. Therefore, for any sufficiently large 
$n \in \mathbb{N}$ one has $F(x_* + \alpha_n h_n) \ge F(x_*)$, which implies that
$$
  F'(x_*, h) = \lim_{n \to \infty} \frac{F(x_* + \alpha_n h_n) - F(x_*)}{\alpha_n} \ge 0.
$$
Thus, one has
$$
  F'(x_*, h) = \max_{\omega \in W(x_*)} \langle \nabla_x f(x_*, \omega), h \rangle \ge 0 \quad
  \forall \, h \in T_{\Omega}(x_*),
$$
which thanks to Lemma~\ref{lem:ContingConeToFeasibleSet} implies that $h = 0$ is a globally optimal solution of the
linearised problem \eqref{probl:LinearisedProblem}.

\textbf{Part~\ref{stat:NessOpt_LagrangeMult}.} Clearly, problem \eqref{probl:LinearisedProblem} is a \textit{convex}
cone constrained optimisation problem. For any $h \in \mathbb{R}^d$ and $\lambda \in Y^*$ denote by 
$L_0(h, \lambda) = F'(x_*, h) + \langle \lambda, D G(x_*) h \rangle$ the standard Lagrangian for this problem. Observe
that for all $h \in \mathbb{R}^d$ one has $L_0(h, \lambda) = [L(\cdot, \lambda)]'(x_*, h)$ .

From the facts that the sets $A$ and $K$ convex and $x_*$ is a feasible point it follows that 
$A - x_* \subseteq T_A(x_*)$ and $K - G(x_*) \subseteq T_K(G(x_*))$ (choose any sequence 
$\{ \alpha_n \} \subset (0, 1)$ converging to zero and for any $n \in \mathbb{N}$ define $h_n = z - x_*$ for $z \in A$
or $h_n = z - G(x_*)$ for $z \in K$). Hence RCQ (see~\eqref{eq:RCQ}) implies that
$$
  0 \in \interior\Big\{ D G(x_*) \big( T_A(x_*) \big) - T_K(G(x_*)) \Big\},
$$
i.e. the standard regularity condition (Slater's condition) for problem \eqref{probl:LinearisedProblem} holds true (see,
e.g. \cite[Formula~(3.12)]{BonnansShapiro}). Consequently, by \cite[Thrm.~3.6]{BonnansShapiro} there exists 
$\lambda_* \in T_K(G(x_*))^*$ such that $0 \in \argmin_{h \in T_A(x_*)} L_0(h, \lambda_*)$. 

Observe that $K + G(x_*) \subseteq K$, since $K$ is a convex cone and $G(x_*) \in K$. Consequently, one has
$K \subseteq K - G(x_*) \subseteq T_K(G(x_*))$. Hence bearing in mind the fact that $\lambda_* \in T_K(G(x_*))^*$ one
gets that $\lambda_* \in K^*$, which, in particular, implies that $\langle \lambda_*, G(x_*) \rangle \le 0$. On the
other hand, since $G(x_*) \in K$ and $K$ is a cone, one has $- G(x_*) \in T_K(G(x_*))$ (choose any sequence 
$\{ \alpha_n \} \subset (0, 1)$ converging to zero and put $h_n = - G(x_*)$ for all $n \in \mathbb{N}$), which yields 
$\langle \lambda_*, - G(x_*) \rangle \le 0$, i.e. $\langle \lambda_*, G(x_*) \rangle = 0$. Thus, one has
$\lambda_* \in K^*$, $\langle \lambda_*, G(x_*) \rangle = 0$, and
$$
  [L(\cdot, \lambda_*)]'(x_*, h) = L_0(h, \lambda_*) \ge 0 \quad \forall h \in T_A(x_*),
$$
i.e. $\lambda_*$ is a Lagrange multiplier of the problem $(\mathcal{P})$ at $x_*$. 

Let us show that the set of Lagrange multipliers of the problem $(\mathcal{P})$ at $x_*$, in actuality, coincides with
the set of Lagrange multipliers of the linearised problem \eqref{probl:LinearisedProblem}. Then taking into account the
fact that the set of Lagrange multipliers of the convex problem \eqref{probl:LinearisedProblem} is a convex, bounded,
and weak${}^*$ compact subset of $Y^*$ by \cite[Thrm.~3.6]{BonnansShapiro} we arrive at the required result.

Let $\lambda_*$ be a Lagrange multiplier of the problem $(\mathcal{P})$ at $x_*$. Since 
$L_0(h, \lambda) = [L(\cdot, \lambda)]'(x_*, h)$ for all $h \in \mathbb{R}^d$, by definition it is sufficient to prove
that $\lambda_* \in T_K(G(x_*))^*$. To this end, fix any $v \in T_K(G(x_*))$. By the definition of contingent cone
there exist sequences $\{ \alpha_n \} \subset (0, + \infty)$ and $\{ v_n \} \subset Y$ such that $\alpha_n \to 0$ and
$v_n \to v$ as $n \to \infty$, and $G(x_*) + \alpha_n v_n \in K$ for all $n \in \mathbb{N}$. Since $\lambda_*$ is a
Lagrange multiplier of the problem $(\mathcal{P})$ at $x_*$, one has $\langle \lambda_*, G(x_*) \rangle = 0$ and
$\lambda_* \in K^*$, which implies that 
$0 \ge \langle \lambda_*, G(x_*) + \alpha _n v_n \rangle = \alpha_n \langle \lambda_*, v_n \rangle$ for all 
$n \in \mathbb{N}$. Therefore $\langle \lambda_*, v \rangle \le 0$ for any $v \in T_K(G(x_*))$, i.e. 
$\lambda_* \in T_K(G(x_*))^*$, and the proof is complete.
\end{proof}

Let us now turn to sufficient optimality conditions. Typically, sufficient optimality conditions ensure not only that a
given point is a locally optimal solution of an optimisation problem under consideration, but also that a certain
(usually, second order) growth condition holds at this point. Therefore it is natural to study sufficient optimality
conditions simultaneously with growth conditions.

Recall that \textit{the first order growth condition} (for the problem $(\mathcal{P})$) is said to hold true at a
feasible point $x_*$ of the problem $(\mathcal{P})$, if there exist $\rho > 0$ and a neighbourhood $\mathcal{O}(x_*)$ of
$x_*$ such that $F(x) \ge F(x_*) + \rho | x - x_* |$ for any $x \in \mathcal{O}(x_*) \cap \Omega$, where, as above,
$\Omega$ is the feasible region of $(\mathcal{P})$ and $| \cdot |$ is the Euclidean norm.

By Theorem~\ref{thrm:NessOptCond} the condition
$$
  \max_{\omega \in W(x_*)} \langle \nabla_x f(x_*, \omega), h \rangle \ge 0
  \quad \forall h \in T_A(x_*) \colon D G(x_*) h \in T_K\big( G(x_*) \big)
$$
is a first order necessary optimality condition for the problem $(\mathcal{P})$. Keeping this condition in mind, let us
obtain the natural ``no gap'' sufficient optimality condition that is, in fact, equivalent to the validity of the first
order growth condition.

\begin{theorem} \label{thrm:SuffOptCond}
Let $x_*$ be a feasible point of the problem $(\mathcal{P})$. If
\begin{equation} \label{eq:SuffOptCond}
  \max_{\omega \in W(x_*)} \langle \nabla_x f(x_*, \omega), h \rangle > 0
  \quad \forall h \in T_A(x_*) \setminus \{ 0 \} \colon D G(x_*) h \in T_K\big( G(x_*) \big),
\end{equation}
i.e. if $h = 0$ is a unique globally optimal solution of the linearised problem \eqref{probl:LinearisedProblem}, then
the first order growth condition holds at $x_*$. Conversely, if the first order growth condition and RCQ hold at $x_*$,
then inequality \eqref{eq:SuffOptCond} is valid.
\end{theorem}

\begin{proof}
Let condition \eqref{eq:SuffOptCond} hold true. Arguing by reductio ad absurdum, suppose that the first order growth
condition does not hold true at $x_*$. Then for any $n \in \mathbb{N}$ one can find $x_n \in \Omega$ such that
$F(x_n) < F(x_*) + |x_n - x_*| / n$ and $x_n \to x_*$ as $n \to \infty$. 

Denote $h_n = (x_n - x_*) / |x_n - x_*|$. Without loss of generality one can suppose that the sequence $\{ h_n \}$
converges to a vector $h$ such that $|h| = 1$. From the fact that $x_n \in \Omega = \{ x \in A \mid G(x) \in K \}$ it
follows that $h \in T_A(x_*)$ and $G(x_n) = G(x_*) + |x_n - x_*| D G(x_*) h_n + o(|x_n - x_*|) \in K$ for any 
$n \in \mathbb{N}$, which obviously implies that $D G(x_*) h \in T_K(G(x_*))$. Furthermore, taking into account
\eqref{eq:DirectDerivOfMaxFunc} and the definition of $x_n$ one obtains that
$$
  \max_{\omega \in W(x_*)} \langle \nabla_x f(x_*, \omega), h \rangle = F'(x_*, h)
  = \lim_{n \to \infty} \frac{F(x_n) - F(x_*)}{|x_n - x_*|} \le 0,
$$
which contradicts optimality condition \eqref{eq:SuffOptCond}. Thus, the first order growth condition holds at $x_*$.

Suppose now that RCQ and the first order growth condition hold at $x_*$. Then there exist a neighbourhood
$\mathcal{O}(x_*)$ of $x_*$ and $\rho > 0$ such that $F(x) \ge F(x_*) + \rho |x - x_*|$ for any 
$x \in \mathcal{O}(x_*) \cap \Omega$. 

Fix an arbitrary $h \in T_A(x_*) \setminus \{ 0 \}$ such that $D G(x_*) h \in T_K( G(x_*) )$. 
By Lemma~\ref{lem:ContingConeToFeasibleSet} one has $h \in T_{\Omega}(x_*)$. Hence by definition there exist sequences
$\{ \alpha_n \} \subset (0, + \infty)$ and $\{ h_n \} \subset \mathbb{R}^d$ such that $\alpha_n \to 0$ and $h_n \to h$
as $n \to \infty$, and $x_* + \alpha_n h_n \in \Omega$ for all $n \in \mathbb{N}$. Clearly, 
$x_* + \alpha_n h_n \in \mathcal{O}(x_*)$ for any sufficiently large $n$. Therefore
$$
  F'(x_*, h) = \lim_{n \to \infty} \frac{F(x_* + \alpha_n h_n) - F(x_*)}{\alpha_n} \ge 
  \lim_{n \to \infty} \frac{\rho |\alpha_n h_n|}{\alpha_n} = \rho |h| > 0,
$$
i.e. \eqref{eq:SuffOptCond} holds true.
\end{proof}

\begin{remark}
From the proof of the theorem above it follows that if RCQ and the first order growth condition with constant 
$\rho > 0$ hold true at a feasible point $x_*$ of the problem $(\mathcal{P})$, then the first order growth condition
with the same constant holds true at the origin for the linearised problem \eqref{probl:LinearisedProblem}, which due
to the positive homogeneity of the problem implies that
\begin{equation} \label{eq:LinProbl_FirstOrderGrowth}
  \max_{\omega \in W(x_*)} \langle \nabla_x f(x_*, \omega), h \rangle \ge \rho |h|
  \quad \forall h \in T_A(x_*) \colon D G(x_*) h \in T_K\big( G(x_*) \big).
\end{equation}
Conversely, if this condition holds true, then arguing in almost the same way as in the proof of the first part of
Theorem~\ref{thrm:SuffOptCond} one can check that for any $\rho' \in (0, \rho)$ the first order growth condition with
constrant $\rho'$ holds true at $x_*$. Thus, there is a direct connection between the first order growth conditions for
the problem $(\mathcal{P})$ and the linearised problem \eqref{probl:LinearisedProblem}. Moreover, note that if
\eqref{eq:SuffOptCond} holds true, then there exists $\rho > 0$ such that \eqref{eq:LinProbl_FirstOrderGrowth} is
satisfied, and the least upper bound of all such $\rho$ is equal to 
$\rho_* = \min_h \max_{\omega \in W(x_*)} \langle \nabla_x f(x_*, \omega), h \rangle$, where the minimum is taken over
all those $h \in T_A(x_*)$ for which $D G(x_*) h \in T_K(G(x_*))$ and $|h| = 1$ (the set of all such $h$ is obviously
compact, which implies that the minimum in the definition of $\rho_*$ is attained and $\rho_* > 0$). \qed
\end{remark}

\begin{remark} \label{rmrk:SuffOptCond_Lagrangian}
Note that optimality condition \eqref{eq:SuffOptCond} is satisifed, provided there exists a Lagrange multiplier
$\lambda_*$ of $(\mathcal{P})$ at $x_*$ such that $[L(\cdot, \lambda_*)]'(x_*, h) > 0$ for all 
$h \in T_A(x_*) \setminus \{ 0 \}$. Indeed, fix any $h \in T_A(x_*) \setminus \{ 0 \}$ such that
$D G(x_*) h \in T_K(G(x_*))$. By the definition of Lagrange multiplier one has
$\lambda_* \in K^*$ and $\langle \lambda_*, G(x_*) \rangle = 0$, which implies that
$\langle \lambda_*, y - G(x_*) \rangle \le 0$ for all $y \in K$. Since $K$ is a closed convex set, one has
$T_K(G(x_*)) = \cl[\cup_{t \ge 0} t(K - G(x_*))]$ (see, e.g. \cite[Prp.~2.55]{BonnansShapiro}). Therefore
for any $y \in T_K(G(x_*))$ one has $\langle \lambda_*, y \rangle \le 0$. Consequently, one has
$\langle \lambda_*, D G(x_*) h \rangle \le 0$ and 
\begin{align*}
  \max_{\omega \in W(x_*)} \langle \nabla_x f(x_*, \omega), h \rangle
  &\ge \max_{\omega \in W(x_*)} \langle \nabla_x f(x_*, \omega), h \rangle + \langle \lambda_*, D G(x_*) h \rangle \\
  &= [L(\cdot, \lambda_*)]'(x_*, h) > 0
\end{align*}
for any $h \in T_A(x_*) \setminus \{ 0 \}$ such that $D G(x_*) h \in T_K(G(x_*))$, i.e. optimality condition
\eqref{eq:SuffOptCond} holds true. However, note that the converse statement does not hold true in the general case.
Indeed, for any smooth problem with $A = \mathbb{R}^d$ one has $\nabla_x L(x_*, \lambda_*) = 0$ by the definition of
Lagrange multiplier, and the inequality $[L(\cdot, \lambda_*)]'(x_*, h) > 0$ for all $h \ne 0$ cannot be satisfied, but
sufficient optimality condition \eqref{eq:SuffOptCond} might hold true. Consider, for example, the problem
$$
  \min\: f(x) = -x \quad \text{subject to} \quad g(x) = x \le 0.
$$
The point $x_* = 0$ is a globally optimal solution of this problem. Moreover, 
$\langle \nabla f(x_*), h \rangle = - h > 0$ for any $h \ne 0$ such that 
$\langle \nabla g(x_*), h \rangle = h \le 0$, i.e. optimality condition \eqref{eq:SuffOptCond} holds true. \qed
\end{remark}

Let us also note that in the convex case a necessary optimality condition becomes a sufficient condition for a
global minimum. Recall that the mapping $G$ is called \textit{convex} with respect to the cone $-K$ (or
$(-K)$-\textit{convex}), if $G(\alpha x_1 + (1 - \alpha) x_2) - \alpha G(x_1) - (1 - \alpha) G(x_2) \in K$
for any $x_1, x_2 \in \mathbb{R}^d$ and $\alpha \in [0, 1]$ (see \cite[Def.~2.103]{BonnansShapiro}).

\begin{theorem} \label{thrm:OptCond_ConvexCase}
Let for any $\omega \in W$ the function $f(\cdot, \omega)$ be convex, the mapping $G$ be $(-K)$-convex, and let $x_*$ be
a feasible point of the problem $(\mathcal{P})$. Then:
\begin{enumerate}
\item{$\lambda_* \in K^*$ is a Lagrange multiplier of $(\mathcal{P})$ at $x_*$ iff $(x_*, \lambda_*)$ is a global saddle
point of the Lagrangian $L(x, \lambda) = F(x) + \langle \lambda, G(x) \rangle$, that is,
\begin{equation} \label{eq:GlobalSaddlePoint}
  L(x, \lambda_*) \ge F(x_*) \ge L(x_*, \lambda) \quad \forall x \in A, \: \lambda \in K^*;
\end{equation}
\vspace{-7mm}\label{stat:GlobalSaddlePoint}}

\item{if a Lagrange multiplier of the problem $(\mathcal{P})$ at $x_*$ exists, then $x_*$ is a globally optimal solution
of $(\mathcal{P})$; conversely, if $x_*$ is a globally optimal solution of the problem $(\mathcal{P})$ and Slater's
condition $0 \in \interior\{ G(A) - K \}$ holds true, then there exists a Lagrange multiplier of $(\mathcal{P})$ 
at $x_*$.
\label{stat:LagrangeMult_ConvexCase}}
\end{enumerate}
\end{theorem}

\begin{proof}
\textbf{Part~\ref{stat:GlobalSaddlePoint}.} Let $\lambda_*$ be a Lagrange multiplier of $(\mathcal{P})$ at $x_*$. Note
that $\langle \lambda, G(x_*) \rangle \le 0$ for any $\lambda \in K^*$, since $x_*$ is a feasible point 
(i.e. $G(x_*) \in K$), which implies that $L(x_*, \lambda) \le F(x_*)$ for all $\lambda \in K^*$. Thus, the second
inequality in \eqref{eq:GlobalSaddlePoint} holds true.

By the definition of Lagrange multiplier $\langle \lambda_*, G(x_*) \rangle = 0$, which yields 
$L(x_*, \lambda_*) = F(x_*)$. Thus, the first inequality in \eqref{eq:GlobalSaddlePoint} is satisfied iff $x_*$ is a
point of global minimum of the function $L(\cdot, \lambda_*)$ on the set $A$. Arguing by reductio ad absurdum, suppose
that this statement is false. Then there exists $x_0 \in A$ such that $L(x_0, \lambda_*) < L(x_*, \lambda_*)$.

Under our assumptions the function $F$ is convex as the maximum of a family of convex functions. Moreover, for any
$\lambda \in K^*$ the function $\langle \lambda, G(\cdot) \rangle$ is convex as well, since for any 
$x_1, x_2 \in \mathbb{R}^d$ and $\alpha \in [0, 1]$
$\langle \lambda, G(\alpha x_1 + (1 - \alpha) x_2) - \alpha G(x_1) - (1 - \alpha) G(x_2) \rangle \le 0$. Thus, the
Lagrangian $L(\cdot, \lambda_*)$ is convex. Therefore, for any $\alpha \in [0, 1]$ one has
$$
  L(\alpha x_0 + (1 - \alpha) x_*, \lambda_*) - L(x_*, \lambda_*) 
  \le \alpha \big( L(x_0, \lambda_*) - L(x_*, \lambda_*) \big).
$$
Dividing this inequality by $\alpha$ and passing to the limit as $\alpha \to +0$ one obtains that 
$[L(\cdot, \lambda_*)]'(x_*, x_0 - x_*) \le L(x_0, \lambda_*) - L(x_*, \lambda_*) < 0$, which contradicts the fact that
$\lambda_*$ is a Lagrange multiplier, since $x_0 - x_* \in T_A(x_*)$ by the fact that $A$ is a convex set. Thus, 
the first inequality in \eqref{eq:GlobalSaddlePoint} holds true and $(x_*, \lambda_*)$ is a global saddle point of the
Lagrangian.

Let us prove the converse statement. Suppose that $(x_*, \lambda_*)$ is a global saddle point of $L(x, \lambda)$. Then
$L(x, \lambda_*) \ge F(x_*) \ge L(x_*, \lambda_*)$ for any $x \in A$ (see~\eqref{eq:GlobalSaddlePoint}), which implies
that $x_*$ is a point of global minimum of the function $L(\cdot, \lambda_*)$ and 
$\langle \lambda_*, G(x_*) \rangle = 0$, since by the definition of global saddle point
$F(x_*) = L(x_*, \lambda_*) = F(x_*) + \langle \lambda_*, G(x_*) \rangle$. 

Recall that the function $F$ is Hadamard directionally differentiable by \cite[Thrm.~4.4.3]{IoffeTihomirov}.
Consequently, the function $L(\cdot, \lambda_*)$ is Hadamard directionally differentiable as well. Therefore, applying
the necessary optimality condition in terms of directional derivatives (see, e.g. \cite[Lemma~V.1.2]{DemyanovRubinov})
one obtains that $[L(\cdot, \lambda_*)]'(x_*, h) \ge 0$ for all $h \in T_A(x_*)$, i.e. $\lambda_*$ is a Lagrange
multiplier of the problem $(\mathcal{P})$ at $x_*$.

\textbf{Part~\ref{stat:LagrangeMult_ConvexCase}.} Let $\lambda_*$ be a Lagrange multiplier of $(\mathcal{P})$ at $x_*$.
Then by the first part of the theorem $L(x, \lambda_*) \ge F(x_*)$ for all $x \in A$. By the definition of Lagrange
multiplier $\lambda_* \in K^*$, which implies that $\langle \lambda_*, G(x) \rangle \le 0$ for any $x$ such that 
$G(x) \in K$. Thus, for any feasible point of the problem $(\mathcal{P})$ one has 
$F(x) \ge L(x, \lambda_*) \ge F(x_*)$, i.e. $x_*$ is a globally optimal solution of $(\mathcal{P})$.

It remains to note that the converse statement follows directly from Theorem~\ref{thrm:NessOptCond} and the fact that 
by \cite[Prp.~2.104]{BonnansShapiro} Slater's condition $0 \in \interior\{ G(A) - K \}$ is equivalent to RCQ, provided
$G$ is $(-K)$-convex.
\end{proof}

\subsection{Subdifferentials and exact penalty functions}
\label{subsect:Subdifferentials_ExactPenaltyFunc}

Note that both necessary and sufficient optimality conditions stated in Theorems~\ref{thrm:NessOptCond} and
\ref{thrm:SuffOptCond} are very difficult to verify directly. Let us show how one can reformulate them in a more
convenient way.

Denote by $N_A(x) = \{ z \in \mathbb{R}^d \mid \langle z, v \rangle \le 0 \: \forall v \in T_A(x) \}$ \textit{the normal
cone} to the convex set $A$ at a point $x \in A$. Note that $N_A(x)$ is the polar cone of $T_A(x)$ and
$N_A(x) = \{ z  \in \mathbb{R}^d \mid \langle z, v - x \rangle \: \forall v \in A \}$, since
$T_A(x) = \cl[\cup_{t \ge 0} t(A - x)]$ by virtue of the fact that the set $A$ is convex 
(see, e.g. \cite[Prp.~2.55]{BonnansShapiro}). For any subspace $Y_0 \subset Y$ denote by 
$Y_0^{\perp} = \{ y^* \in Y^* \mid \langle y^*, y \rangle = 0 \: \forall y \in Y_0 \}$ \textit{the annihilator} of
$Y_0$. For the sake of correctness, for any linear operator $T \colon \mathbb{R}^d \to Y$ denote by
$[T]^*$ the composition of the natural isomorphism $i$ between $(\mathbb{R}^d)^*$ and $\mathbb{R}^d$, and the adjoint
operator $T^* \colon Y^* \to (\mathbb{R}^d)^*$, i.e. $[T]^* = i \circ T^*$.

Introduce the cone 
$$
  \mathcal{N}(x) = [D G(x)]^* (K^* \cap \linhull(G(x))^{\perp}) 
  = \{ i(\lambda \circ D G(x)) \mid \lambda \in K^*, \: \langle \lambda, G(x) \rangle = 0 \}.
$$
Let us verify that the convex cone $\mathcal{N}(x) \subset \mathbb{R}^d$ is, in actuality, the normal cone to the set 
$\Xi = \{ z \in \mathbb{R}^d \mid G(z) \in K \}$ at the point $x$.

\begin{lemma} \label{lem:NormalCone_ConeConstr}
Let $x \in \mathbb{R}^d$ be such that $G(x) \in K$. Then
\begin{equation} \label{eq:NormalCone_ConeConstr}
  \mathcal{N}(x) \subseteq \Big( \big\{ h \in \mathbb{R}^d \bigm| D G(x) h \in T_K(G(x)) \big\} \Big)^*.
\end{equation}
Furthermore, if the weakened Robinson constraint qualification $0 \in \interior\{ G(x) + D G(x)(\mathbb{R}^n) - K \}$
is satisfied at $x$, then the opposite inclusion holds true and $\mathcal{N}(x) = (T_{\Xi}(x))^* = N_{\Xi}(x)$.
\end{lemma}

\begin{proof}
Choose any $v \in \mathcal{N}(x)$. Then $v = [D G(x)]^* \lambda$ for some $\lambda \in K^*$ such that 
$\langle \lambda, G(x) \rangle = 0$. By definition $\langle \lambda, y - G(x)  \rangle \le 0$ for any $y \in K$. Hence
with the use of the well-known equality $T_K(G(x)) = \cl[\cup_{t \ge 0} t(K - G(x))]$ (see, e.g.
\cite[Prp.~2.55]{BonnansShapiro}) one obtains that $\langle \lambda, y \rangle \le 0$ for any $y \in T_K(G(x))$.
Consequently, for any $h \in \mathbb{R}^d$ satisfying the condition $D G(x) h \in T_K(G(x))$ one has 
$\langle v, h \rangle = \langle \lambda, D G(x) h \rangle \le 0$, that is, $v$ belongs to the right-hand side of
\eqref{eq:NormalCone_ConeConstr}.

Suppose now that the weakened RCQ holds at $x_*$, and let $v$ belong to the right-hand side of
\eqref{eq:NormalCone_ConeConstr}, that is, $\langle v, h \rangle \le 0$ for any $h \in \mathbb{R}^d$ such that 
$D G(x) h \in T_K(G(x))$. In other words, $h = 0$ is a point of global minimum of the conic linear problem:
\begin{equation} \label{probl:ConicLinearProblem}
  \min \: \langle -v, h \rangle \quad \text{subject to} \quad D G(x) h \in T_K(G(x)).
\end{equation}
Note that the contingent cone $T_K(G(x))$ is convex, since the cone $K$ is convex. Furthermore, from the weakened RCQ
and the inclusion $(K - G(x)) \subset T_K(G(x))$ it follows that the regularity condition
$0 \in \interior\{ D G(x)(\mathbb{R}^d) - T_K(G(x)) \}$ holds true for problem \eqref{probl:ConicLinearProblem}.
Therefore by \cite[Thrm.~3.6]{BonnansShapiro} there exists a Lagrange multiplier $\lambda$ for problem
\eqref{probl:ConicLinearProblem}, i.e. $- v + [D G(x)]^* \lambda = 0$ and
$\lambda \in T_K(G(x))^*$. Bearing in mind the equality $T_K(G(x)) = \cl[\cup_{t \ge 0} t(K - G(x))]$ one obtains
that $\langle \lambda, y - G(x) \rangle \le 0$ for any $y \in K$. Putting $y = 2 G(x)$ and $y = 0$ one gets that
$\langle \lambda, G(x) \rangle = 0$, while putting $y = z + G(x) \in K$ for any $z \in K$ (recall that $K$ is a convex
cone) one gets that $\langle \lambda, z \rangle \le 0$ for any $z \in K$ or, equivalently, $\lambda \in K^*$. Thus, 
$v = [D G(x)]^* \lambda$ for some $\lambda \in K^*$ such that $\langle \lambda, G(x) \rangle = 0$, i.e.
$v \in \mathcal{N}(x)$ and the inclusion opposite to \eqref{eq:NormalCone_ConeConstr} is valid.

It remains to note that $T_{\Xi}(x) = \{ h \in \mathbb{R}^d \mid D G(x) h \in T_K(G(x)) \}$, since the weakened RCQ is
satisfied at $x_*$ (see, e.g. \cite[Corollary~2.91]{BonnansShapiro}). 
Thus, $\mathcal{N}(x) = T_{\Xi}(x)^* = N_{\Xi}(x)$ and the proof is complete.
\end{proof}

For any $x \in \mathbb{R}^d$ denote by $\partial F(x) = \co\{ \nabla_x f(x_*, \omega) \mid \omega \in W(x_*) \}$ 
\textit{the Hadamard subdifferential} of the function $F(x) = \max_{\omega \in W} f(x, \omega)$. Introduce a set-valued
mapping $\mathcal{D} \colon \Omega \rightrightarrows \mathbb{R}^d$ as follows:
$$
  \mathcal{D}(x) = \partial F(x) + \mathcal{N}(x) + N_A(x).
$$
The multifunction $\mathcal{D}$ is obviously convex-valued. Our first aim is to show that optimality conditions for the
problem $(\mathcal{P})$ can be rewritten in the form of the inclusion $0 \in \mathcal{D}(x)$.

\begin{theorem} \label{thrm:EquivOptCond_Subdiff}
Let $x_*$ be a feasible point of the problem $(\mathcal{P})$. Then: 
\begin{enumerate}
\item{a Lagrange multiplier of $(\mathcal{P})$ at $x_*$ exists iff $0 \in \mathcal{D}(x_*)$;
\label{stat:NessOpt_Subdiff}}

\item{sufficient optimality condition \eqref{eq:SuffOptCond} holds true at $x_*$ iff $0 \in \interior \mathcal{D}(x_*)$.
\label{stat:SuffOpt_Subdiff}}
\end{enumerate}
\end{theorem}

\begin{proof}
\textbf{Part~\ref{stat:NessOpt_Subdiff}.}~Let $\lambda_*$ be a Lagrange multiplier of $(\mathcal{P})$ at $x_*$ and
$Q(x_*) = \partial F(x_*) + [D G(x_*)]^* \lambda_*$. By the definition of Lagrange multiplier one has
\begin{equation} \label{eq:Lm_Ness_Cond}
  [L(\cdot, \lambda_*)]'(x_*, h) = \max_{v \in Q(x_*)} \langle v, h \rangle \ge 0 \quad \forall h \in T_A(x_*).
\end{equation}
Let us check that this inequality implies that $0 \in Q(x_*) + N_A(x_*)$. Indeed, arguing by reductio ad absurdum,
suppose that $Q(x_*) \cap (- N_A(x_*)) = \emptyset$. Observe that $Q(x_*)$ is a compact convex set, while $N_A(x_*)$ is
a closed convex cone. Consequently, applying the separation theorem one obtains that there exists $h \ne 0$ such that
\begin{equation} \label{eq:SeparationThrm}
  \langle v, h \rangle < \langle u, h \rangle \quad \forall v \in Q(x_*) \quad \forall u \in \big(- N_A(x_*) \big).
\end{equation}
Since $N_A(x_*)$ is a cone, the inequality above implies that $\langle u, h \rangle \le 0$ for all 
$u \in N_A(x_*)$, i.e. $h$ belongs to the polar cone of $N_A(x_*)$. Recall that $N_A(x_*)$ is a polar cone of
$T_A(x_*)$. Therefore, $h \in T_A(x_*)^{**} = T_A(x_*)$ (see, e.g. \cite[Prp.~2.40]{BonnansShapiro}).

Taking into account inequality \eqref{eq:SeparationThrm} and the facts that $0 \in N_A(x_*)$ and $Q(x_*)$ is a compact
set one obtains that $\max_{v \in Q(x_*)} \langle v, h \rangle < 0$, which contradicts \eqref{eq:Lm_Ness_Cond}. Thus, 
$0 \in Q(x_*) + N_A(x_*)$, which implies that $0 \in \mathcal{D}(x_*)$ due to the fact that by the definition of
Lagrange multiplier one has $\lambda_* \in K^*$ and $\langle \lambda_*, G(x_*) \rangle = 0$.

Let us prove the converse statement. Suppose that $0 \in \mathcal{D}(x_*)$. Then there exist $v_* \in \partial F(x_*)$
and $\lambda_* \in K^*$ such that $v_* + [D G(x_*)]^* \lambda_* \in - N_A(x_*)$ and 
$\langle \lambda_*, G(x_*) \rangle = 0$. By the definition of $N_A(x_*)$ one has
$$
  \max_{v \in \partial F(x_*)} \langle v, h \rangle + \langle \lambda_*, D G(x_*) h \rangle
  \ge \langle v_*, h \rangle + \langle \lambda_*, D G(x_*) h \rangle \ge 0 \quad \forall h \in T_A(x_*).
$$
In other words, $[L(\cdot, \lambda_*)]'(x_*, h) \ge 0$ for all $h \in T_A(x_*)$. Thus, $\lambda_*$ is a Lagrange
multiplier of $(\mathcal{P})$ at $x_*$.

\textbf{Part~\ref{stat:SuffOpt_Subdiff}.}~Let sufficient optimality condition \eqref{eq:SuffOptCond} be satisfied. Let
us show at first that zero belongs to the relative interior $\relint \mathcal{D}(x_*)$ of $\mathcal{D}(x_*)$. Indeed,
arguing by reductio ad absurdum, suppose that $0 \notin \relint \mathcal{D}(x_*)$. Then by the separation theorem
(see, e.g. \cite[Thrm.~2.17]{BonnansShapiro}) there exists $h \ne 0$ such that $\langle v, h \rangle \le 0$ for all 
$v \in \mathcal{D}(x_*)$. Hence taking into account the fact that both $\mathcal{N}(x_*)$ and $N_A(x_*)$ are convex
cones one obtains that
$$
  \max_{v \in \partial F(x_*)} \langle v, h \rangle \le 0, \quad
  \langle v, h \rangle \le 0 \quad \forall v \in \mathcal{N}(x_*), \quad
  \langle v, h \rangle \le 0 \quad \forall v \in N_A(x_*).
$$
Therefore $h \in N_A(x_*)^* = T_A(x_*)^{**} = T_A(x_*)$ and
\begin{equation} \label{eq:FrDerInTangCone}
  \langle \lambda, D G(x_*) h \rangle \le 0 \quad 
  \forall \lambda \in K^* \colon \langle \lambda, G(x_*) \rangle = 0.
\end{equation}
Let us verify that this inequality implies that $D G(x_*) h \in T_K(G(x_*))$. Then one obtains that we found 
$h \in T_A(x_*) \setminus \{ 0 \}$ such that $D G(x_*) h \in T_K\big( G(x_*) \big)$ and 
$\max_{\omega \in W(x_*)} \langle \nabla_x f(x_*, \omega), h \rangle \le 0$, which contradicts \eqref{eq:SuffOptCond}. 

Arguing by reductio ad absurdum, suppose that $D G(x_*) h \notin T_K(G(x_*))$. The cone $T_K(G(x_*))$ is closed and
convex, since $K$ is a convex cone. Therefore, by the separation theorem there exists 
$\lambda \in Y^* \setminus \{ 0 \}$ such that 
\begin{equation} \label{eq:InclCondSepTh}
  \langle \lambda, D G(x_*) h \rangle > 0, \quad \langle \lambda, y \rangle \le 0 \quad \forall y \in T_K(G(x_*)).
\end{equation}
Since $K$ is a cone and $G(x_*) \in K$, one has $G(x_*) + \alpha G(x_*) \in K$ for all $\alpha \in [-1, 1]$, which
implies that $G(x_*) \in T_K(G(x_*))$, $- G(x_*) \in T_K(G(x_*))$, and $\langle \lambda, G(x_*) \rangle = 0$.
Furthermore, as was noted above, $K \subseteq K - G(x_*) \subseteq T_K( G(x_*) )$ due to the facts that $G(x_*) \in K$
and $K$ is a convex cone. Hence with the use of \eqref{eq:InclCondSepTh} one obtains that $\lambda \in K^*$, 
$\langle \lambda, G(x_*) \rangle = 0$, and $\langle \lambda, D G(x_*) h \rangle > 0$, which contradicts
\eqref{eq:FrDerInTangCone}. Thus, $D G(x_*) h \in T_K(G(x_*))$ and $0 \in \relint \mathcal{D}(x_*)$.

Let us now show that $\interior \mathcal{D}(x_*) \ne \emptyset$. Then $0 \in \interior \mathcal{D}(x_*)$ and the proof
is complete. Arguing by reductio ad absurdum, suppose that $\interior \mathcal{D}(x_*) = \emptyset$. From the facts that
$0 \in \relint \mathcal{D}(x_*)$ and $\interior \mathcal{D}(x_*) = \emptyset$ it follows 
that $\linhull \mathcal{D}(x_*) \ne \mathbb{R}^d$. Therefore, there exists $h \ne 0$ such 
that $\langle v, h \rangle = 0$ for all $v \in \linhull \mathcal{D}(x_*)$. Consequently, with the use of the fact that
both $\mathcal{N}(x_*)$ and $N_A(x_*)$ are convex cones one obtains that
$$
  \max_{v \in \partial F(x_*)} \langle v, h \rangle = 0, \quad
  \langle v, h \rangle = 0 \quad \forall v \in \mathcal{N}(x_*), \quad
  \langle v, h \rangle = 0 \quad \forall v \in N_A(x_*).
$$
Hence $h \in N_A(x_*)^* = T_A(x_*)^{**} = T_A(x_*)$ and inequality \eqref{eq:FrDerInTangCone} holds true. As was
shown above, this inequality implies that $D G(x_*) h \in T_K(G(x_*))$. Thus, we found 
$h \in T_A(x_*) \setminus \{ 0 \}$ such that $\max_{\omega \in W(x_*)} \langle \nabla_x f(x_*, \omega), h \rangle = 0$
and $D G(x_*) h \in T_K(G(x_*))$, which contradicts \eqref{eq:SuffOptCond}. 
Therefore $0 \in \interior \mathcal{D}(x_*)$.

Let us prove the converse statement. Suppose that $0 \in \interior \mathcal{D}(x_*)$. Then there exists $\rho > 0$ such
that $\max_{v \in \mathcal{D}(x_*)} \langle v, h \rangle \ge \rho |h|$ for all $h \in \mathbb{R}^d$. 

Fix any $h \in T_A(x_*)$ such that $D G(x_*) h \in T_K(G(x_*))$. By definition any $v \in \mathcal{D}(x_*)$ has
the form $v = v_1 + v_2 + v_3$, where $v_1 \in \partial F(x_*)$, $v_2 = [D G(x_*)]^* \lambda_2$ for some 
$\lambda_2 \in K^* \cap \linhull(G(x_*))^{\perp}$, and $v_3 \in N_A(x_*)$. 

Firstly, note that $\langle v_3, h \rangle \le 0$, since $h \in T_A(x_*)$. Secondly, recall that $K$ is a convex cone,
which implies that $T_K(G(x_*)) = \cl[\cup_{t \ge 0} t(K - G(x_*))]$ (see, e.g. \cite[Prp.~2.55]{BonnansShapiro}).
Hence taking into account the facts that $\lambda_2 \in K^*$ and $\langle \lambda_2, G(x_*) \rangle = 0$ one gets that
$\langle \lambda_2, y \rangle \le 0$ for all $y \in T_K(G(x_*))$. Consequently, 
$\langle v_2, h \rangle = \langle \lambda_2, D G(x_*) h \rangle \le 0$, since $D G(x_*) h \in T_K(G(x_*))$.
Thus, for any $v \in \mathcal{D}(x_*)$ one has $\langle v, h \rangle \le \langle v_1, h \rangle$ for the corresponding
vector $v_1 \in \partial F(x_*)$, which implies that
$$
  \max_{v \in \partial F(x_*)} \langle v, h \rangle \ge \max_{v \in \mathcal{D}(x_*)} \langle v, h \rangle 
  \ge \rho |h| \quad \forall h \in T_A(x_*) \colon D G(x_*) h \in T_K(G(x_*)),
$$
i.e. sufficient optimality condition \eqref{eq:SuffOptCond} holds true.
\end{proof}

\begin{remark} \label{rmrk:LagrangeMultViaSubdiff}
From the proof of the first part of the theorem above it follows that $\lambda_*$ is a Lagrange multiplier of the
problem $(\mathcal{P})$ at $x_*$ iff $(\partial F(x_*) + [D G(x_*)]^* \lambda_*) \cap (- N_A(x_*)) \ne \emptyset$. 
In particular, if $A = \mathbb{R}^d$, then $\lambda_*$ is a Lagrange multiplier at $x_*$ iff 
$0 \in \partial F(x_*) + [D G(x_*)]^* \lambda_* = \partial_x L(x_*, \lambda_*)$, where $\partial_x L(x_*, \lambda_*)$
is the Hadamard subdifferential of the function $L(\cdot, \lambda_*)$ at $x_*$. \qed
\end{remark}

The theorem above contains a reformulation of necessary and sufficient optimality conditions for the problem
$(\mathcal{P})$ in terms of the set $\mathcal{D}(x_*) = \partial F(x_*) + \mathcal{N}(x_*) + N_A(x_*)$. Note that this
convex set need not be closed, since it is the sum of a compact convex set $\partial F(x_*)$ and two closed convex
cones. In the case of necessary conditions, one can rewrite inclusion $0 \in \mathcal{D}(x_*)$ as the condition 
$( \partial F(x_*) + \mathcal{N}(x_*) ) \cap (- N_A(x_*)) \ne \emptyset$ involving only closed sets; however, sufficient
optimality conditions cannot be directly rewritten in this way. 

Our next goal is to show that one can replace the set $\mathcal{D}(x)$ in Theorem~\ref{thrm:EquivOptCond_Subdiff} with a
smaller \textit{closed} convex set and to simultaneously show a close connection between sufficient optimality
conditions for the problem $(\mathcal{P})$ and exact penalty functions. To this end, denote by 
$\Phi_c(x) = F(x) + c \dist(G(x), K)$ a nonsmooth penalty function for the cone constraint of the problem
$(\mathcal{P})$. Here $c \ge 0$ is the penalty parameter and $\dist(y, K) = \inf\{ \| y - z \| \mid z \in K \}$ is the
distance between a point $y \in Y$ and the cone $K$. Note that the function $\Phi_c$ is nondecreasing in $c$.

Before we proceed to an analysis of optimality conditions, let us first compute a subdifferential of the penalty
function $\Phi_c$. Denote $\varphi(x) = \dist(G(x), K)$.

\begin{lemma} \label{lem:ConeConstrPenFunc_Subdiff}
Let $x$ be such that $G(x) \in K$. Then for any $c \ge 0$ the penalty function $\Phi_c$ is Hadamard subdifferentiable at
$x$ and its Hadamard subdifferential has the form $\partial \Phi_c(x) = \partial F(x) + c \partial \varphi(x)$, where
\begin{equation} \label{eq:ConeConstPenTerm_Subdiff}
  \partial \varphi(x) =
  \Big\{ [D G(x)]^* y^* \in \mathbb{R}^d \Bigm| y^* \in Y^*, \: \| y^* \| \le 1, \:
  \langle y^*, y - G(x) \rangle \le 0 \enspace \forall y \in K \Big\}
\end{equation}
i.e. $\Phi_c$ is Hadamard directionally differentiable at $x$, for any $h \in \mathbb{R}^d$ one has
$$
  \Phi'_c(x, h) = \lim_{[\alpha, h'] \to [+0, h]} \frac{\Phi_c(x + \alpha h') - \Phi_c(x)}{\alpha}
  = \max_{v \in \partial \Phi_c(x)} \langle v, h \rangle,
$$
and the set $\partial \Phi_c(x)$ is convex and compact.
\end{lemma}

\begin{proof}
As was noted in the proof of Theorem~\ref{thrm:NessOptCond}, by \cite[Thrm.~4.4.3]{IoffeTihomirov} the function
$F(x)$ is Hadamard subdifferentiable. Since the sum of Hadamard subdifferentiable functions is obviously Hadamard
subdifferentiable and the Hadamard subdifferential of the sum is equal to the sum of Hadamard subdifferentials (see,
e.g. \cite[Thrm.~4.4.1]{IoffeTihomirov}), it is sufficient to prove that the penalty term $\varphi(x)$ is Hadamard
subdifferentiable and the set \eqref{eq:ConeConstPenTerm_Subdiff} is its Hadamard subdifferential.

Denote $d(y) = \dist(y, K)$. The function $d(\cdot)$ is convex due to the fact that $K$ is a convex set. By
\cite[Example~2.130]{BonnansShapiro} its subdifferential (in the sense of convex analysis) at any point $y \in K$ has
the form
$$
  \partial d(y) = \Big\{ y^* \in Y^* \Bigm| \| y^* \| \le 1, \: 
  \langle y^*, z - y \rangle \le 0 \enspace \forall z \in K \Big\}.
$$
In turn, by \cite[Prp.~4.4.1]{IoffeTihomirov} the function $d(\cdot)$ is Hadamard subdifferentiable at $y$ and its
Hadamard subdifferential coincides with its subdifferential in the sense of convex analysis. Finally, by 
\cite[Thrm.~4.4.2]{IoffeTihomirov} the function $\varphi(\cdot) = d(G(\cdot))$ is Hadamard subdifferentiable at $x$ as
well, and its Hadamard subdifferential at this point has the form $\partial \varphi(x) = [D G(x)]^* \partial d(G(x))$,
i.e. \eqref{eq:ConeConstPenTerm_Subdiff} holds true.
\end{proof}

\begin{remark} \label{remark:Subdiff_ConeConstrPenFunc}
From the equality $T_K(G(x_*)) = \cl[\cup_{t \ge 0} t(K - G(x_*))]$ (see, e.g. \cite[Prp.~2.55]{BonnansShapiro}) it
follows that 
$$
  \partial \varphi(x) = \Big\{ [D G(x)]^* y^* \in \mathbb{R}^d \Bigm| y^* \in (T_K(G(x)))^*, \: \| y^* \| \le 1 
  \Big\}.
$$
Moreover, since $\partial \varphi(x)$ is a convex set and $0 \in \partial \varphi(x)$, one has 
$c \partial \varphi(x) \subseteq r \partial \varphi(x)$ for any $r \ge c \ge 0$, which implies that
$\partial \Phi_c(x) \subseteq \partial \Phi_r(x)$ for any $r \ge c \ge 0$. In addition, the inclusion 
$0 \in \partial \varphi(x)$ implies that $\affine(c\partial \varphi(x)) = \linhull \partial \varphi(x)$ for any $c > 0$,
where ``$\affine$'' stands for the affine hull. As is well-known and easy to check, 
$\affine( S_1 + S_2 ) = \affine S_1 + \affine S_2$ for any subsets $S_1$ and $S_2$ of a real vector space, which implies
that
$$
  \affine \partial \Phi_c(x) = \affine \partial F(x) + \linhull \partial \varphi(x) = \affine \partial \Phi_r(x)
  \quad \forall c, r > 0,
$$
that is, the affine hull of the subdifferential $\partial \Phi_c(x)$ does not depend on $c > 0$ and for any 
$r \ge c > 0$ one has $\relint \partial \Phi_c(x) \subseteq \relint \partial \Phi_r(x)$. \qed
\end{remark}

Instead of the problem $(\mathcal{P})$ one can consider the following penalised problem:
\begin{equation} \label{probl:PenalisedProblem}
  \min \: \Phi_c(x) = \max_{w \in W} f(x, \omega) + c \dist( G(x), K )
  \quad \text{subject to} \quad x \in A.
\end{equation}
Recall that the penalty function $\Phi_c$ is called \textit{locally exact} at a locally optimal solution $x_*$ of the
problem $(\mathcal{P})$, if there exists $c_* \ge 0$ such that $x_*$ is a point of local minimum of the penalised
problem \eqref{probl:PenalisedProblem} for any $c \ge c_*$. We say that $\Phi_c$ satisfies \textit{the first order
growth condition} on the set $A$ at a point $x_* \in A$, if there exist a neighbourhood $\mathcal{O}(x_*)$ of $x_*$
and $\rho > 0$ such that $\Phi_c(x) \ge \Phi_c(x_*) + \rho |x - x_*|$ for all $x \in \mathcal{O}(x_*) \cap A$. 

From the fact that $\Phi_c(x) = F(x)$ for any $x$ such that $G(x) \in K$ it follows that if the first order growth
condition holds true for $\Phi_c$ on $A$ at a feasible point $x_*$ of the problem $(\mathcal{P})$, then $x_*$ is a
locally optimal solution of this problem, the first order growth condition for the problem $(\mathcal{P})$ holds at
$x_*$, and $\Phi_c$ is locally exact at $x_*$.

The following theorem describes interrelations between optimality conditions for the problem $(\mathcal{P})$,
optimality conditions for the penalised problem \eqref{probl:PenalisedProblem}, the local exactness of the penalty
function $\Phi_c$, and the first order growth conditions.

\begin{theorem} \label{thrm:EquivOptCond_PenaltyFunc}
Let $x_*$ be a feasible point of the problem $(\mathcal{P})$. Then: 
\begin{enumerate}
\item{a Lagrange multiplier of the problem $(\mathcal{P})$ at $x_*$ exists iff there exists $c \ge 0$ such that
$0 \in \partial \Phi_c(x_*) + N_A(x_*)$;
\label{stat:NessOpt_PenaltyFunc}}

\item{sufficient optimality condition \eqref{eq:SuffOptCond} is satisfied at $x_*$ if and only iff there exists 
$c \ge 0$ such that $0 \in \interior(\partial \Phi_c(x_*) + N_A(x_*))$ iff there exists $c \ge 0$ such that $\Phi_c$
satisfies the first order growth condition on $A$ at $x_*$;
\label{stat:SuffOpt_PenaltyFunc}}

\item{if RCQ holds at $x_*$, then the penalty function $\Phi_c$ is locally exact at $x_*$; furthermore, in this case
the first order growth condition for the problem $(\mathcal{P})$ holds at $x_*$ iff $\Phi_c$ satisfies the first order
growth condition on $A$ at $x_*$.
\label{stat:Exactness_GrowthConditions}}
\end{enumerate}
\end{theorem}

\begin{proof}
\textbf{Part~\ref{stat:NessOpt_PenaltyFunc}.} Let $\lambda_*$ be a Lagrange multiplier of the problem $(\mathcal{P})$ 
at $x_*$. By definition $\lambda_* \in K^*$ and $\langle \lambda_*, G(x_*) \rangle = 0$, which implies that
$\langle \lambda_*, y - G(x_*) \rangle \le 0$ for all $y \in K$ and for any $c \ge \| \lambda_* \|$ one has
$[D G(x_*)]^* \lambda_* \in c \partial \varphi(x_*)$  (see \eqref{eq:ConeConstPenTerm_Subdiff}). Hence by the definition
of Lagrange multiplier and equality \eqref{eq:DirectDerivOfMaxFunc} for any $c \ge \| \lambda_* \|$ and $h \in T_A(x_*)$
one has
$$
  \max_{v \in \partial \Phi_c(x_*)} \langle v, h \rangle 
  \ge \max_{v \in \partial F(x_*) + [D G(x_*)]^* \lambda_*} \langle v, h \rangle 
  = [L(\cdot, \lambda_*)]'(x_*, h) \ge 0.
$$
Then applying the separation theorem one can easily check that this inequality implies that 
$0 \in \partial \Phi_c(x_*) + N_A(x_*)$ for any $c \ge \| \lambda_* \|$.

Let us now prove the converse statement. Suppose that $0 \in \partial \Phi_c(x_*) + N_A(x_*)$ for some $c \ge 0$. Recall
that by Lemma~\ref{lem:ConeConstrPenFunc_Subdiff} one has 
$\partial \Phi_c(x_*) = \partial F(x_*) + c \partial \varphi(x_*)$. Therefore, there exist 
$v_0 \in \partial F(x_*)$ and $y^* \in Y^*$ such that $\langle y^*, y - G(x_*) \rangle \le 0$ for any $y \in K$,
$\| y^* \| \le 1$, and $(v_0 + c [D G(x_*)]^* y^*) \in - N_A(x_*)$. Denote $\lambda_* = c y^*$. Then by the definition
of normal cone and equality \eqref{eq:DirectDerivOfMaxFunc} one has
$$
  [L(\cdot, \lambda_*)]'(x_*, h) 
  = \max_{v \in \partial F(x_*)} \langle v, h \rangle + \langle \lambda_*, D G(x_*) h \rangle
  \ge \langle v_0 + c [D G(x_*)]^* y^*, h \rangle \ge 0
$$
for all $h \in T_A(x_*)$. Furthermore, from the facts that $\langle \lambda_*, y - G(x_*) \rangle \le 0$ for any 
$y \in K$, $K$ is a convex cone, and $G(x_*) \in K$ it follows that $\lambda_* \in K^*$ and 
$\langle \lambda_*, G(x_*) \rangle$. Therefore $\lambda_*$ is a Lagrange multiplier of $(\mathcal{P})$ at $x_*$.

\textbf{Part~\ref{stat:SuffOpt_PenaltyFunc}.} Let sufficient optimality condition \eqref{eq:SuffOptCond} be satisfied 
at $x_*$. Firstly, we show that $0 \in \relint( \partial \Phi_c(x_*) + N_A(x_*) )$ for some $c > 0$. Arguing by
reductio ad absurdum, suppose that $0 \notin \relint( \partial \Phi_c(x_*) + N_A(x_*) )$ for any $c > 0$. Then by
the separation theorem (see, e.g. \cite[Thrm.~2.17]{BonnansShapiro}) for any $n \in \mathbb{N}$ there exists $h_n \ne 0$
such that $\langle v, h_n \rangle \le 0$ for all $v \in \partial \Phi_n(x_*) + N_A(x_*)$. Replacing, if necessary,
$h_n$ by $h_n / |h_n|$ one can suppose that $|h_n| = 1$. Consequently, there exists a subsequence $\{ h_{n_k} \}$
converging to some $h_*$ with $|h_*| = 1$.

Fix any $c > 0$. As was noted in Remark~\ref{remark:Subdiff_ConeConstrPenFunc}, 
$\partial \Phi_c(x_*) \subseteq \partial \Phi_{n_k}(x_*)$ for any $n_k \ge c$. Therefore, for any $n_k \ge c$ and for
all $v \in \partial \Phi_c(x_*) + N_A(x_*)$ one has $\langle v, h_{n_k} \rangle \le 0$. Passing to the limit as 
$k \to \infty$ one obtains that $\langle v, h_* \rangle \le 0$ for any $v \in \partial \Phi_c(x_*) + N_A(x_*)$ and 
$c > 0$ or, equivalently,
\begin{equation} \label{eq:SepThrm_SubdiffConeConstrPenFunc}
  \langle v_1 + v_2 + v_3, h_* \rangle \le 0 \quad
  \forall v_1 \in \partial F(x_*), \: v_2 \in \bigcup_{c > 0} c \partial \varphi(x_*), \: v_3 \in N_A(x_*).
\end{equation}
Since both $\cup_{c > 0} c \partial \varphi(x_*)$ and $N_A(x_*)$ are cones (recall that $0 \in \partial \varphi(x_*)$),
one has $\langle v_2, h_* \rangle \le 0$ for all $v_2 \in \cup_{c > 0} c \partial \varphi(x_*)$, and
$\langle v_3, h_* \rangle \le 0$ for all $v_3 \in N_A(x_*)$. Consequently, by definition
$h_* \in N_A(x_*)^* = T_A(x_*)^{**} = T_A(x_*)$. Moreover, by Remark~\ref{remark:Subdiff_ConeConstrPenFunc} one
has 
$$
  \bigcup_{c > 0} c \partial \varphi(x_*) 
  = \Big\{ [D G(x)]^* y^* \in \mathbb{R}^d \Bigm| y^* \in (T_K(G(x)))^* \Big\}
$$
which implies that $\langle y^*, D G(x_*) h_* \rangle \le 0$ for all $y^* \in T_K(G(x_*))^*$, i.e. by the definition of
polar cone $D G(x_*) h_* \in [T_K(G(x_*))]^{**} = T_K(G(x_*))$. Thus, taking into account
\eqref{eq:SepThrm_SubdiffConeConstrPenFunc} one obtains that we found $h_* \in T_A(x_*) \setminus \{ 0 \}$ such that
$D G(x_*) h_* \in T_K(G(x_*))$ and $\max_{v \in \partial F(x_*)} \langle v, h_* \rangle \le 0$, which contradicts our
assumption that sufficient optimality condition \eqref{eq:SuffOptCond} holds true at $x_*$. 
Therefore, $0 \in \relint( \partial \Phi_c(x_*) + N_A(x_*) )$ for some $c > 0$.

Let us verify that the topological interior of the set $\partial \Phi_c(x_*) + N_A(x_*)$ is not empty. Then one can
conclude that $0 \in \interior( \partial \Phi_c(x_*) + N_A(x_*) )$. Arguing by reductio ad absurdum, suppose that the
interior of the set $\partial \Phi_c(x_*) + N_A(x_*)$ is empty. Then taking into account the fact that
$0 \in \relint( \partial \Phi_c(x_*) + N_A(x_*) )$ one can conclude that 
$$
  \mathcal{E} 
  = \affine( \partial \Phi_c(x_*) + N_A(x_*) ) = \linhull(\partial \Phi_c(x_*) + N_A(x_*)) \ne \mathbb{R}^d.
$$
Therefore, there exists $h_* \ne 0$ such that $\langle v, h_* \rangle = 0$ for all $v \in \mathcal{E}$. Bearing in mind
the equality $\affine( \partial \Phi_c(x_*) + N_A(x_*) ) = \affine \partial \Phi_c(x_*) + \affine N_A(x_*)$
and the fact that the affine hull of $\partial \Phi_c(x_*)$ does not depend on $c > 0$ by
Remark~\ref{remark:Subdiff_ConeConstrPenFunc} one obtains that $\langle v, h_* \rangle = 0$
for all $v \in \partial \Phi_r(x_*) + N_A(x_*)$ and $r > 0$. Consequently, inequality
\eqref{eq:SepThrm_SubdiffConeConstrPenFunc} is valid, which, as was shown above, contradicts \eqref{eq:SuffOptCond}.
Thus, $0 \in \interior( \partial \Phi_c(x_*) + N_A(x_*) )$ for some $c > 0$.

Suppose now that $0 \in \interior( \partial \Phi_c(x_*) + N_A(x_*) )$ for some $c \ge 0$. Then there exists $\rho > 0$
such that 
$$
  \max_{v \in \partial \Phi_c(x_*) + N_A(x_*)} \langle v, h \rangle \ge \rho |h| \quad \forall h \in \mathbb{R}^d.
$$
Note that by definition for any $h \in T_A(x_*)$ one has $\langle v, h \rangle \le 0$ for all $v \in N_A(x_*)$.
Therefore 
\begin{equation} \label{eq:ZeroInPenFuncSubdiffIneq}
  \max_{v \in \partial \Phi_c(x_*)} \langle v, h \rangle \ge \rho |h| \quad \forall h \in T_A(x_*).
\end{equation}
Fix any $\rho' \in (0, \rho)$. Let us check that $\Phi_c(x) \ge \Phi_c(x_*) + \rho' |x - x_*|$ for any $x \in A$ lying
sufficiently close to $x_*$, i.e. $\Phi_c$ satisfies the first order growth condition on $A$ at $x_*$.

Arguing by reductio ad absurdum, suppose that there exists a sequence $\{ x_n \} \subset A$ converging to $x_*$ such
that $\Phi_c(x_n) < \Phi_c(x_*) + \rho'|x_n - x_*|$. Put $h_n = (x_n - x_*) / |x_n - x_*|$ and 
$\alpha_n = |x_n - x_*|$. Without loss of generality one can suppose that the sequence $\{ h_n \}$ converges to some
vector $h_*$ with $|h_*| = 1$, which obviously belongs to $T_A(x_*)$, since $x_* + \alpha_n x_n = x_n \in A$ by
definition. Hence with the use of Lemma~\ref{lem:ConeConstrPenFunc_Subdiff} one obtains that
$$
  \rho' \ge \lim_{n \to \infty} \frac{\Phi_c(x_n) - \Phi_c(x_*)}{|x_n - x_*|} 
  = \lim_{n \to \infty} \frac{\Phi_c(x_* + \alpha_n h_n) - \Phi_c(x_*)}{\alpha_n} 
  = \max_{v \in \partial \Phi_c(x_*)} \langle v, h_* \rangle,
$$
which contradicts \eqref{eq:ZeroInPenFuncSubdiffIneq}. 

Suppose finally that $\Phi_c$ satisfies the first order growth condition on $A$ at $x_*$. Let us check that sufficient
optimality condition \eqref{eq:SuffOptCond} holds true at $x_*$. Indeed, by our assumption there exist $c \ge 0$, 
$\rho > 0$, and a neighbourhood $\mathcal{O}(x_*)$ of the point $x_*$ such that 
$\Phi_c(x) \ge \Phi_c(x_*) + \rho |x - x_*|$ for all $x \in \mathcal{O}(x_*) \cap A$.

Fix any $h \in T_A(x_*) \setminus \{ 0 \}$ such that $D G(x_*) h \in T_K(G(x_*))$. By the definition of contingent cone
there exist sequences $\{ \alpha_n \} \subset (0, + \infty)$ and $\{ h_n \} \subset \mathbb{R}^d$ such that 
$\alpha_n \to 0$ and  $h_n \to h$ as $n \to \infty$, and $x_* + \alpha_n h_n \in A$ for all $n \in \mathbb{N}$. Hence
for any sufficiently large $n$ one has $\Phi_c(x_* + \alpha_n h_n) - \Phi_c(x_*) \ge \rho \alpha_n |h_n|$, which
obviously implies that $\Phi_c'(x_*, h) \ge \rho |h|$. 

By Remark~\ref{remark:Subdiff_ConeConstrPenFunc} for any $v \in \partial \varphi(x_*)$ one can find a vector 
$y^*(v) \in (T_K(G(x_*)))^*$ such that $v = [D G(x_*)]^* y^*(v)$. Therefore for any $v \in \partial \varphi(x_*)$ one
has $\langle v, h \rangle = \langle y^*(v), D G(x_*) h \rangle \le 0$, since $D G(x_*) h \in T_K(G(x_*))$ by our
assumption. Consequently, by Lemma~\ref{lem:ConeConstrPenFunc_Subdiff} one has
$$
  \max_{v \in \partial F(x_*)} \langle v, h \rangle \ge \max_{v \in \partial \Phi_c(x_*)} \langle v, h \rangle
  = \Phi_c'(x_*, h) \ge \rho |h| > 0,
$$
i.e. sufficient optimality condition \eqref{eq:SuffOptCond} is satisfied at $x_*$.

\textbf{Part~\ref{stat:Exactness_GrowthConditions}.} If RCQ holds true at $x_*$, then by
\cite[Corollary~2.2]{Cominetti} there exist $a > 0$ and a neighbourhood $\mathcal{O}(x_*)$ of $x_*$ such that 
$$
  \varphi(x) = \dist(G(x), K) \ge a \dist(x, A \cap G^{-1}(K))  = a \dist(x, \Omega)
  \quad \forall x \in \mathcal{O}(x_*) \cap A,
$$
where, as above, $\Omega$ is the feasible region of the problem $(\mathcal{P})$. Let us check that the objective
function $F$ is Lipschitz continuous near $x_*$. Then by \cite[Corollary~2.9 and Prp.~2.7]{Dolgopolik_ExPen} one can
conclude that the penalty function $\Phi_c$ is locally exact at $x_*$.

Fix any $r > 0$ and denote $B(x_*, r) = \{ x \in \mathbb{R}^d \mid |x - x_*| \le r \}$. By a nonsmooth version of the
mean value theorem (see, e.g. \cite[Prp.~2]{Dolgopolik_MCD}) for any $x_1, x_2 \in B(x_*, r)$ there exist a point 
$z \in \co\{ x_1, x_2 \} \subset B(x_*, r)$ and 
$v \in \partial F(z)$ such that $F(x_1) - F(x_2) = \langle v, x_1 - x_2 \rangle$. 
Define $L = \max\{ |\nabla_x f(x, \omega)| \mid x \in B(x_*, r), \omega \in W \} < + \infty$. By definition
$v$ belongs to the convex hull $\co\{ \nabla_x f(z, \omega) \mid \omega \in W(z) \}$, which yields $|v| \le L$. Thus, 
$|F(x_1) - F(x_2)| \le L |x_1 - x_2|$ for all $x_1, x_2 \in B(x_*, r)$, i.e. $F$ is Lipschitz continuous near $x_*$.

It remains to note that if RCQ and the first order growth condition for the problem $(\mathcal{P})$ hold at $x_*$,
then by Theorem~\ref{thrm:SuffOptCond} sufficient optimality condition \eqref{eq:SuffOptCond} holds true at $x_*$,
which by the second part of this theorem implies that $\Phi_c$ satisfies the first growth condition on $A$ at $x_*$.
The converse statement, as was noted before this theorem, holds true regardless of RCQ.
\end{proof}

\begin{remark}
{(i)~From the proof of the previous theorem it follows that $\lambda_*$ is a Lagrange multiplier of $(\mathcal{P})$ at
$x_*$ iff $0 \in \partial \Phi_c(x_*) + N_A(x_*)$ for any $c \ge \| \lambda_* \|$.
}

\noindent{(ii)~Observe that if $0 \in  \partial \Phi_c(x_*) + N_A(x_*)$ for some $c \ge 0$, then 
$0 \in \partial \Phi_r(x_*) + N_A(x_*)$ for any $r \ge c$, since $\partial \Phi_c(x_*) \subseteq \partial \Phi_r(x_*)$
by Remark~\ref{remark:Subdiff_ConeConstrPenFunc}. Furthermore, from this inclusion it follows that 
if $0 \in \interior( \partial \Phi_c(x_*) + N_A(x_*) )$ for some $c \ge 0$, then
$0 \in \interior( \partial \Phi_r(x_*) + N_A(x_*) )$ for any $r \ge c$ as well.
}

\noindent{(iii)~Unlike the set $\mathcal{D}(x_*)$ from Theorem~\ref{thrm:EquivOptCond_Subdiff},
the set $\partial \Phi_c(x_*) + N_A(x_*)$ is always closed as the sum of a compact and a closed sets. Moreover, 
the inclusion $\partial \Phi_c(x_*) + N_A(x_*) \subset \mathcal{D}(x_*)$ holds true for any $c \ge 0$. Indeed, by
Lemma~\ref{lem:ConeConstrPenFunc_Subdiff} one has $\partial \Phi_c(x_*) = \partial F(x_*) + c \partial \varphi(x_*)$.
Therefore, it is sufficient to check that $\partial \varphi(x_*) \subset \mathcal{N}(x_*)$, since $\mathcal{N}(x_*)$ is
a cone. Choose any $z^* \in \partial \varphi(x_*)$. By Lemma~\ref{lem:ConeConstrPenFunc_Subdiff} 
one has $z^* = [D G(x_*)]^* y^*$ for some $y^* \in Y^*$ such that $\| y^* \| \le 1$ and 
$\langle y^*, y - G(x_*) \rangle \le 0$ for all $y \in K$. Observe that $0 \in K$ and $2 G(x_*) \in K$, since $K$ is a
cone and $G(x_*) \in K$, which yields $\langle y^*, G(x_*) \rangle = 0$. Furthermore, from the fact that $K$ is a convex
cone it follows that $K + G(x_*) \subseteq K$, which implies that $\langle y^*, y \rangle \le 0$ for all $y \in K$, i.e.
$y^* \in K^*$. Thus, one can conclude that $z^* \in [D G(x_*)](K^* \cap \linhull(G(x_*))^{\perp}) = \mathcal{N}(x_*)$,
i.e. $\partial \varphi(x_*) \subset \mathcal{N}(x_*)$.
}

\noindent{(iv)~From the proof of the second part of the theorem above it follows that the inclusion
$0 \in \interior(\partial \Phi_c(x_*) + N_A(x_*))$ is a sufficient optimality condition for the penalised problem
\eqref{probl:PenalisedProblem}. Moreover, both this condition and optimality condition \eqref{eq:SuffOptCond} are
sufficient conditions for the local exactness of $\Phi_c$. Finally, note that arguing in the same way as in the proof of
the first part of Theorem~\ref{thrm:EquivOptCond_Subdiff} one can easily check that the inclusion
$0 \in \partial \Phi_c(x_*) + N_A(x_*)$ is a necessary optimality condition for problem \eqref{probl:PenalisedProblem}.
\qed
}
\end{remark}

\subsection{Alternance optimality conditions and cadres}
\label{subsect:Alternance_Cadre}

Note that the optimality condition $0 \in \mathcal{D}(x_*)$ from the previous section means that zero can be represented
as the sum of some vectors from the sets $\partial F(x_*)$, $\mathcal{N}(x_*)$, and $N_A(x_*)$. Our aim is to show that
these vectors can be chosen in such a way that they have some useful additional properties, which, in particular, allow
one to check whether the sufficient optimality condition $0 \in \interior \mathcal{D}(x_*)$ is satisfied.

Let $Z \subset \mathbb{R}^d$ be a set consisting of $d$ linearly independent vectors. Let also 
$\eta(x_*) \subseteq \mathcal{N}(x_*)$ and $n_A(x_*) \subseteq N_A(x_*)$ be such that 
$\mathcal{N}(x_*) = \cone \eta(x_*)$ and $N_A(x_*) = \cone n_A(x_*)$, where
$$
  \cone D = \Big\{ \sum_{i = 1}^n \alpha_i x_i \Bigm| 
  x_i \in D, \enspace \alpha_i \ge 0, \enspace i \in \{ 1, \ldots, n \}, \enspace n \in \mathbb{N} \Big\}
$$
is the \textit{convex conic hull} of a set $D \subset \mathbb{R}^d$ (i.e. the smallest convex cone containing the set
$D$). Usually, one chooses $\eta(x_*)$ and $n_A(x_*)$ as the sets of those vectors that correspond to extreme rays
of the cones $\mathcal{N}(x_*)$ and $N_A(x_*)$ respectively.

\begin{definition} \label{def:AlternanceOptCond}
Let $p \in \{ 1, \ldots, d + 1 \}$ be fixed and $x_*$ be a feasible point of the problem $(\mathcal{P})$. One says that
\textit{a $p$-point alternance} exists at $x_*$, if there exist $k_0 \in \{ 1, \ldots, p \}$, 
$i_0 \in \{ k_0 + 1, \ldots, p \}$, vectors
\begin{gather} \label{eq:AlternanceDef}
  V_1, \ldots, V_{k_0} \in \Big\{ \nabla_x f(x_*, \omega) \Bigm| \omega \in W(x_*) \Big\}, \\
  V_{k_0 + 1}, \ldots, V_{i_0} \in \eta(x_*), \quad
  V_{i_0 + 1}, \ldots, V_p \in n_A(x_*), \label{eq:AlternanceDef2}
\end{gather}
and vectors $V_{p + 1}, \ldots, V_{d + 1} \in Z$ such that the d-th order determinants $\Delta_s$ of the matrices
composed of the columns $V_1, \ldots, V_{s - 1}, V_{s + 1}, \ldots V_{d + 1}$ satisfy the following conditions:
\begin{gather} \label{eq:DeterminantsProp}
  \Delta_s \ne 0, \quad s \in \{ 1, \ldots, p \}, \quad
  \sign \Delta_s = - \sign \Delta_{s + 1}, \quad s \in \{ 1, \ldots, p - 1 \}, \\
  \Delta_s = 0, \quad s \in \{ p + 1, \ldots d + 1 \}. \label{eq:DeterminantsProp2}
\end{gather}
Such collection of vectors $\{ V_1, \ldots, V_p \}$ is called a $p$-point alternance at $x_*$. Any $(d + 1)$-point
alternance is called \textit{complete}.
\end{definition}

\begin{remark}
{(i)~Note that in the case of complete alternance one has
$$
  \Delta_s \ne 0 \quad s \in \{ 1, \ldots, d + 1 \}, \quad
  \sign \Delta_s = - \sign \Delta_{s + 1} \quad s \in \{ 1, \ldots, d \},
$$
i.e. the determinants $\Delta_s$, $s \in \{1, \ldots, d + 1 \}$ are not equal to zero and have \textit{alternating}
signs, which explains the term \textit{alternance}.
}

\noindent{(ii)~It should be mentioned that the sets $\eta(x_*)$ and $n_A(x_*)$ are introduced in order to
simplify verification of alternance optimality conditions. It is often difficult to deal with the entire cones
$\mathcal{N}(x_*)$ and $N_A(x_*)$. In turn, the introduction of the sets $\eta(x_*)$ and $n_A(x_*)$ allows
one to use only extreme rays of $\mathcal{N}(x_*)$ and $N_A(x_*)$ respectively. \qed
}
\end{remark}

Before we proceed to an analysis of optimality conditions, let us first show that the definition of $p$-point
alternance with $p \le d$ is invariant with respect to the choice of the set $Z$ and is directly connected to the notion
of \textit{cadre} (meaning \textit{frame}) of a minimax problem (see, e.g. \cite{Descloux,ConnLi92}).

\begin{proposition} \label{prp:AlternanceVsCadre}
Let $x_*$ be a feasible point of the problem $(\mathcal{P})$. Then a $p$-point alternance with 
$p \in \{ 1, \ldots, d + 1 \}$ exists at $x_*$ if and only if there exist $k_0 \in \{ 1, \ldots, p \}$, 
$i_0 \in \{ k_0 + 1, \ldots, p \}$, and vectors
\begin{gather} \label{eq:CadreDef}
  V_1, \ldots, V_{k_0} \in \Big\{ \nabla_x f(x_*, \omega) \Bigm| \omega \in W(x_*) \Big\}, \\
  V_{k_0 + 1}, \ldots, V_{i_0} \in \eta(x_*), \quad
  V_{i_0 + 1}, \ldots, V_p \in n_A(x_*).	\label{eq:CadreDef2}
\end{gather}
such that $\rank([V_1, \ldots, V_p]) = p - 1$ and 
\begin{equation} \label{eq:CadreMultDef}
  \sum_{i = 1}^p \beta_i V_i = 0
\end{equation}
for some $\beta_i > 0$, $i \in \{ 1, \ldots, p \}$. Furthermore, a collection of vectors $\{ V_1, \ldots, V_p \}$
satisfying \eqref{eq:CadreDef} and \eqref{eq:CadreDef2} is a $p$-point alternance at $x_*$ iff 
$\rank([V_1, \ldots, V_p]) = p - 1$ and \eqref{eq:CadreMultDef} holds true.
\end{proposition}

\begin{proof}
Let a $p$-point alternance exist at $x_*$ and let vectors $V_i \in \mathbb{R}^d$ and indices
$k_0 \in \{ 1, \ldots, p \}$, $i_0 \in \{ k_0 + 1, \ldots, p \}$ be from the definition of $p$-point alternance.
Consider the system of linear equations $\sum_{i = 2}^{d + 1} \beta_i V_i = - V_1$ with respect to $\beta_i$. Solving
this system with the use of Cramer's rule one obtains that $\beta_i = (-1)^{i-1} \Delta_i / \Delta_1$ for all 
$i \in \{ 2, \ldots, d + 1 \}$, where $\Delta_i$ are from the definition of $p$-point alternance. Taking into account
\eqref{eq:DeterminantsProp} and \eqref{eq:DeterminantsProp2} one obtains that $\beta_i > 0$ for any 
$i \in \{ 2, \ldots, p \}$ and $\beta_i = 0$ for all $i \in \{ p + 1, \ldots, d + 1 \}$. Note that zero coefficients
$\beta_i$ correspond exactly to those $V_i$ that belong to $Z$. 

Thus, one has $V_1 + \sum_{i = 2}^p \beta_i V_i = 0$ and $\beta_i > 0$ for all $i \in \{ 2, \ldots, p \}$.
Furthermore, from the fact that that by the definition of $p$-point alternance one has
$\Delta_1 = \determ([V_2, \ldots, V_{d + 1}]) \ne 0$ it follows that the vectors $V_2, \ldots, V_p$ are linearly
independent, which implies that $\rank([V_1, \ldots, V_p]) = p - 1$. Hence taking into account
\eqref{eq:AlternanceDef} and \eqref{eq:AlternanceDef2} one obtains that the proof of the ``only if'' part of the
proposition is complete.

Let us prove the converse statement. Suppose at first that $p = 1$. Then $V_1 = 0$ due to \eqref{eq:CadreMultDef}. Take
as $V_2, \ldots, V_{d + 1}$ all vectors from the set $Z$ in an arbitrary order. Since these vectors are linearly
independent, one has $\Delta_1 = \determ([V_2, \ldots, V_{d + 1}]) \ne 0$, and the system 
$\sum_{i = 2}^{d + 1} \gamma_i V_i = - V_1$ has the unique solution $\gamma_i = 0$ for all $i$. Solving this system
with the use of Cramer's rule one obtains that $0 = \gamma_i = (-1)^{i - 1} \Delta_i / \Delta_1$ for all 
$i \in \{ 2, \ldots, d + 1 \}$, where $\Delta_i = \determ([V_1, \ldots, V_{i - 1}, V_{i + 1}, \ldots V_{d + 1}])$.
Thus, $\Delta_i = 0$ for all $i \ge 2$ and the collection $\{ V_1, \ldots, V_{d + 1} \}$ satisfies the definition of
$1$-point alternance.

Suppose now that $p \ge 2$. Rewrite \eqref{eq:CadreMultDef} as follows: 
$\sum_{i = 2}^p (\beta_i / \beta_1) V_i = - V_1$.
Taking into account this equality and the fact that $\rank([V_1, \ldots, V_p]) = p - 1$ one can conclude that 
the vectors $V_2, \ldots, V_p$ are linearly independent. Therefore one can choose
$V_{p + 1}, \ldots, V_{d + 1} \in Z$ such that the vectors $V_2, \ldots, V_{d + 1}$ are linearly independent as well.
Consequently, $\Delta_1 = \determ([V_2, \ldots, V_{d + 1}]) \ne 0$, and the system of linear equations
$\sum_{i = 2}^{d + 1} \gamma_i V_i = - V_1$ with respect to $\gamma_i$ has the unique solution:
$\gamma_i = \beta_i / \beta_1 > 0$ for any $i \in \{ 2, \ldots, p \}$, and $\gamma_i = 0$ for all $i \ge p + 1$. On the
other hand, by Cramer's rule one has $\gamma_i = (-1)^{i - 1} \Delta_i / \Delta_1$ for all $i$, where
$\Delta_i = \determ([V_1, \ldots, V_{i - 1}, V_{i + 1}, \ldots V_{d + 1}])$. Hence conditions
\eqref{eq:DeterminantsProp} and \eqref{eq:DeterminantsProp2} hold true and the collection 
$\{ V_1, \ldots, V_{d + 1} \}$ satisfies the definition of $p$-point alternance.
\end{proof}

\begin{remark}
{(i)~Any collection of vectors $V_1, \ldots, V_p$ with $p \in \{ 1, \ldots, d + 1 \}$ satisfying \eqref{eq:CadreDef},
\eqref{eq:CadreDef2} and such that 
$\rank([V_1, \ldots, V_p]) = \rank([V_1, \ldots, V_{i - 1}, V_{i + 1}, \ldots, V_p]) = p - 1$ for any 
$i \in \{ 1, \ldots, p \}$ is called a $p$-point \textit{cadre} for the problem $(\mathcal{P})$ at $x_*$. One can easily
verify that a collection $V_1, \ldots, V_p$ satisfying \eqref{eq:CadreDef}, \eqref{eq:CadreDef2} is a $p$-point cadre at
$x_*$ iff $\rank([V_1, \ldots, V_p]) = p - 1$ and $\sum_{i = 1}^p \beta_i V_i = 0$ for some $\beta_i \ne 0$, 
$i \in \{ 1, \ldots, p \}$. Any such $\beta_i$ are called \textit{cadre multipliers}. Thus, the proposition above
can be reformulated as follows: a $p$-point alternance exists at $x_*$ iff a $p$-point cadre with positive cadre
multipliers exists at this point. Furthermore, a collection $\{ V_1, \ldots, V_p \}$ with $p \in \{ 1, \ldots, d + 1 \}$
is a $p$-point alternance at $x_*$ iff it is a $p$-point cadre with positive cadre multipliers, which implies that the
definition of $p$-point alternance is invariant with respect to the set $Z$. Note finally that optimality conditions in
terms of such cadres were utilised in \cite{ConnLi92} to design an efficient method for solving unconstrained minimax
problems, while the definition of \textit{cadre} was first given by Descloux in \cite{Descloux}.
}

\noindent{(ii)~It is worth mentioning that from the previous proposition it follows that if any $d$ vectors from the
set $\{ \nabla_x f(x_*, \omega) \mid \omega \in W(x_*) \} \cup \eta(x_*) \cup n_A(x_*)$ are linearly
independent, then only a complete alternance can exist at $x_*$. \qed
}
\end{remark}

Our next goal is demonstrate that both necessary and sufficient optimality conditions for the problem $(\mathcal{P})$
can be written in an \textit{alternance} form. To this end, we will need the following simple geometric result
illustrated by Figure~\ref{fig:ell1_interior}. This result allows one to easily prove that the origin belongs to 
the interior or the relative interior of certain polytopes.

\begin{figure}[t]
\centering
\includegraphics[width=0.5\linewidth]{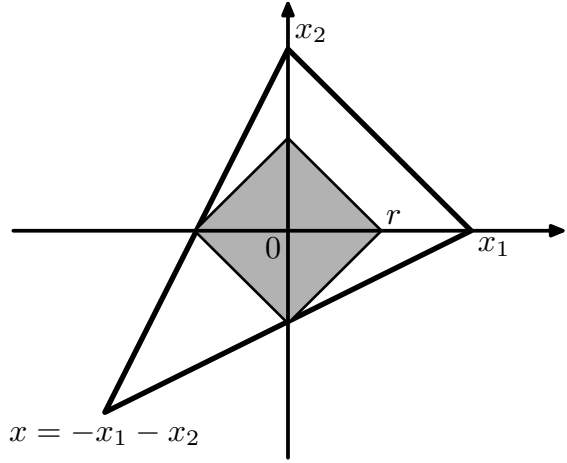}
\caption{The polytope $S = \co\{ x_1, x_2, -x_1 - x_2 \}$ with $x_1 = (1, 0)^T$ and $x_2 = (0, 1)^T$ contains the open
$\ell_1$ ball centered at zero with sufficiently small radius $r > 0$ that can be described as
$\{ z = \alpha_1 x_1 + \alpha_2 x_2 \in \mathbb{R}^2 \mid |\alpha_1| + |\alpha_2| < r \}$.}
\label{fig:ell1_interior}
\end{figure}

\begin{lemma} \label{lem:SimplexZeroInterior}
Let $x_1, \ldots, x_k \in \mathbb{R}^d$ be given vectors, $x = \sum_{i = 1}^k \beta_i x_i$ for some $\beta_i > 0$, and
$S = \co\{ x_1, \ldots, x_k, - x \}$. Then there exists $r > 0$ such that
\begin{equation} \label{eq:ell1_SimplexRelInt}
  \Big\{ z = \sum_{i = 1}^k \alpha_i x_i \Bigm| \sum_{i = 1}^k |\alpha_i| < r \Big\} \subset S.
\end{equation}
\end{lemma}

\begin{proof}
Observe that $0 \in S$, since
$$
  0 = \frac{1}{1 + \beta_1 + \ldots + \beta_k} x 
  + \frac{1}{1 + \beta_1 + \ldots + \beta_k} \sum_{i = 1}^k \beta_i x_i \in S.
$$
Hence, in particular, $\co\{ 0, z \} \subset S$ for all $z \in S$. 

Denote $\gamma_i = 1 + \sum_{j \ne i} \beta_j$. Then
$$
  - \frac{\beta_i}{\gamma_i} x_i = \frac{1}{\gamma_i} x
  + \sum_{j \ne i} \frac{\beta_j}{\gamma_i} x_j \in S \quad \forall i \in \{ 1, \ldots, k \}. 
$$
Define $r = \min\{ 1, \beta_1 / \gamma_1, \ldots, \beta_k / \gamma_k \}$. Then taking into account the fact
that $\co\{ 0, z \} \subset S$ for all $z \in S$ one obtains that $\pm r x_i \in S$ for all $i \in \{ 1, \ldots, k \}$.

Fix any $z = \sum_{i = 1}^k \alpha_i x_i$ with $\theta(z) = \sum_{i = 1}^k |\alpha_i| < r$. If $\theta(z) = 0$, then 
$z = 0$ and $z \in S$. Therefore, suppose that $\theta(z) \ne 0$. 
Then $\pm \theta(z) x_i \in \co \{ \pm r x_i \} \subset S$, which implies that
$$
  z = \sum_{i = 1}^k \frac{|\alpha_i|}{\theta(z)} \Big( \sign(\alpha_i) \theta(z) x_i \Big) \in S
$$
(here $\sign(0) = 0$). Thus, \eqref{eq:ell1_SimplexRelInt} holds true.
\end{proof}

\begin{theorem} \label{thrm:AlternanceCond}
Let $x_*$ be a feasible point of the problem $(\mathcal{P})$. Then: 
\begin{enumerate}
\item{$0 \in \mathcal{D}(x_*)$ iff for some $p \in \{1, \ldots, d + 1 \}$ a $p$-point alternance exists at~$x_*$;
\label{stat:NessOpt_Alternance}}

\item{if a complete alternance exists at $x_*$, then $0 \in \interior \mathcal{D}(x_*)$ and $\partial F(x_*) \ne \{ 0
\}$.
\label{stat:CompleteAlternance_SuffOpt}}
\end{enumerate}
\end{theorem}

\begin{proof}
\textbf{Part~\ref{stat:NessOpt_Alternance}.} ``$\implies$'' Let $0 \in \mathcal{D}(x_*)$. If 
$0 \in \partial F(x_*) = \co\{ \nabla_x f(x_*, \omega) \mid \omega \in W(x_*) \}$, then by Carath\'{e}odory's theorem
(see, e.g. \cite[Corollary~17.1.1]{Rockafellar}) zero can be expressed as a convex combination
of $d + 1$ or fewer affinely independent vectors from $\{ \nabla_x f(x_*, \omega) \mid \omega \in W(x_*) \}$.
Thus, there exist $p \in \{ 1, \ldots, d + 1 \}$, $V_i \in \{ \nabla_x f(x_*, \omega) \mid \omega \in W(x_*) \}$, and
$\alpha_i > 0$, $i \in \{ 1, \ldots, p \}$, such that the vectors $V_i$ are affinely independent and
\begin{equation} \label{eq:ZeroElementOfMaxFuncSubdiff}
  0 = \sum_{i = 1}^p \alpha_i V_i, \quad \sum_{i = 1}^p \alpha_i = 1.
\end{equation}
If $p = 1$, then denote by $V_2, \ldots, V_{d + 1}$ all vectors from the set $Z$. Then $\Delta_1
\ne 0$, and  $\Delta_s = 0$ for all $s \in \{ 2, \ldots, d + 1 \}$, since $V_1 = 0$, that is, a $1$-point alternance
exists at $x_*$. Otherwise, note that by the definition of affine independence the vectors 
$V_2 - V_1, \ldots, V_p - V_1$ are linearly independent. Hence taking into account
\eqref{eq:ZeroElementOfMaxFuncSubdiff} and the fact that 
$\linhull( V_2 - V_1, \ldots, V_p - V_1) \subseteq \linhull(V_1, \ldots, V_p)$ one obtains that
$\dimens \linhull(V_1, \ldots, V_p) = p - 1$. Consequently, the collection $\{ V_1, \ldots, V_p \}$ contains exactly 
$p - 1$ linearly independent vectors. Renumbering $V_i$, if necessary, one can suppose that the vectors 
$V_2, \ldots, V_p$ are linearly independent. Since the set $Z$ contains $d$ linearly independent vectors, one can
choose vectors $V_{p + 1}, \ldots, V_{d + 1} \in Z$ in such a way that the vectors $V_2, \ldots, V_{d + 1}$ are linearly
independent, which yields $\Delta_1 \ne 0$. 

Now, consider the system of linear equations $- V_1 = \sum_{i = 2}^{d + 1} \beta_i V_i$ with respect to $\beta_i$.
Solving this system with the use of Cramer's rule and bearing in mind equalities \eqref{eq:ZeroElementOfMaxFuncSubdiff}
one obtains that $\beta_i = (-1)^{i - 1} \Delta_i / \Delta_1 = \alpha_i / \alpha_1 > 0$ for any 
$i \in \{ 2, \ldots, p \}$, and $\beta_i = (-1)^{i - 1} \Delta_i / \Delta_1 = 0$ for any $i \ge p + 1$. Thus, conditions
\eqref{eq:DeterminantsProp} and \eqref{eq:DeterminantsProp2} hold true, i.e. a $p$-point alternance exists at $x_*$.
Therefore, one can suppose that $0 \notin \partial F(x_*)$.

Since $0 \in \mathcal{D}(x_*)$ and $0 \notin \partial F(x_*)$, there exist $k, r, \ell \in \mathbb{N}$, 
$\omega_i \in W(x_*)$, $\alpha_i \in (0, 1]$ , $u_j \in \eta(x_*)$, $\beta_j \ge 0$, $z_s \in n_A(x_*)$, and 
$\gamma_s \ge 0$ (here $i \in \{ 1, \ldots, k \}$, $j \in \{ 1, \ldots, r \}$, and $s \in \{ 1, \ldots, \ell \}$) such
that
$$
  0 = \sum_{i = 1}^k \alpha_i v_i + \sum_{j = 1}^r \beta_j u_j + \sum_{s = 1}^{\ell} \gamma_s z_s, 
  \quad \sum_{i = 1}^k \alpha_i = 1,
$$
where $v_i = \nabla_x f(x_*, \omega_i)$ for all $i \in \{ 1, \ldots, k \}$. Hence
$$
  \sum_{i = 2}^k \frac{\alpha_i}{\alpha_1} v_i + \sum_{j = 1}^r \frac{\beta_j}{\alpha_1} u_j + 
  \sum_{s = 1}^{\ell} \frac{\gamma_s}{\alpha_1} z_s = - v_1,
$$
i.e. $- v_1$ belongs to $\cone(\mathcal{E})$ with 
$\mathcal{E} = \{ v_2, \ldots, v_k, u_1, \ldots, u_r, z_1, \ldots, z_{\ell} \}$. Applying a simple modification of the
Carath\'eodory's theorem to the case of convex conic combinations (see, e.g. \cite[Corollary~17.1.2]{Rockafellar}) one
obtains that there exist $p \in \{ 2, \ldots, d + 1 \}$ and linearly independent vectors 
$V_2, \ldots V_p \in \mathcal{E}$ such that $- v_1 = \sum_{i = 2}^p \lambda_i V_i$ for some $\lambda_i > 0$. 
Clearly, one can suppose that there exist $k_0 \in \{ 1, \ldots p \}$ and $i_0 \in \{ k_0 + 1, \ldots, p \}$ such that
\eqref{eq:AlternanceDef} and \eqref{eq:AlternanceDef2} hold true.

Put $V_1 = v_1$, and choose vectors $V_{p + 1}, \ldots V_{d + 1}$ from the set $Z$ in such a way that the vectors 
$V_2, \ldots, V_{d + 1}$ are linearly independent. Then one obtains that the systems
\begin{equation} \label{eq:AlternanceLinSystm}
  \sum_{i = 2}^{d + 1} \beta_i V_i = - V_1.
\end{equation}
has the unique solution $\beta_i = \lambda_i$, if $2 \le i \le p$, and $\beta_i = 0$, if $p + 1 \le i \le d + 1$.
Applying Cramer's rule to system \eqref{eq:AlternanceLinSystm} one gets that 
$\beta_i = (-1)^{i - 1} \Delta_i / \Delta_1$ for all $i \in \{ 2, \ldots, d + 1 \}$, where $\Delta_i$ are from
Def.~\ref{def:AlternanceOptCond}, which implies that \eqref{eq:DeterminantsProp} and \eqref{eq:DeterminantsProp2} hold
true. Thus, a $p$-point alternance exists at the point $x_*$.

\textbf{Part~\ref{stat:NessOpt_Alternance}.} ``$\impliedby$''. Let vectors $V_1, \ldots, V_{d + 1}$ be from the
definition of $p$-point alternance. Applying Cramer's rule to system \eqref{eq:AlternanceLinSystm} one obtains
that
$$
  - V_1 = \sum_{i = 2}^p \beta_i V_i, \quad 
  \beta_i = (- 1)^{i - 1} \frac{\Delta _i}{\Delta_1} > 0 \quad \forall i \in \{ 2, \ldots, p \}.
$$
Denote $\beta_0 = 1 + \beta_2 + \ldots + \beta_{k_0} > 0$, and define $\alpha_1 = 1 / \beta_0 > 0$ and
$\alpha_i = \beta_i / \beta_0 \ge 0$ for all $i \in \{ 2, \ldots, d + 1 \}$. Then one has
\begin{equation} \label{eq:ConvConicCombOfAlterPoints}
  \sum_{i = 1}^p \alpha_i V_i = 0, \quad \sum_{i = 1}^{k_0} \alpha_i = 1,
\end{equation}
i.e. $v_1 + v_2 + v_3 = 0$, where
$$
  v_1 = \sum_{i = 1}^{k_0} \alpha_i V_i, \quad v_2 = \sum_{i = k_0 + 1}^{i_0} \alpha_i V_i, \quad
  v_3 = \sum_{i = i_0 + 1}^{p} \alpha_i V_i
$$
(here, $v_2 = v_3 = 0$, if $k_0 = p$, and $v_3 = 0$, if $i_0 = p$). From the definition of alternance and the second
equality in \eqref{eq:ConvConicCombOfAlterPoints} it follows that $v_1 \in \partial F(x_*)$, $v_2 \in \mathcal{N}(x_*)$,
and $V_3 \in N_A(x_*)$. Thus, $0 \in \mathcal{D}(x_*)$.

\textbf{Part~\ref{stat:CompleteAlternance_SuffOpt}}. Suppose that a complete alternance $V_1, \ldots, V_{d + 1}$ exists
at $x_*$. Note that $V_1 \ne 0$, since all $\Delta_i$ are nonzero, which implies that $\partial F(x_*) \ne \{ 0 \}$.

Applying Cramer's rule to system \eqref{eq:AlternanceLinSystm} one gets that
\begin{equation} \label{eq:ReDefOfComplAlternance}
  - V_1 = \sum_{i = 2}^{d + 1} \beta_i V_i, \quad
  \beta_i = (- 1)^{i - 1} \frac{\Delta _i}{\Delta_1} > 0 \quad \forall i \in \{ 2, \ldots, d + 1 \}.
\end{equation}
Denote $\beta_0 = 1 + \beta_2 + \ldots + \beta_{k_0} > 0$, and define $\alpha_1 = 1 / \beta_0 > 0$ and
$\alpha_i = \beta_i / \beta_0 > 0$ for all $i \in \{ 2, \ldots, d + 1 \}$. Then \eqref{eq:ConvConicCombOfAlterPoints}
with $p = d + 1$ holds true.

Recall that by the definition of alternance $V_1, \ldots, V_{k_0} \in \partial F(x_*)$. Therefore, 
$V_1, \ldots, V_{k_0} \in \mathcal{D}(x_*) = \partial F(x_*) + \mathcal{N}(x_*) + N_A(x_*)$, since 
$0 \in \mathcal{N}(x_*)$ and $0 \in N_A(x_*)$. Moreover, from \eqref{eq:ConvConicCombOfAlterPoints} and the fact that
both $\mathcal{N}(x_*)$ and $N_A(x_*)$ are convex cones it follows that
\begin{align*}
  V_i = 0 + V_i &= \sum_{j = 1}^{k_0} \alpha_j V_j + \sum_{j = k_0 + 1}^{i_0} \alpha_j V_j + V_i
  + \sum_{j = i_0 + 1}^{d + 1} \alpha_j V_j \\
  &\in \partial F(x_*) + \mathcal{N}(x_*) + N_A(x_*) = \mathcal{D}(x_*)
\end{align*}
for any $i \in \{ k_0 + 1, \ldots, d + 1 \}$. 
Therefore $S(x_*) = \co\{ V_1, \ldots, V_{d + 1} \} \subset \mathcal{D}(x)$ by virtue of the fact that
$\mathcal{D}(x_*)$ is a convex set.

Let $e_1, \ldots, e_d$ be the canonical basis of $\mathbb{R}^d$ and $\overline{e} = (- \beta_1, \ldots, - \beta_d)^T$,
where $\beta_i$ are from \eqref{eq:ReDefOfComplAlternance}. Denote $S = \co \{ e_1, \ldots, e_d, \overline{e} \}$ and
define a linear mapping $T \colon \mathbb{R}^d \to \mathbb{R}^d$ by setting $T e_i = V_{i + 1}$ for 
all $i \in \{ 1, \ldots, d \}$. Then $T \overline{e} = V_1$ due to \eqref{eq:ReDefOfComplAlternance} and $T S = S(x_*)$.
Bearing in mind the fact that by the definition of complete alternance 
$\Delta_1 = \determ([V_2, \ldots, V_{d + 1}]) \ne 0$, i.e. the vectors $V_2, \ldots, V_{d + 1}$ are linearly
independent, one obtains that $T$ is a linear bijection, which, in particular, implies that $T$ is an open mapping. Let
us show that $0 \in \interior S$. Then taking into account the facts that $T(\interior S)$ is an open set and 
by definitions $0 \in T(\interior S) \subset S(x_*) \subset \mathcal{D}(x_*)$ one arrives at the required result.

For any $x = (x^{(1)}, \ldots, x^{(d)})^T \in \mathbb{R}^d$ denote $\| x \|_1 = |x^{(1)}| + \ldots + |x^{(d)}|$.
Applying Lemma~\ref{lem:SimplexZeroInterior} with $k = d$, $x_i = e_i$ for all $i \in \{ 1, \ldots, d \}$, and 
$x = - \overline{e}$ one obtains that there exists $r > 0$ such that
$\{ x \in \mathbb{R}^d \mid \| x \|_1 < r \} \subset S$, that is, $0 \in \interior S$, and the proof is
complete.
\end{proof}

Thus, the existence of a $p$-point alternance (or, equivalently, the existence of a $p$-point cadre with positive cadre
multipliers) at a feasible point $x_*$ for some $p \in \{ 1, \ldots, d + 1 \}$  is a necessary optimality condition for
the problem $(\mathcal{P})$, while the existence of a complete alternance is a sufficient optimality condition, which
by Theorems~\ref{thrm:SuffOptCond} and \ref{thrm:EquivOptCond_Subdiff} implies that the first order growth condition
holds at $x_*$. As the following example shows, the converse statement is not true, that is, the sufficient optimality
condition $0 \in \interior \mathcal{D}(x_*)$ does not necessarily imply that a complete alternance exists at $x_*$.

\begin{example} \label{exmpl:ComplAtern_CounterExampl}
Consider the unconstrained problem
\begin{equation} \label{probl:ComplAltern_CounterExmpl}
  \min_{x \in \mathbb{R}^d} \: F(x) = \| x \|_{\infty} = \max\big\{ \pm x^{(1)}, \ldots, \pm x^{(d)} \big\}.
\end{equation}
Clearly, $x_* = 0$ is a point of global minimum of this problem and the first order growth condition holds at $x_*$,
since, as is easy to see, $F(x) \ge |x| / \sqrt{n}$ for all $x \in \mathbb{R}^d$. Observe that by definition
$\partial F(0) = \co\big\{ \pm e_1, \ldots, \pm e_d \big\}$. Thus, in accordance with Thrms.~\ref{thrm:SuffOptCond}
and \ref{thrm:EquivOptCond_Subdiff} the sufficient optimality condition $0 \in \interior \partial F(0)$ is satisfied.
However, a complete alternance does not exists at $x_* = 0$.

Indeed, suppose that a $p$-point alternance for some $p \in \{ 1, \ldots, d + 1 \}$ exists at $x_*$. Then by 
Proposition~\ref{prp:AlternanceVsCadre} there exist $V_1, \ldots, V_p \in \{ \pm e_1, \ldots, \pm e_d \}$ such that
$\rank([V_1, \ldots, V_p]) = p - 1$ and $\sum_{i = 1}^p \beta_i V_i = 0$ for some $\beta_i > 0$. Renumbering vectors
$V_i$, if necessary, one can suppose that the vectors $V_1, \ldots, V_{p - 1}$ are linearly independent. Hence taking
into account the fact that each $V_i$ is equal to either $e_{k_i}$ or $-e_{k_i}$ for some $k_i \in \{ 1, \ldots, d \}$
and $\sum_{i = 1}^p \beta_i V_i = 0$ for some $\beta_i > 0$ one obtains that $p = 2$. Thus, for any $d \in \mathbb{N}$
only a $2$-point alternance exists at $x_* = 0$ (note that for any $i \in \{ 1, \ldots, d \}$ the collection 
$\{ e_i, -e_i \}$ satisfies the assumptions of Proposition~\ref{prp:AlternanceVsCadre}, i.e.
a $2$-point alternance does exist at $x_*$).

Note, however, that if one modifies the definition of alternance by allowing the vectors $V_1, \ldots, V_{k_0}$ to
belong to the entire subdifferential $\partial F(x_*)$ (see Definition~\ref{def:AlternanceOptCond}), then 
a complete alternance exists at $x_* = 0$ in the problem under consideration. Indeed, defined $V_i = e_i$ for any 
$i \in \{ 1, \ldots, d \}$ and put $V_{d + 1} = (-1/d, \ldots, -1/d)^T \in \partial F(x_*)$. Then, as is easily seen,
$\Delta_i = \determ([V_1, \ldots, V_{i - 1}, V_{i + 1}, \ldots, V_{d + 1}]) = (-1)^{d - i} (- 1/d)$ for any 
$i \in \{ 1, \ldots, d \}$ and $\Delta_{d + 1} = 1$, i.e. conditions \eqref{eq:DeterminantsProp} and
\eqref{eq:DeterminantsProp2} are satisfied. \qed
\end{example}

The example above motivates us to introduce a weakened definition of alternance. 

\begin{definition}
One says that \textit{a generalised $p$-point alternance} exists at $x_*$, if there exist $k_0 \in \{ 1, \ldots, p \}$, 
$i_0 \in \{ k_0 + 1, \ldots, p \}$, vectors
\begin{equation} \label{eq:GenAlternanceDef}
  V_1, \ldots, V_{k_0} \in \partial F(x_*), \quad
  V_{k_0 + 1}, \ldots, V_{i_0} \in \mathcal{N}(x_*), \quad
  V_{i_0 + 1}, \ldots, V_p \in N_A(x_*),
\end{equation}
and vectors $V_{p + 1}, \ldots, V_{d + 1} \in Z$ such that conditions \eqref{eq:DeterminantsProp} and
\eqref{eq:DeterminantsProp2} hold true. Such collection of vectors $\{ V_1, \ldots, V_p \}$ is called a 
\textit{a generalised $p$-point alternance} at $x_*$. Any generalised $(d + 1)$-point alternance is called
\textit{complete}.
\end{definition}

\begin{remark} \label{rmrk:GenAlternanceVsCadre}
Almost literally repeating the proof of Proposition~\ref{prp:AlternanceVsCadre} one obtains that a generalised $p$-point
alternance with  $p \in \{ 1, \ldots, d + 1 \}$ exists at $x_*$ iff there exist $k_0 \in \{ 1, \ldots, p \}$, 
$i_0 \in \{ k_0 + 1, \ldots, p \}$, and vectors $V_1, \ldots, V_p$ satisfying \eqref{eq:GenAlternanceDef} such that
$\rank([V_1, \ldots, V_p]) = p - 1$ and $\sum_{i = 1}^p \beta_i V_i = 0$ for some $\beta_i > 0$, 
$i \in \{ 1, \ldots, p \}$. \qed
\end{remark}

Clearly, any $p$-point alternance is a generalised $p$-point alternance as well. Therefore by
Theorem~\ref{thrm:AlternanceCond} the existence of a generalised $p$-point alternance is a necessary optimality
condition for the problem $(\mathcal{P})$ that is equivalent to the existence of a Lagrange multiplier (the fact that
the existence of a generalised $p$-point alternance implies the inclusion $0 \in \mathcal{D}(x_*)$ is proved in exactly
the same
way as the analogous statement for non-generalised $p$-point alternance). 

In the general case the existence of a generalised complete alternance is not equivalent to the sufficient optimality
condition $0 \in \interior \mathcal{D}(x_*)$ (see Example~\ref{exmpl:ConstrComplAltern_CounterEx} in the following
section); however, under some additional assumptions one can prove that these conditions are indeed equivalent. To prove
this result we will need the following characterisation of relative interior points of a convex cone, which can be
viewed as an extension of a similar result for polytopes \cite[Lemma~2.9]{Zeigler} to the case of cones. Recall that
\textit{the dimension} of a convex cone $\mathcal{K} \subset \mathbb{R}^d$, denoted $\dimens \mathcal{K}$, is the
dimension of its affine hull, which obviously coincides with the linear span of $\mathcal{K}$.

\begin{lemma} \label{lem:RelIntConvexCone}
Let $\mathcal{K} \subset \mathbb{R}^d$ be a convex cone such that $k = \dimens \mathcal{K} \ge 1$. Then a point 
$x \ne 0$ belongs to the relative interior $\relint \mathcal{K}$ of the cone $\mathcal{K}$ iff $x$ can be expressed as 
$x = \sum_{i = 1}^k \beta_i x_i$ for some $\beta_i > 0$ and linearly independent vectors 
$x_1, \ldots, x_k \in \mathcal{K}$.
\end{lemma}

\begin{proof}
Let $x \in \relint \mathcal{K}$ and $x \ne 0$. If $k = 1$, then put $x_1 = x$ and $\beta_1 = 1$. Otherwise, denote 
$X_0 = \linhull \mathcal{K}$, and let $E_0 = \{ z \in X_0 \mid \langle z, x \rangle = 0 \}$ be the orthogonal complement
of $\linhull\{ x \}$ in $X_0$. As is well known, $\dimens E_0 = k - 1 \ge 1$. Let $z_1, \ldots, z_{k - 1} \in E_0$ be
any basis of $E_0$, and define $z_k = - \sum_{i = 1}^{k - 1} z_i$. 

By the definition of relative interior $B(x, r) \cap X_0 \subset \mathcal{K}$ for some $r > 0$, where, as above, 
$B(x, r) = \{ z \in \mathbb{R}^d \mid |z - x| \le r \}$. Let 
$\delta = \max\{ |z_1|, \ldots, |z_k| \}$ and $\gamma = r / \delta$. Then for all $i \in \{ 1, \ldots, k \}$ one has
$x_i = \gamma z_i + x \in B(x, r) \cap X_0 \subset \mathcal{K}$. Furthermore,
observe that $x = \sum_{i = 1}^k (1/k) x_i$. Therefore, it remains to show that the vectors $x_1, \ldots, x_k$ are
linearly independent.

Indeed, suppose that $\sum_{i = 1}^k \alpha_i x_i = 0$ for some $\alpha_i \in \mathbb{R}$. Then by definition
$$
  \sum_{i = 1}^k \alpha_i \gamma z_i = - \Big( \sum_{i = 1}^k \alpha_i \Big) x.
$$
Recall that $z_i$ belong to the orthogonal complement of $x$, i.e. $\langle z_i, x \rangle = 0$. Therefore
$\sum_{i = 1}^k \alpha_i = 0$. Hence taking into account the fact that $z_k = - \sum_{i = 1}^{k - 1} z_i$ one obtains
that $\sum_{i = 1}^{k - 1} (\alpha_i - \alpha_k) z_i = 0$, which implies that $\alpha_i = \alpha_k$ for all
$i \in \{ 1, \ldots, k - 1 \}$, since the vectors $z_1, \ldots, z_{k - 1}$ form a basis of $E_0$. Thus,
$\sum_{i = 1}^k \alpha_i = k \alpha_k = 0$, i.e. $\alpha_i = 0$ for all $i$, and one can conclude that the vectors
$x_1, \ldots, x_k$ are linearly independent.

Let us prove the converse statement. Suppose that $x = \sum_{i = 1}^k \beta_i x_i$ for some $\beta_i > 0$ and linearly
independent vectors $x_1, \ldots, x_k \in \mathcal{K}$. Denote $S(x) = \co\{ x_1, \ldots, x_k, - x \}$. Let us show that
there exists $r > 0$ such that $B(0, r) \cap X_0 \subset S(x)$, where, as above, $X_0 = \linhull \mathcal{K}$. Then
taking into account the fact that $\mathcal{K}$ is a convex cone one obtains that
\begin{equation} \label{eq:BallWithinConeSubSimples}
  \big( B(x, r) \cap X_0 \big) \subset x + S(x) = \co\{ x_1 + x, \ldots, x_k + x, 0 \} \subset \mathcal{K},
\end{equation}
and the proof is complete.

Since $k = \dimens \mathcal{K}$, the collection $x_1, \ldots, x_k \in \mathcal{K}$ is a basis of the subspace 
$X_0 = \linhull \mathcal{K}$. Therefore, for any $z \in X_0$ there exist unique $\alpha_i$ such that 
$z = \sum_{i = 1}^k \alpha_i x_i$. For any $z \in X_0$ denote $\| z \|_{X_0} = \sum_{i = 1}^k |\alpha_i|$. One can
readily check that $\| \cdot \|_{X_0}$ is a norm on $X_0$. 

With the use of Lemma~\ref{lem:SimplexZeroInterior} one obtains that 
$\{ z \in X_0 \mid \| z \|_{X_0} < r \} \subset S(x)$ for some $r > 0$. Taking into account the fact that all norms on a
finite dimensional space are equivalent one gets that there exists $C > 0$ such that $\| z \|_{X_0} \le C |z|$ for all 
$z \in X_0$. Therefore the inclusions 
$(B(0, r / 2C) \cap X_0) \subset \{ z \in X_0 \mid \| z \|_{X_0} < r \} \subset S(x)$ hold true,
and the proof is complete.
\end{proof}

Recall that a convex cone $\mathcal{K} \subset \mathbb{R}^d$ is called \textit{pointed}, if
$\mathcal{K} \cap (-\mathcal{K}) = \{ 0 \}$.

\begin{theorem}	\label{thrm:GenAlternanceCond}
Let $x_*$ be a feasible point of $(\mathcal{P})$. Then the existence of a generalised complete alternance at
$x_*$ implies that $0 \in \interior \mathcal{D}(x_*)$ and $\partial F(x_*) \ne \{ 0 \}$. Conversely, if 
$0 \in \interior \mathcal{D}(x_*)$, $\partial F(x_*) \ne \{ 0 \}$, and one of the following assumptions is valid:
\begin{enumerate}
\item{$\interior \partial F(x_*) \ne \emptyset$,
}

\item{$\mathcal{N}(x_*) + N_A(x_*) \ne \mathbb{R}^d$ and either $\interior \mathcal{N}(x_*) \ne \emptyset$ 
or $\interior N_A(x_*) \ne \emptyset$,
}

\item{$N_A(x_*) = \{ 0 \}$ and there exists $w \in \relint \mathcal{N}(x_*) \setminus \{ 0 \}$ such that 
$0 \in \partial F(x_*) + w$ (in particular, it is sufficient to suppose that $0 \notin \partial F(x_*)$ or the cone
$\mathcal{N}(x_*)$ is pointed),
}

\item{$\mathcal{N}(x_*) = \{ 0 \}$ and there exists $w \in \relint N_A(x_*) \setminus \{ 0 \}$ such that 
$0 \in \partial F(x_*) + w$,
}
\end{enumerate}
then a generalised complete alternance exists at $x_*$.
\end{theorem}

\begin{proof}
If a generalised complete alternance exists at $x_*$, then literally repeating the proof of the second part of
Theorem~\ref{thrm:AlternanceCond} one obtains that $0 \in \interior \mathcal{D}(x_*)$ and $\partial F(x_*) \ne \{ 0 \}$.
Let us prove the converse statement. Consider four cases corresponding to four assumptions of the theorem.

\textbf{Case I.} Suppose that $\interior \partial F(x_*) \ne \emptyset$. If $0 \in \interior \partial F(x_*)$, then
one can find $r > 0$ such that $r e_1, \ldots, r e_d \in \partial F(x_*)$ and 
$\overline{e} = (-r, \ldots, - r)^T \in \partial F(x_*)$. Note that $\rank([r e_1, \ldots, r e_d, \overline{e}]) = d$
and $\sum_{i = 1}^d r e_i + \overline{e} = 0$. Hence by Remark~\ref{rmrk:GenAlternanceVsCadre} a generalised
complete alternance exists at $x_*$.

Thus, one can suppose that $0 \notin \interior \partial F(x_*)$. Let there exists $w \in \mathcal{N}(x_*) \cup N_A(x_*)$
such that $0 \in \interior \partial F(x_*) + w$. Clearly, $w \ne 0$ and $-w \in \interior \partial F(x_*)$. If $d = 1$,
then define $V_1 = -w$, $V_2 = w$. Then $\rank([V_1, V_2]) = 1$ and $V_1 + V_2 = 0$, which due to
Remark~\ref{rmrk:GenAlternanceVsCadre} implies that a generalised complete alternance exists at $x_*$. If $d \ge 2$,
then denote by $X_0$ the orthogonal complement of the subspace $\linhull\{ w \}$. Obviously, $\dimens X_0 = d - 1$. 
Let $z_1, \ldots, z_{d - 1}$ be a basis of $X_0$, and $z_d = - \sum_{i = 1}^{d - 1} z_i$. 

Since $-w \in \interior \partial F(x_*)$, there exists $r > 0$ such that $V_i = -w + r z_i \in \partial F(x_*)$ for all
$i \in \{ 1, \ldots, d \}$. Denote $V_{d + 1} = w$. Observe that $\sum_{i = 1}^{d} (1/d) V_i + V_{d + 1} = 0$.
Furthermore, the vectors $V_1, \ldots, V_{d - 1}, V_{d + 1}$ are linearly independent. Indeed, suppose that
$\sum_{i = 1}^{d - 1} \alpha_i V_i + \alpha_{d + 1} V_{d + 1} = 0$ for some $\alpha_i \in \mathbb{R}$. Then
$$
  r \sum_{i = 1}^{d - 1} \alpha_i z_i = \Big( \sum_{i = 1}^{d - 1} \alpha_i - \alpha_{d + 1} \Big) w.
$$
Bearing in mind the fact that $z_1, \ldots, z_{d - 1}$ is a basis of the orthogonal complement of $\linhull\{ w \}$ one
obtains that $\alpha_{d + 1} = \sum_{i = 1}^{d - 1} \alpha_i$ and $\alpha_i = 0$ for all 
$i \in \{ 1, \ldots, d - 1 \}$, which implies that the vectors $V_1, \ldots, V_{d - 1}, V_{d + 1}$ are linearly
independent. Consequently, $\rank([V_1, \ldots, V_{d + 1}]) = d$ and by Remark~\ref{rmrk:GenAlternanceVsCadre} 
a generalised complete alternance exists at $x_*$.

Thus, one can suppose that 
\begin{equation} \label{eq:SubdiffShiftViaCones}
  0 \notin \interior \partial F(x_*) + w	\quad	\forall w \in \mathcal{N}(x_*) \cup N_A(x_*).
\end{equation}
Note that $0 \in \interior \partial F(x_*) + w$ for some $w \in \mathcal{N}(x_*) + N_A(x_*)$. Indeed, arguing by
reductio ad absurdum, suppose that $(- \interior \partial F(x_*)) \cap (\mathcal{N}(x_*) + N_A(x_*)) = \emptyset$. Then
by the separation theorem (see, e.g. \cite[Thrm.~2.13]{BonnansShapiro}) there exists 
$h \ne 0$ such that $\langle h, v \rangle \le \langle h, w \rangle$ for all $v \in - \partial F(x_*)$ and 
$w \in \mathcal{N}(x_*) + N_A(x_*)$. Hence $\langle h, v \rangle \ge 0$ for all 
$v \in \partial F(x_*) + \mathcal{N}(x_*) + N_A(x_*) = \mathcal{D}(x_*)$, which contradicts the assumption that 
that $0 \in \interior \mathcal{D}(x_*)$.

By definition $w = w_1 + w_2$ for some $w_1 \in \mathcal{N}(x_*)$ and $w_2 \in N_A(x_*)$. The vectors $w_1$ and $w_2$
are linearly independent. Indeed, if $w_1 = \alpha w_2$ for some $\alpha \ge 0$, then 
$w = (1 + \alpha) w_2 \in N_A(x_*)$, since $N_A(x_*)$ is a cone, which contradicts
\eqref{eq:SubdiffShiftViaCones}. Similarly, if $w_1 = - \alpha w_2$ for some $\alpha > 0$, then 
$w = (1 - \alpha) w_2 \in N_A(x_*)$ in the case $\alpha \in (0, 1]$, and $w = (1 - 1/\alpha) w_1 \in \mathcal{N}(x_*)$
in the case $\alpha > 1$, which once again contradicts \eqref{eq:SubdiffShiftViaCones}. Thus, $w_1$ and $w_2$ are
linearly independent and $d \ge 2$.

If $d = 2$, denote $V_1 = -w \in \partial F(x_*)$, $V_2 = w_1$, and $V_3 = w_2$. Then $V_1 + V_2 + V_3 = 0$ and
$\rank([V_1, V_2, V_3]) = 2$, which implies that a generalised complete alternance exists at $x_*$ due to
Remark~\ref{rmrk:GenAlternanceVsCadre}. If $d \ge 3$, then denote by $X_0$ the orthogonal complement of 
$\linhull\{ w_1, w_2 \}$. Clearly, $\dimens X_0 = d - 2$. Let $z_1, \ldots, z_{d - 2}$ be a basis of $X_0$ and 
$z_{d - 1} = - \sum_{i = 1}^{d - 2} z_i$.

Since $- w \in \interior \partial F(x_*)$, there exists $r > 0$ such that $V_i = -w + r z_i \in \partial F(x_*)$ for all
$i \in \{ 1, \ldots, d - 1 \}$. Denote $V_d = w_1$ and 
$V_{d + 1} = w_2$. Then $\sum_{i = 1}^{d - 1} (1/(d - 1)) V_i + V_d + V_{d + 1} = 0$. Moreover, the vectors
$V_1, \ldots, V_{d- 2}, V_d, V_{d + 1}$ are linearly independent. Indeed, if for some $\alpha_i \in \mathbb{R}$ one has
$\sum_{i = 1}^{d - 2} \alpha_i V_i + \alpha_d V_d + \alpha_{d + 1} V_{d + 1} = 0$, then
$$
  r \sum_{i = 1}^{d - 2} \alpha_i z_i = \Big( \sum_{i = 1}^{d - 2} \alpha_i - \alpha_d \Big) w_1
  + \Big( \sum_{i = 1}^{d - 2} \alpha_i - \alpha_{d + 1} \Big) w_2.
$$
Taking into account the facts that $z_1, \ldots, z_{d - 2}$ is a basis of the orthogonal complement of 
$\linhull\{ w_1, w_2 \}$ and the vectors $w_1$ and $w_2$ are linearly independent one can easily check that
$\alpha_i = 0$ for any $i \in \{ 1, \ldots, d - 2, d, d + 1 \}$. Thus, the vectors 
$V_1, \ldots, V_{d- 2}, V_d, V_{d + 1}$ are linearly independent, which by Remark~\ref{rmrk:GenAlternanceVsCadre}
implies that a generalised complete alternance exists at $x_*$.

\textbf{Case II.} Let $\mathcal{N}(x_*) + N_A(x_*) \ne \mathbb{R}^d$ and $\interior \mathcal{N}(x_*) \ne \emptyset$ (the
case when $\interior N_A(x_*) \ne \emptyset$ is proved in the same way). Suppose that there exists 
$w \in \partial F(x_*)$ such that $-w \in \interior \mathcal{N}(x_*)$. Let us show that one can assume that $w \ne 0$.
Indeed, if $w = 0$, then $0 \in \interior \mathcal{N}(x_*)$. Recall that by our assumption 
$\partial F(x_*) \ne \{ 0 \}$. Choose any $v \in \partial F(x_*) \setminus \{ 0 \}$. 
Since $0 \in \interior \mathcal{N}(x_*)$, there exists $\alpha \in (0, 1]$ such that
$\alpha v \in \interior \mathcal{N}(x_*)$ and $\alpha v \in \co\{ 0, v \} \subseteq \partial F(x_*)$. Thus, there
exists $w \in \partial F(x_*) \setminus \{ 0 \}$ such that $-w \in \interior \mathcal{N}(x_*)$.

Denote $V_1 = w$. Since $\interior \mathcal{N}(x_*) \ne \emptyset$, one has $\dimens \mathcal{N}(x_*) = d$. Therefore by
Lemma~\ref{lem:RelIntConvexCone} there exist linearly independent vectors $V_2, \ldots, V_{d + 1} \in \mathcal{N}(x_*)$
such that $V_1 + \sum_{i = 2}^{d + 1} \beta_i V_i = 0$ for some $\beta_i > 0$, $i \in \{2, \ldots, d + 1 \}$. Thus, 
$\rank([V_1, \ldots, V_{d + 1}]) = d$, which by Remark~\ref{rmrk:GenAlternanceVsCadre} implies that a generalised
complete alternance exists at $x_*$.

Suppose now that 
\begin{equation} \label{eq:SolidCone_Subdiff}
  (- \partial F(x_*)) \cap \interior \mathcal{N}(x_*) = \emptyset.
\end{equation}
Then there exist $v \in \partial F(x_*)$ and $w \in N_A(x_*)$ such that $- v - w \in \interior \mathcal{N}(x_*)$.
Indeed, otherwise the sets $- (\partial F(x_*) + N_A(x_*))$ and $\interior \mathcal{N}(x_*)$ do not intersect, which by
the separation theorem implies that there exists $h \in \mathbb{R}^d \setminus \{ 0 \}$ such 
that $\langle h, v \rangle \le 0$ for all  $v \in - (\partial F(x_*) + N_A(x_*))$ and $\langle h, w \rangle \ge 0$ for
all $w \in \mathcal{N}(x_*)$. Hence $\langle h, v \rangle \ge 0$ for 
all $v \in \partial F(x_*) + \mathcal{N}(x_*) + N_A(x_*) = \mathcal{D}(x_*)$, which contradicts the
assumption that $0 \in \interior \mathcal{D}(x_*)$.

Thus, there exist $v \in \partial F(x_*)$ and $w \in N_A(x_*)$ such that $- v - w \in \interior \mathcal{N}(x_*)$. Note
that $w \ne 0$ due to \eqref{eq:SolidCone_Subdiff}. Furthermore, one can suppose that the vectors $v$ and $w$ are
linearly independent. Indeed, if $v = \alpha w$ for some $\alpha < - 1$, then one obtains that 
$- \beta v \in \interior \mathcal{N}(x_*)$, where $\beta = 1 + 1 / \alpha \in (0, 1)$. Therefore there exists
$\varepsilon > 0$ such that $- \beta v + B(0, \varepsilon) \subset \mathcal{N}(x_*)$, which implies that 
$- v + B(0, \varepsilon / \beta) \subset \mathcal{N}(x_*)$ due to the fact that $\mathcal{N}(x_*)$ is a cone. 
Thus, $- v \in \interior \mathcal{N}(x_*)$, which contradicts \eqref{eq:SolidCone_Subdiff}.

On the other hand, if $v = \alpha w$ for some $\alpha \ge -1$, then for $z = (1 + \alpha) w \in N_A(x_*)$ one has
$- z \in \interior \mathcal{N}(x_*)$. By definition there exists $\varepsilon > 0$ such that
$-z + B(0, \varepsilon) \subset \mathcal{N}(x_*)$. Consequently, one has 
$B(0, \varepsilon) = -z + B(0, \varepsilon) + z \subset \mathcal{N}(x_*) + N_A(x_*)$. Hence with the use of the fact
that the sets $\mathcal{N}(x_*)$ and $N_A(x_*)$ are cones one obtains that $\mathcal{N}(x_*) + N_A(x_*) = \mathbb{R}^d$,
which contradicts our assumption. Thus, the vectors $v$ and $w$ are linearly independent, which implies that $d \ge 2$.

If $d = 2$, define $V_1 = v \in \partial F(x_*)$, $V_2 = - v - w \in \mathcal{N}(x_*)$, and $V_3 = w \in N_A(x_*)$. Then
$\rank([V_1, V_2, V_3]) = 2$ and $V_1 + V_2 + V_3 = 0$. Therefore by Remark~\ref{rmrk:GenAlternanceVsCadre} a
generalised complete alternance exists at $x_*$. If $d \ge 3$, denote by $X_0$ the orthogonal complement of 
$\linhull\{ v, w \}$. Since $v$ and $w$ are linearly independent, one has $\dimens X_0 = d - 2$. 
Let $z_1, \ldots, z_{d - 2}$ be a basis of $X_0$ and $z_{d - 1} = - \sum_{i = 1}^{d - 2} z_i$.

Since $- v - w \in \interior \mathcal{N}(x_*)$, there exists $r > 0$ such that $-v - w + r z_i \subset \mathcal{N}(x_*)$
for all $i \in \{ 1, \ldots, d - 1 \}$. Denote $V_1 = v$, $V_i = r z_{i - 1} - v - w \in \mathcal{N}(x_*)$ for all 
$i \in \{ 2, \ldots, d \}$, and $V_{d + 1} = w \in N_A(x_*)$. 
Then $V_1 + \sum_{i = 2}^d (1/(d - 1)) V_i + V_{d + 1} = 0$. Let us check that the vectors 
$V_1, \ldots, V_{d - 1}, V_{d + 1}$ are linearly independent. Then $\rank([V_1, \ldots, V_{d + 1}]) = d$ and by
Remark~\ref{rmrk:GenAlternanceVsCadre} one concludes that a generalised complete alternance exists at $x_*$.

Let $\sum_{i = 1}^{d - 1} \alpha_i V_i + \alpha_{d + 1} V_{d + 1} = 0$ for some $\alpha_i \in \mathbb{R}$. Then
$$
  r \sum_{i = 1}^{d - 2} \alpha_{i + 1} z_i = \Big( \sum_{i = 2}^{d - 1} \alpha_i - \alpha_1 \Big) v 
  + \Big( \sum_{i = 2}^{d - 1} \alpha_i - \alpha_{d + 1} \Big) w.
$$
Hence bearing in mind the fact that $z_1, \ldots, z_{d - 2}$ is a basis of the orthogonal complement of 
$\linhull\{ v, w \}$ one obtains that $\alpha_i = 0$ for all $i \in \{ 2, \ldots, d - 1 \}$,
$\alpha_1 = \sum_{i = 2}^{d - 1} \alpha_i = 0$, and $\alpha_{d + 1} = \sum_{i = 2}^{d - 1} \alpha_i = 0$. Thus, the
vectors $V_1, \ldots, V_{d - 1}, V_{d + 1}$ are linearly independent and the proof of Case II is complete.

\textbf{Case III.} Let $N_A(x_*) = \{ 0 \}$ and there exists $w \in \relint \mathcal{N}(x_*) \setminus \{ 0 \}$ such
that $0 \in \partial F(x_*) + w$. Let us check at first that it is sufficient to assume that $N_A(x_*) = \{ 0 \}$ and 
either $0 \notin \partial F(x_*)$ or the cone $\mathcal{N}(x_*)$ is pointed.

Indeed, let $0 \notin \partial F(x_*)$. Let us verify that 
$(-\partial F(x_*)) \cap \relint \mathcal{N}(x_*) \ne \emptyset$. Then taking into account the fact that 
$0 \notin \partial F(x_*)$ one obtains that there exists $w \in \relint \mathcal{N}(x_*) \setminus \{ 0 \}$ such that 
$0 \in \partial F(x_*) + w$. 

Arguing by reductio ad absurdum, suppose that $(-\partial F(x_*)) \cap \relint \mathcal{N}(x_*) = \emptyset$. Then by
the separation theorem (see, e.g. \cite[Thrm.~11.3]{Rockafellar}) there exists $h \ne 0$ such that
$\langle v, h \rangle \le \langle w, h \rangle$ for all $v \in - \partial F(x_*)$ and $w \in \mathcal{N}(x_*)$. Hence
$\langle h, v \rangle \ge 0$ for all $v \in \partial F(x_*) + \mathcal{N}(x_*) = \mathcal{D}(x_*)$ 
(recall that $N_A(x_*) = \{ 0 \}$), which is impossible, since $0 \in \interior \mathcal{D}(x_*)$.

Let now the cone $\mathcal{N}(x_*)$ be pointed. If $\interior F(x_*) \ne \emptyset$, then a generalised
complete alternance exists at $x_*$ by Case~I. Therefore, we can suppose that $\interior F(x_*) = \emptyset$. 

Arguing by reductio ad absurdum, suppose that $0 \notin \partial F(x_*) + w$ for any 
$w \in \relint \mathcal{N}(x_*) \setminus \{ 0 \}$. As was shown above, 
$(-\partial F(x_*)) \cap \relint \mathcal{N}(x_*) \ne \emptyset$, that is, there exists $w \in \relint \mathcal{N}(x_*)$
such that $0 \in \partial F(x_*) + w$. Consequently, by our assumption $0 \in \relint \mathcal{N}(x_*)$. Hence either
$\mathcal{N}(x_*) = \{ 0 \}$ or $\dimens \mathcal{N}(x_*) \ge 1$. In the former case one has 
$\mathcal{D}(x_*) = \partial F(x_*)$. Therefore $0 \in \interior \partial F(x_*)$, which contradicts our assumption. In
the latter case there exists $z \in \mathcal{N}(x_*) \setminus \{ 0 \}$ and by the definition of relative interior there
exists $r > 0$ such that $\linhull \mathcal{N}(x_*) \cap B(0, r) \subset \mathcal{N}(x_*)$. Consequently, 
$r z / |z| \in \mathcal{N}(x_*)$ and $-r z / |z| \in \mathcal{N}(x_*)$, which contradicts the assumption that the cone
$\mathcal{N}(x_*)$ is pointed.

Let us now turn to the proof of the main statement. Let $w_* \in \relint \mathcal{N}(x_*)$, $w_* \ne 0$, be any
vector such that $0 \in \partial F(x_*) + w_*$. By Lemma~\ref{lem:RelIntConvexCone} there exists 
$k = \dimens \mathcal{N}(x_*)$ linearly independent vectors $w_1, \ldots, w_k \in \mathcal{N}(x_*)$ such that 
$w_* = \sum_{i = 1}^k \beta_i w_i$ for some $\beta_i > 0$. 
Note that $\linhull\{ w_1, \ldots, w_k \} = \linhull \mathcal{N}(x_*)$.

\begin{figure}[t]
\centering
\includegraphics[width=0.6\linewidth]{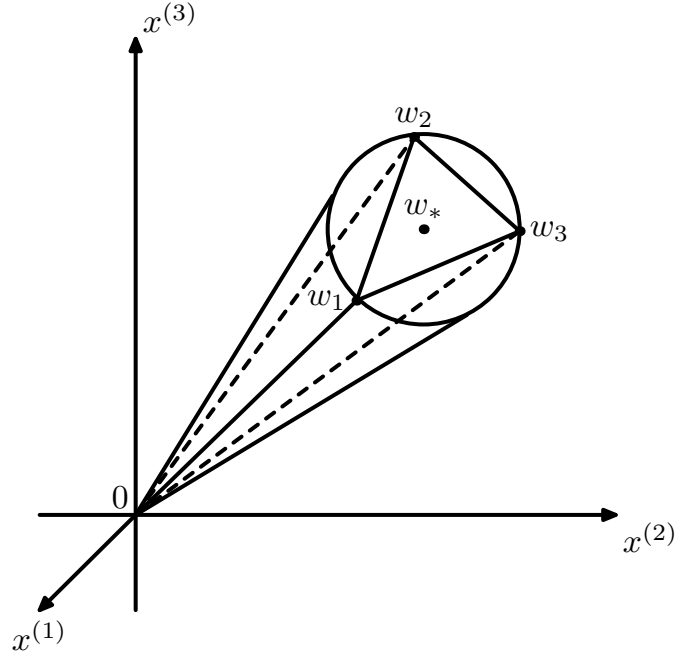}
\caption{In Case III we assume that $0 \in \interior (\partial F(x_*) + \mathcal{N}(x_*))$ and there exists
$w_* \in \relint \mathcal{N}(x_*) \setminus \{ 0 \}$ such that $0 \in \partial F(x_*) + w_*$. The first step of the
proof consists in showing that one can replace the cone $\mathcal{N}(x_*)$ in the condition
$0 \in \interior (\partial F(x_*) + \mathcal{N}(x_*))$ by a polyhedral cone 
$\mathcal{C}_k = \cone\{ w_1, \ldots, w_k \}$ such that $w_* \in \relint \mathcal{C}_k$, where the vectors 
$w_i \in \mathcal{N}(x_*)$ are linearly independent and $k = \dimens \mathcal{N}(x_*)$.}
\label{fig:reduction}
\end{figure}

Denote $\mathcal{C}_k = \cone\{ w_1, \ldots, w_k \}$. Our first goal is to check the validity of the inclusion 
$0 \in \interior(\partial F(x_*) + \mathcal{C}_k)$ (see Fig.~\ref{fig:reduction} below). Indeed, let 
$X_k = \linhull \mathcal{N}(x_*)$. As was shown in the proof of the ``only if'' part of Lemma~\ref{lem:RelIntConvexCone}
(see \eqref{eq:BallWithinConeSubSimples}), 
$X_k \cap B(w_*, r) \subset \co\{ w_1 + w_*, \ldots, w_k + w_*, 0 \} \subset \mathcal{C}_k$ for some $r > 0$, where the
last inclusion follows from the definition of $\mathcal{C}_k$ and the fact that $w_* = \sum_{i = 1}^k \beta_i w_i$.

By our assumptions $0 \in \interior \mathcal{D}(x_*)$ and $\mathcal{D}(x_*) = \partial F(x_*) + \mathcal{N}(x_*)$.
Therefore there exists  $\gamma > 0$ such that for any $i \in \{ 1, \ldots, d + 1 \}$ one can find 
$v_i \in \partial F(x_*)$ and $u_i \in \mathcal{N}(x_*) \subset X_k$ for which $v_i + u_i = \gamma e_i$, 
where $e_1, \ldots, e_d$ is the canonical basis of $\mathbb{R}^d$ and $e_{d + 1} = - \sum_{i = 1}^d e_i$. Clearly, there
exists $\alpha \in (0, 1)$ such that $(1 - \alpha) w_* + \alpha u_i \in X_k \cap B(w_*, r)$ 
for any $i \in \{ 1, \ldots, d + 1 \}$.

Let $v_* \in \partial F(x_*)$ be such that $v_* + w_* = 0$. Then for any $i \in \{ 1, \ldots, d + 1 \}$ one has
\begin{multline*}
  \alpha \gamma e_i = (1 - \alpha) (v_* + w_*) + \alpha (v_i + u_i) 
  = \big( (1 - \alpha) v_* + \alpha v_i \big) + \big( (1 - \alpha) w_* + \alpha u_i \big) \\
  \in \partial F(x_*) + \big( X_k \cap B(w_*, r) \big) \subset \partial F(x_*) + \mathcal{C}_k.
\end{multline*}
Hence taking into account the fact that the set $\partial F(x_*) + \mathcal{C}_k$ is obviously convex one gets that
$\co\{ \alpha \gamma e_1, \ldots, \alpha \gamma e_d, - \alpha \gamma \sum_{i = 1}^d e_i \} 
\subset \partial F(x_*) + \mathcal{C}_k$. Consequently, with the use of Lemma~\ref{lem:SimplexZeroInterior} one obtains
that there exists $r > 0$ such that
\begin{align*}
  B\left( 0, \frac{\alpha \gamma r}{2\sqrt{d}} \right) 
  &\subset \Big\{ x = (x^{(1)}, \ldots, x^{(d)})^T \in \mathbb{R}^d \Bigm| 
  \sum_{i = 1}^d |x^{(i)}| < \alpha \gamma r \Big\} \\
  &\subset \co\Big\{ \alpha \gamma e_1, \ldots, \alpha \gamma e_d, - \alpha \gamma \sum_{i = 1}^d e_i \Big\} 
  \subset \partial F(x_*) + \mathcal{C}_k,
\end{align*}
that is, $0 \in \interior(\partial F(x_*) + \mathcal{C}_k)$.

Now we turn to the proof of the existence of generalised complete alternance. Denote $k_0 = d + 1 - k \ge 1$ and
$V_{k_0 + i} = w_i$ for any $i \in \{ 1, \ldots, k \}$. Observe that 
$$
  \mathbb{R}^d = \linhull\Big( \partial F(x_*) + \mathcal{C}_k \Big) 
  \subseteq \linhull\Big\{ \partial F(x_*), \mathcal{C}_k \Big\} \subseteq \mathbb{R}^d,
$$
where the first equality follows from the fact that and $0 \in \interior (\partial F(x_*) + \mathcal{C}_k)$. Therefore,
there exists vectors $V_2, \ldots, V_{k_0} \in \partial F(x_*)$ such that the vectors $V_2, \ldots, V_{d + 1}$ are
linearly independent. 

Denote $Q(x_*) = \cone\{ V_2, \ldots, V_{d + 1} \}$ (see Fig.~\ref{fig:negative_cone}). Observe that by definition the
affine hull of $Q(x_*)$ coincides with $\mathbb{R}^d$, since $Q(x_*)$ contains $d + 1$ affinely independent vectors: 
$0, V_2, \ldots, V_{d + 1}$. Therefore the relative interior of $Q(x_*)$ coincides with its topological interior, which
implies that $\interior Q(x_*) \ne \emptyset$ due to the fact that the relative interior of a convex subset of a finite
dimensional space is always nonempty. 

\begin{figure}[t]
\centering
\includegraphics[width=0.6\linewidth]{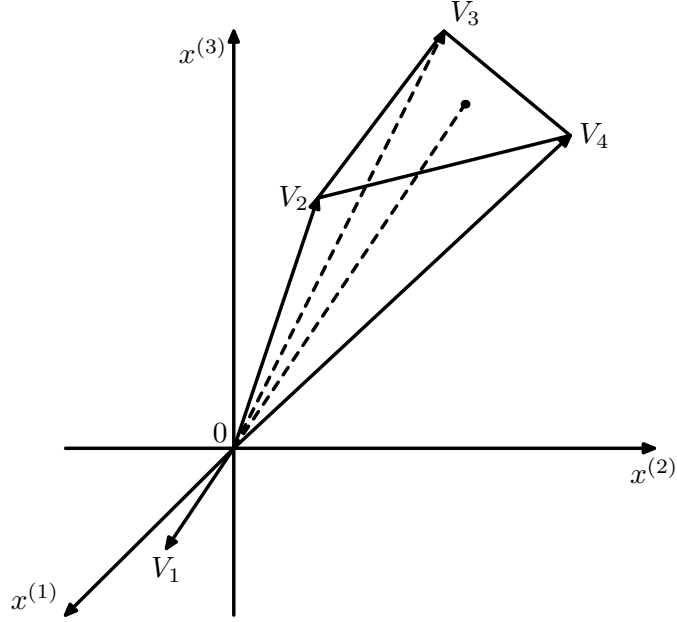}
\caption{As soon as the condition $0 \in \interior (\partial F(x_*) + \mathcal{C}_k)$ has been checked, one can easily
find linearly independent vectors $V_2, \ldots, V_{d + 1} \in \partial F(x_*) \cup \mathcal{C}_k$. The next step is to
prove that there exists $V_1 \in \partial F(x_*)$ such that $V_1 \in - \interior \cone\{ V_2, \ldots, V_{d + 1} \}$.
Then $V_1, \ldots, V_{d + 1}$ is the desired generalised complete alternance. However, to prove the existence of
such $V_1$ one needs to properly choose the cone $\mathcal{C}_k$.}
\label{fig:negative_cone}
\end{figure}

Let us verify that $(- \interior Q(x_*)) \setminus Q(x_*) \ne \emptyset$. Indeed, arguing by reductio ad absurdum
suppose that $- \interior Q(x_*) \subset Q(x_*)$. Choose any $z \in \interior Q(x_*)$. Then 
$z + B(0, \varepsilon) \subset \interior Q(x_*) \subset Q(x_*)$ for some $\varepsilon > 0$. Consequently, one has 
$- z - B(0, \varepsilon) \subset - \interior Q(x_*) \subset Q(x_*)$. Hence taking into account the fact that $Q(x_*)$ is
a convex cone (which implies that $Q(x_*)$ is closed under addition) one obtains that 
$$
  B(0, \varepsilon) \subset \big( z + B(0, \varepsilon) \big) + \big( -z - B(0, \varepsilon) \big) \subset Q(x_*).
$$
Choose any $u \in B(0, \varepsilon)$, $u \ne 0$. Then $u \in Q(x_*)$ and $- u \in Q(x_*)$. By the definition of $Q(x_*)$
one has $u = \sum_{i = 2}^{d + 1} \alpha_i V_i$ for some $\alpha_i \ge 0$ and $- u = \sum_{i = 2}^{d + 1} \beta_i V_i$
for some $\beta_i \ge 0$. Summing up these equalities one obtains $\sum_{i = 2}^{d + 1} (\alpha_i + \beta_i) V_i = 0$,
which implies that $\alpha_i = \beta_i = 0$ for all $i \in \{ 2, \ldots, d + 1 \}$, since the vectors 
$V_2, \ldots, V_{d + 1}$ are linearly independent. Consequently, $u = 0$, which contradicts our assumption that 
$u \ne 0$.

Thus, there exists a nonzero vector $\xi \in (- \interior Q(x_*)) \setminus Q(x_*)$. By definition one can find
$\varepsilon > 0$ such that $- \xi + B(0, \varepsilon) \subset Q(x_*)$. Since $Q(x_*)$ is a cone,  
$- \alpha \xi + B(0, \alpha \varepsilon) \subset Q(x_*)$ for any $\alpha > 0$, that is, 
$\alpha \xi \in - \interior Q(x_*)$. Furthermore, $\alpha \xi \notin Q(x_*)$, since otherwise $\xi \in Q(x_*)$.

Since $0 \in \interior(\partial F(x_*) + \mathcal{C}_k)$, by choosing a sufficiently small $\alpha > 0$ we can
suppose that $\alpha \xi \in \partial F(x_*) + \mathcal{C}_k$. Therefore there exists $V_1 \in \partial F(x_*)$ and 
$u \in \mathcal{C}_k \subset Q(x_*)$ such that $\alpha \xi = V_1 + u$ (the inclusion $\mathcal{C}_k \subset Q(x_*)$
follows from the fact that $\mathcal{C}_k = \cone\{ V_{k_0 + 1}, \ldots, V_{d + 1} \} \subset Q(x_*)$ by definition).
Observe that $V_1 = \alpha \xi - u \in ( - \interior Q(x_*)) - Q(x_*) = - \interior Q(x_*)$, where the last equality
follows from the fact that if $z_1 \in \interior Q(x_*)$ and $z_2 \in Q(x_*)$, then for some $\varepsilon > 0$ one has
$z_1 + B(0, \varepsilon) \subset Q(x_*)$, which implies that $z_1 + B(0, \varepsilon) + z_2 \subset Q(x_*)$, i.e.
$z_1 + z_2 \in \interior Q(x_*)$.

Note that if a vector $v \in Q(x_*)$ can be represented as a linear combination with positive coefficients of $d - 1$
or fewer vectors from the set $V_2, \ldots, V_{d + 1}$, then $v \notin \interior Q(x_*)$. Indeed, let 
$v \in Q(x_*) = \cone\{ V_2, \ldots, V_{d + 1} \}$ have the form
$$
  v = \beta_2 V_2 + \ldots + \beta_{i - 1} V_{i - 1} + \beta_{i + 1} V_{i + 1} + \ldots + 
  \beta_{d + 1} V_{d + 1},
$$
for some $\beta_j \ge 0$ and $i \in \{ 2, \ldots, d + 1 \}$. For any $\varepsilon > 0$ define 
$v_{\varepsilon} = v - \varepsilon V_i$. Observe that $v_{\varepsilon} \notin Q(x_*)$, since otherwise by the
definition of $Q(x_*)$ one could find $\gamma_j \ge 0$, $j \in \{ 2, \ldots, d + 1 \}$, such that
$$
  \sum_{j = 2}^{i - 1} \beta_j V_j + \sum_{j = i + 1}^{d + 1} \beta_j V_j 
  - \varepsilon V_i = \sum_{j = 2}^{d + 1} \gamma_j V_j,
$$
which contradicts the fact that the vectors $V_2, \ldots, V_{d + 1}$ are linearly independent. On the other hand, note
that choosing $\varepsilon > 0$ sufficiently small one can ensure that $v_{\varepsilon}$ belongs to an arbitrarily
small neighbourhood of $v$, which implies that $v \notin \interior Q(x_*)$. Thus, the vector $-V_1 \in \interior Q(x_*)$
can only be represented in the form $- V_1 = \sum_{i = 2}^{d + 1} \beta_i V_i$ for some $\beta_i > 0$,
$i \in \{ 2, \ldots, d + 1 \}$. Looking at this representation as a system of linear equations with respect to
$\beta_i$ and applying Cramer's rule one obtains that $\Delta_1 \ne 0$ and 
$\beta_i = (-1)^{i - 1} \Delta_i / \Delta_1 > 0$ for any $i \in \{ 2, \ldots, d + 1 \}$, where, as in the definition of
alternance, $\Delta_i = \determ([V_1, \ldots, V_{i - 1}, V_{i + 1}, \ldots, V_{d + 1}])$.
Therefore, all determinants $\Delta_s$ are nonzero, and $\sign \Delta_s = - \sign \Delta_{s + 1}$
for all $s \in \{ 1, \ldots, d \}$, that is, a generalised complete alternance exists at $x_*$.

\textbf{Case IV.} The proof of this case repeats the proof of the previous one with $\mathcal{N}(x_*)$ replaced by
$N_A(x_*)$.
\end{proof}

\begin{remark} \label{rmrk:IsolatePoint_WeakAlternance}
{(i)~Note that the condition $\mathcal{N}(x_*) + N_A(x_*) \ne \mathbb{R}^d$ in the second assumption of the theorem
above simply means that $x_*$ is not an isolated point of the feasible region $\Omega$ of the problem $(\mathcal{P})$.
Indeed, fix any $v_1 \in \mathcal{N}(x_*)$ and $v_2 \in N_A(x_*)$. One can easily verify that, regardless of whether RCQ
holds true or not, one has $T_{\Omega}(x_*) \subseteq \{ h \in T_A(x_*) \mid D G(x_*) h \in T_K(G(x)) \}$, which by
Lemma~\ref{lem:NormalCone_ConeConstr} implies that $\langle v_1, h \rangle \le 0$ and $\langle v_2, h \rangle \le 0$ for
any $h \in T_{\Omega}(x_*)$. Therefore $\mathcal{N}(x_*) + N_A(x_*) \subset (T_{\Omega}(x_*))^* = N_{\Omega}(x_*)$.
Thus, if $\mathcal{N}(x_*) + N_A(x_*) = \mathbb{R}^d$, then $N_{\Omega}(x_*) = \mathbb{R}^d$, which with the use of
\cite[Prp.~2.40]{BonnansShapiro} implies that 
$\cl \cone(T_{\Omega}(x_*)) = T_{\Omega}(x_*)^{**} = N_{\Omega}(x_*)^* = \{ 0 \}$.
On the other hand, if $x_*$ is a non-isolated point of $\Omega$, then there exists a sequence 
$x_n \subset \Omega \setminus \{ x_* \}$ converging to $x_*$. Replacing $\{ x_n \}$, if necessary, with its subsequence
one can suppose that the sequence $\{ (x_n - x_*) / |x_n - x_*| \}$ converges to some $v \ne 0$, which obviously
belongs to $T_{\Omega}(x_*)$. Thus, one can conclude that the condition $\mathcal{N}(x_*) + N_A(x_*) = \mathbb{R}^d$
implies that $x_*$ is an isolated point of $\Omega$.
}

\noindent{(ii)~Let us note that by further weakening the definition of generalised alternance one can obtain sufficient
optimality conditions for the problem $(\mathcal{P})$ in an alternance form that are equivalent to the condition 
$0 \in \interior \mathcal{D}(x_*)$ under less restrictive assumptions. Namely, one says that \textit{a weak $p$-point
alternance} exists at $x_*$, if there exist $k_0 \in \{ 1, \ldots, p \}$, vectors 
$V_1, \ldots, V_{k_0} \in \partial F(x_*)$, $V_{k_0 + 1}, \ldots, V_p \in \mathcal{N}(x_*) + N_A(x_*)$, and 
$V_{p + 1}, \ldots, V_{d + 1} \in Z$ such that conditions \eqref{eq:DeterminantsProp} and \eqref{eq:DeterminantsProp2}
hold true. Almost literally repeating the proof of the third case of the previous theorem with $\mathcal{N}(x_*)$
replaced by $\mathcal{N}(x_*) + N_A(x_*)$ one can prove that $0 \in \interior \mathcal{D}(x_*)$ and 
$\partial F(x_*) \ne \{ 0 \}$, provided a weak complete alternance exists at $x_*$ and 
$0 \in \partial F(x_*) + w$ for some $w \in \relint(\mathcal{N}(x_*) + N_A(x_*)) \setminus \{ 0 \}$ (in particular,
it is sufficient to assume that the necessary condition for an unconstrained local minimum $0 \in \partial F(x_*)$
is not satisfied at $x_*$). However, to obtain alternance conditions that are equivalent to the conditions 
$0 \in \interior \mathcal{D}(x_*)$ and $\partial F(x_*) \ne \{ 0 \}$, in the general case one must assume that 
$V_1, \ldots, V_p \in \mathcal{D}(x_*)$. Indeed, let $d = 2$ and consider the following minimax problem:
$$
  \min \: F(x) = \max\{ \pm x^{(1)} \} \quad
  \text{s.t.} \quad x \in A = \{ x = (x^{(1)}, x^{(2)})^T \in \mathbb{R}^2 \mid x^{(2)} = 0 \}.
$$
The point $x_* = 0$ is a globally optimal solution of this problem. Note that 
$\partial F(x_*) = \co\{ (\pm 1, 0)^T \}$ and $N_A(x_*) = \{ x \in \mathbb{R}^2 \mid x^{(1)} = 0 \}$, which implies that
$\mathcal{D}(x_*) = \{ x \in \mathbb{R}^2 \mid |x^{(1)}| \le 1 \}$ and $0 \in \interior \mathcal{D}(x_*)$. However, as
is easily seen, a weak complete alternance does not exist at $x_*$ (only a $2$-point alternance exists at this point).
Note that in this example $(- \partial F(x_*)) \cap \relint N_A(x_*) = \{ 0 \}$. \qed
}
\end{remark}

Let us comment on the number $p$ in the definition of alternance (or cadre). Suppose for the sake of simplicity that
there are no constraints. From the proofs of Proposition~\ref{prp:AlternanceVsCadre} and
Theorem~\ref{thrm:AlternanceCond} it follows that a $p$-point alternance exists at $x_*$ for some 
$p \in \{ 1, \ldots, d + 1 \}$ iff zero can be represented as a convex combination with nonzero coefficients of $p$
affinely independent points from the set $\{ \nabla_x f(x_*, \omega) \mid \omega \in W(x_*) \}$. Hence, in particular,
for a $p$-point alternance to exist at $x_*$ it is necessary that the cardinality of $W(x_*)$ is at least $p$ (i.e. the
maximum in the definition of $F(x_*) = \max_{\omega \in W} f(x_*, \omega)$ must be attained in at least $p$ points
$\omega$) and the set $\{ \nabla_x f(x_*, \omega) \mid \omega \in W(x_*) \}$ contains $p$ affinely independent vectors.
Thus, roughly speaking, the number $p$ in the definition of alternance (or cadre) reflects the size of the
subdifferential $\partial F(x_*)$ at a given point $x_*$ and usually corresponds to its affine dimension plus one. In
particular, in the smooth case (i.e. when $F$ is differentiable at $x_*$) only a $1$-point alternance can exist at
$x_*$. If $\partial F(x_*)$ is a line segment, then only $1$-point or $2$-point alternance can exists at $x_*$, etc. In
the constrained case, the number $p$, roughly speaking, reflects the dimension of the subdifferential $\partial F(x_*)$
and the number of active constraints at $x_*$. However, one must underline that, as
Example~\ref{exmpl:ComplAtern_CounterExampl} demonstrates, in some cases $p$ can be much smaller that the dimension of
the subdifferential.

\begin{remark}
It should be noted that in the proofs of Theorems~\ref{thrm:AlternanceCond} and \ref{thrm:GenAlternanceCond} we do not
use any particular structure of the sets $\partial F(x_*)$, $\mathcal{N}(x_*)$, and $N_A(x_*)$. Therefore, these
theorems can be restated in an abstract form. Namely, suppose that a compact convex set $P \subset \mathbb{R}^d$ and
closed convex cones $K_1, K_2 \subset \mathbb{R}^d$ are given, and let $P = \co P^0$, $K_1 = \cone K_1^0$, and 
$K_2 = \cone K_2^0$ for some sets $P^0 \subseteq P$, $K_1^0 \subseteq K_1$, and $K_2^0 \subseteq K_2$. Then, for
instance, the first part of Theorem~\ref{thrm:AlternanceCond} can be reformulated as follows: $0 \in P + K_1 + K_2$ iff
there exists $p \in \{ 1, \ldots, d + 1 \}$, $k_0 \in \{ 1, \ldots, p \}$, $i_0 \in \{ k_0 + 1, \ldots, p \}$, and
vectors
$$
  V_1, \ldots, V_{k_0} \in P^0, \quad 
  V_{k_0 + 1}, \ldots, V_{i_0} \in K_1^0, \quad
  V_{i_0 + 1}, \ldots, V_p \in K_2^0
$$
such that $\rank([V_1, \ldots, V_p]) = p - 1$ and $\sum_{i = 1}^p \beta_i V_i = 0$ for some $\beta_i > 0$. Such approach
to an analysis of the condition $0 \in P$, where $P$ is a polytope, was studied in detailed by Demyanov and
Malozemov \cite{DemyanovMalozemov_Alternance,DemyanovMalozemov_Collect}. These papers, in particular, describe a
different (but equivalent) approach to the definition of alternance optimality conditions, in which instead of adding
vectors $V_{p + 1}, \ldots, V_{d + 1} \in Z$ one considers submatrices of order $p$ of the matrix $[V_1, \ldots, V_p]$.
\qed
\end{remark}

\subsection{Examples}
\label{subsect:Examples}

In this section we apply the general theory of first order optimality conditions for cone constrained minimax problems
developed in the previous sections to four particular types of such problems: problems with equality and inequality
constraints, problems with second order cone constraints, as well as problems with semidefinite and semi-infinite
constraints. We demonstrate how general conditions can be reformulated in a more convenient way for these problems and
present several examples illustrating theoretical results.

\subsubsection{Constrained minimax problems}

Let the problem $(\mathcal{P})$ be a constrained minimax problem of the form:
\begin{equation} \label{probl:ConstrainedMinimax}
  \min_x \max_{\omega \in W} f(x, \omega) \quad 
  \text{s.t.} \quad g_i(x) \le 0, \quad i \in I, \quad g_j(x) = 0, \quad j \in J, \quad x \in A,
\end{equation}
where $g_i \colon \mathbb{R}^d \to \mathbb{R}$, $i \in I \cup J$, $I = \{ 1, \ldots, l \}$, and 
$J = \{ l + 1, \ldots, l + s \}$. In this case, $Y = \mathbb{R}^{l + s}$, 
$G(\cdot) = (g_1(\cdot), \ldots, g_{l + s}(\cdot))$, and $K =  (- \mathbb{R}_+)^l \times 0_s$, where 
$\mathbb{R}_+ = [0, + \infty)$ and $0_s$ is the zero vector from $\mathbb{R}^s$. Then one has
$K^* = \mathbb{R}_+^l \times \mathbb{R}^s$ and $L(x, \lambda) = F(x) + \sum_{i = 1}^{l + s} \lambda_i g_i(x)$.
Furthermore, as is easily seen, in the case $A = \mathbb{R}^d$, RCQ for problem \eqref{probl:ConstrainedMinimax}
coincides with the well-known Mangasarian-Fromovitz constraint qualification.

If we equip the space $Y$ with the $\ell_1$-norm, then the penalty function for problem \eqref{probl:ConstrainedMinimax}
takes the form
$$
  \Phi_c(x) = \max_{\omega \in W} f(x, \omega) + c \sum_{i = 1}^l \max\{ 0, g_i(x) \}
  + c \sum_{j = l + 1}^{l + s} |g_j(x)|.
$$
Denote $I(x) = \{ i \in I \mid g_i(x) = 0 \}$. As is easy to see, one has 
$$
  \mathcal{N}(x) =
  \Big\{ \sum_{i = 1}^{m + l} \lambda_i \nabla g_i(x) \Bigm| \lambda_i \ge 0, \: \lambda_i g_i(x) = 0 \enspace 
  \forall i \in I, \enspace \lambda_j \in \mathbb{R} \enspace \forall j \in J \Big\}
$$
Therefore, it is natural to choose 
$$
  \eta(x) = \big\{ \nabla g_i(x) \bigm| i \in I(x) \big\} \cup
  \big\{ \nabla g_j(x), - \nabla g_j(x) \bigm| j \in J \big\},
$$
since this is the smallest set whose conic hull coincides with $\mathcal{N}(x)$.

Let us give several particular examples in which we demonstrate how one can verify the validity of optimality
conditions derived in the previous sections in the case of minimax problems with equality and inequality constraints. We
pay special attention to alternance optimality conditions, since these conditions along with optimality conditions in
terms of cadres are the most convenient for analytical computations and can be used to develop efficient numerical
methods (cf. \cite{ConnLi92}). To get the flavour of alternance conditions, we start with a simple nonlinear programming
problem.

\begin{example}[\cite{Bazaraa}, Exercise~4.5]
Consider the following problem:
\begin{equation} \label{probl:MathProg2}
  \begin{split}
  {}&\min \: 
  f(x) = (x^{(1)})^4 + (x^{(2)})^4 + 12 (x^{(1)})^2 + 6 (x^{(2)})^2 - x^{(1)} x^{(2)} - x^{(1)} -x^{(2)}  \\
  {}&\text{s.t.} \enspace x^{(1)} + x^{(2)} \ge 6, \quad 2 x^{(1)} - x^{(2)} \ge 3, 
  \quad x^{(1)} \ge 0, \quad x^{(2)} \ge 0.
  \end{split}
\end{equation}
Define $d = 2$, $l = 2$, $J = \emptyset$, and $A = \{ x \in \mathbb{R}^2 \mid x^{(1)} \ge 0, \: x^{(2)} \ge 0 \}$. Put
also $g_1(x) = -x^{(1)} - x^{(2)} + 6$ and $g_2(x) = - 2 x^{(1)} + x^{(2)} + 3$.

Let us check that a complete alternance exists at the point $x_* = (3, 3)^T$ given in \cite[Exercise~4.5]{Bazaraa}.
Indeed, observe that $I(x_*) = I = \{ 1, 2 \}$ and $N_A(x_*) = - A$. Denote $V_1 = \nabla f(x_*) = (176, 140)^T$, 
$V_2 = \nabla g_1(x_*) = (-1, -1)^T \in \eta(x_*)$, and 
$V_3 = \nabla g_2(x_*) = (-2, 1)^T \in \eta(x_*)$. Then one has 
$$
  \Delta_1 = \begin{vmatrix} -1 & -2 \\ -1 & 1 \end{vmatrix} = -3, \quad
  \Delta_2 = \begin{vmatrix} 176 & -2 \\ 140 & 1 \end{vmatrix} = 456, \quad
  \Delta_3 = \begin{vmatrix} 176 & -1 \\ 140 & -1 \end{vmatrix} = -36,
$$
i.e. a complete alternance exists at $x_*$. Therefore applying
Theorems~\ref{thrm:SuffOptCond}, \ref{thrm:EquivOptCond_Subdiff}, and \ref{thrm:AlternanceCond} one obtains that $x_*$
is a strict local minimiser of problem \eqref{probl:MathProg2} at which the \textit{first} order growth condition
holds true. Note that the classical KKT optimality conditions do not allow one to verify whether the \textit{first}
order growth condition is satisfied at $x_*$. \qed
\end{example}

Let us now give a counterexample to the existence of generalised complete alternance in the general case, promised in
the previous section. In this counterexample, a generalised complete alternance does not exist at a non-isolated point
$x_*$ satisfying the sufficient optimality condition $0 \in \interior \mathcal{D}(x_*)$ and such that 
$0 \notin \partial F(x_*)$.

\begin{example} \label{exmpl:ConstrComplAltern_CounterEx}
Consider the following problem:
\begin{equation} \label{probl:MathProg3_CounterEx}
  \begin{split}
    {}&\min \: f(x) = x^{(1)} + (x^{(2)})^2 + x^{(3)} \\
    {}&\text{s.t.} \enspace x^{(2)} - |x^{(3)}| x^{(3)} \le 0, \enspace - x^{(2)} - |x^{(3)}| x^{(3)} \le 0, 
    \enspace x^{(1)} = 0, \enspace x^{(3)} \ge 0.
  \end{split}
\end{equation}
The feasible region of this problem is depicted in Figure~\ref{fig:feasible_region}. Put $d = 3$, $l = 2$, 
$J = \emptyset$, and $A = \{ x \in \mathbb{R}^d \mid x^{(1)} = 0, \: x^{(3)} \ge 0 \}$. Define 
also $g_1(x) = x^{(2)} - |x^{(3)}| x^{(3)}$ and $g_2(x) = -x^{(2)} - |x^{(3)}| x^{(3)}$.

Let us check optimality conditions at the point $x_* = 0$. Firstly, note that $x_*$ is a not an isolated point of
problem \eqref{probl:MathProg3_CounterEx}, since for any $t \ge 0$ the point $x(t) = (0, 0, t)^T$ is feasible. One has
$I(x_*) = \{ 1, 2 \}$, $\nabla g_1(x_*) = (0, 1, 0)^T$, and $\nabla g_2(x_*) = (0, -1, 0)^T$, which implies that
$\mathcal{N}(x_*) = \cone\{ \nabla g_1(x_*), \nabla g_2(x_*) \} = \{ x \in \mathbb{R}^3 \mid x^{(1)} = x^{(3)} = 0 \}$.
Moreover, $N_A(x_*) = \{ x \in \mathbb{R}^3 \mid x^{(2)} = 0, \: x^{(3)} \le 0 \}$. Hence taking into account
the fact that $\nabla f(x_*) = (1, 0, 1)$ one obtains that
$$
  \mathcal{D}(x_*) = \nabla f(x_*) + \mathcal{N}(x_*) + N_A(x_*) = \nabla f(x_*) 
  + \{ x \in \mathbb{R}^3 \mid x^{(3)} \le 0 \}
  = \{ x \in \mathbb{R}^3 \mid x^{(3)} \le 1 \}.
$$
Thus, $0 \in \interior \mathcal{D}(x_*)$ and by Theorems~\ref{thrm:SuffOptCond} and \ref{thrm:EquivOptCond_Subdiff} the
point $x_*$ is a local minimiser of problem \eqref{probl:MathProg3_CounterEx} at which the first order growth condition
holds true. Let us check that a generalised complete alternance does \textit{not} exist at~$x_*$.

\begin{figure}[t]
\centering
\includegraphics[width=0.4\linewidth]{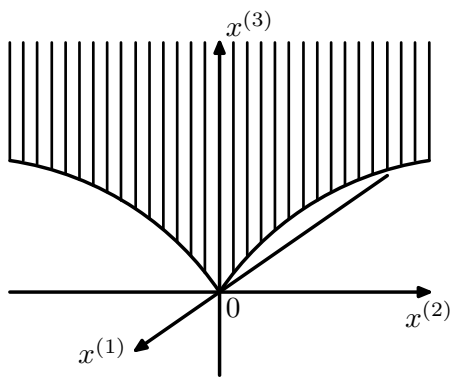}
\caption{The feasible region of problem \eqref{probl:MathProg3_CounterEx} (the shaded area).}
\label{fig:feasible_region}
\end{figure}

Note that $\interior \mathcal{N}(x_*) = \emptyset$, $\interior N_A(x_*) = \emptyset$, 
$- \nabla f(x_*) \notin \mathcal{N}(x_*)$, and $- \nabla f(x_*) \in \relint N_A(x_*)$, but 
$\mathcal{N}(x_*) \ne \{ 0 \}$. Thus, Theorem~\ref{thrm:GenAlternanceCond} is inapplicable. Arguing by reductio ad
absurdum, suppose that a generalised complete alternance $\{ V_1, \ldots, V_4 \}$ exists at $x_*$. Clearly, 
$V_1 = \nabla f(x_*)$ and the vectors $V_2$, $V_3$, and $V_4$ are linearly independent, since 
$\Delta_1 = \determ([V_2, V_3, V_4]) \ne 0$. Hence taking into account the facts that $\mathcal{N}(x_*)$ is one
dimensional and $N_A(x_*)$ is two dimensional one obtains that $V_2 \in \mathcal{N}(x_*) \setminus \{ 0 \}$ and 
$V_3, V_4 \in N_A(x_*) \setminus \{ 0 \}$. However, by Remark~\ref{rmrk:GenAlternanceVsCadre} one has 
$\sum_{i = 1}^4 \beta_i V_i = 0$ for some $\beta_i > 0$, which is impossible due to the fact that $V_2$ is the only
vector whose second coordinate is non-zero. Thus, a generalised complete alternance does not exist at $x_*$.
Nevertheless, observe that putting $V_1 = \nabla f(x_*)$, $V_2 = (-1, 0, 0)^T = N_A(x_*)$, and 
$V_3 = (0, 0, -1)^T \in N_A(x_*)$ one has $V_1 + V_2 + V_3 = 0$ and $\rank([V_1, V_2, V_3]) = 2$, i.e.
a $3$-point alternance exists at $x_*$, which in the case $d = 3$ is not complete. 

Moreover, observe that for $V_1 = \nabla f(x_*)$, $V_2 = (0, 0, -1)^T \in \mathcal{N}(x_*) + N_A(x_*)$,
$V_3 = (-0.5, 1, 0)^T \in \mathcal{N}(x_*) + N_A(x_*)$, and $V_4 = (-0.5, -1, 0)^T \in \mathcal{N}(x_*) + N_A(x_*)$ one
has $V_1 + V_2 + V_3 + V_4 = 0$ and $\rank([V_1, V_2, V_3, V_3]) = 3$. Thus, in accordance with
Remark~\ref{rmrk:IsolatePoint_WeakAlternance} a weak complete alternance exists at $x_*$.

It should be pointed out that RCQ is not satisfied at $x_*$. Therefore we pose an open problem to prove whether the
sufficient optimality condition $0 \in \interior \mathcal{D}(x_*)$ along with RCQ and the assumption that 
$\partial F(x_*) \ne \{ 0 \}$ guarantee the existence of a generalised complete alternance. \qed
\end{example}

Now we give two simple examples of minimax problems.

\begin{example}[\cite{MakelaNeittaanmaki}, Problem~DEM] \label{exmpl:ProblemDEM}
Consider the following problem:
$$
  \min\: F(x) = \max\{ f_1(x), f_2(x), f_3(x) \},
$$
where $f_1(x) = 5 x^{(1)} + x^{(2)}$, $f_2(x) = - 5 x^{(1)} + x^{(2)}$, and 
$f_3(x) = (x^{(1)})^2 + (x^{(2)})^2 + 4 x^{(2)}$. Put $d = 2$ and $W = \{ 1, 2, 3 \}$.

Let us check optimality conditions at the point $x_* = (0, - 3)^T$. One has $W(x_*) = W$ and
$$
  \partial F(x_*) = \co\{ \nabla f_1(x_*), \nabla f_2(x_*), \nabla f_3(x_*) \}
  = \co\left\{ \begin{pmatrix} 5 \\ 1 \end{pmatrix}, \begin{pmatrix} -5 \\ 1 \end{pmatrix},
  \begin{pmatrix} 0 \\ -2 \end{pmatrix} \right\}.
$$

\begin{figure}[t]
\centering
\includegraphics[width=0.5\linewidth]{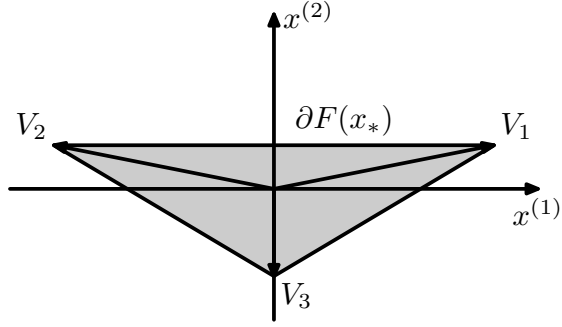}
\caption{The subdifferential $\partial F(x_*)$ (the shaded area) and the vectors $V_1, V_2, V_3 \in \partial F(x_*)$
comprising a complete alternance in Example~\ref{exmpl:ProblemDEM}.}
\label{fig:example_DEM}
\end{figure}

\noindent{}Define $V_1 = \nabla f_1(x_*)$, $V_2 = \nabla f_2(x_*)$, and $V_3 = \nabla f_3(x_*)$. Then
$$
  \Delta_1 = \begin{vmatrix} -5 & 0 \\ 1 & -2 \end{vmatrix} = 10, \quad 
  \Delta_2 = \begin{vmatrix} 5 & 0 \\ 1 & -2 \end{vmatrix} = -10, \quad 
  \Delta_3 = \begin{vmatrix} 5 & -5 \\ 1 & 1 \end{vmatrix} = 10,
$$
that is, a complete alternance exists at $x_*$ (see Fig.~\ref{fig:example_DEM}). Consequently, $x_*$ is a point of
strict local minimum of the function $F$ at which the first order growth condition holds true by
Theorems~\ref{thrm:SuffOptCond}, \ref{thrm:EquivOptCond_Subdiff}, and \ref{thrm:AlternanceCond}. \qed
\end{example}

\begin{example}[\cite{MadsenSchjaer}, modified Example~4] \label{exmpl:ProblemMadsen}
Let $d = 2$ and consider the following constrained minimax problem:
\begin{equation} \label{probl:ConstrMinMaxEx}
  \min\: F(x) = \max\{  f_1(x), f_2(x), f_3(x) \}
  \quad \text{subject to} \quad x^{(1)} \ge 0, \quad x^{(2)} \ge 1,
\end{equation}
where $f_1(x) = (x^{(1)})^2 + (x^{(2)})^2 + x^{(1)} x^{(2)} - 1$, $f_2(x) = \sin x^{(1)}$, $f_3(x) = - \cos x^{(2)}$.
Define $W = \{ 1, 2, 3 \}$ and $A = \{ x \in \mathbb{R}^2 \mid x^{(1)} \ge 0, \: x^{(2)} \ge 1 \}$.

Let us check optimality conditions at the point $x_* = (0, 1)^T$. One has $W(x_*) = \{ 1, 2 \}$, 
$N_A(x_*) = \{ x \in \mathbb{R}^2 \mid x^{(1)} \le 0, \: x^{(2)} \le 0 \}$, and
$$
  \partial F(x_*) = \co\{ \nabla f_1(x_*), \nabla f_2(x_*) \}
  = \co\left\{ \begin{pmatrix} 1 \\ 2 \end{pmatrix}, \begin{pmatrix} 1 \\ 0 \end{pmatrix} \right\}.
$$

\begin{figure}[t]
\centering
\includegraphics[width=0.4\linewidth]{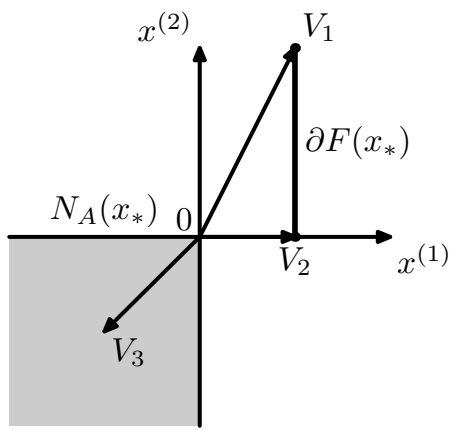}
\caption{The subdifferential $\partial F(x_*)$ (the vertical line segment), the normal cone $N_A(x_*)$ (the shaded
area), and the vectors $V_1, V_2 \in \partial F(x_*)$ and $V_3 \in N_A(x_*)$ comprising a generalised complete
alternance in Example~\ref{exmpl:ProblemMadsen}.}
\label{fig:Example_Madsen}
\end{figure}

\noindent{}Put $V_1 = \nabla f_1(x_*)$, $V_2 = \nabla f_2(x_*)$, and $V_3 = (-1, -1)^T \in N_A(x_*)$. Then
$$
  \Delta_1 = \begin{vmatrix} 1 & -1 \\ 0 & -1 \end{vmatrix} = -1, \quad 
  \Delta_2 = \begin{vmatrix} 1 & -1 \\ 2 & -1 \end{vmatrix} = 1, \quad 
  \Delta_3 = \begin{vmatrix} 1 & 1 \\ 2 & 0 \end{vmatrix} = -2,
$$
that is, a generalised complete alternance exists at $x_*$ (see~Fig.~\ref{fig:Example_Madsen}). Consequently, by
Theorems~\ref{thrm:SuffOptCond}, \ref{thrm:EquivOptCond_Subdiff}, and \ref{thrm:GenAlternanceCond} the point $x_*$ is a
locally optimal solution of problem \eqref{probl:ConstrMinMaxEx} at which the first order growth condition holds true.

Note that it is natural to put $n_A(x_*) = \{ (-1, 0)^T, (0, - 1)^T \}$, since $N_A(x_*) = \cone n_A(x_*)$, and
analyse optimality condition in terms of non-generalised alternance. Similarly, one can consider inequality constraints
$g_1(x) = - x^{(1)} \le 0$ and $g_2(x) - x^{(2)} + 1 \le 0$, and define $A = \mathbb{R}^2$ and
$\eta(x_*) = \{ \nabla g_1(x_*), \nabla g_2(x_*) \}$. However, one can check that in both cases only a $2$-point
alternance exists at $x_*$, which in the case $d = 2$ is not complete. \qed
\end{example}

\subsubsection{Nonlinear second order cone minimax problems}

Let $(\mathcal{P})$ be a nonlinear second order cone minimax problem of the form:
\begin{equation} \label{probl:SecondOrderConeMinimax}
  \min_x \max_{\omega \in W} f(x, \omega) \quad 
  \text{s.t.} \quad g_i(x) \in K_{l_i + 1}, \quad i \in I, 
  \quad b(x) = 0, \quad x \in A,
\end{equation}
where $g_i \colon \mathbb{R}^d \to \mathbb{R}^{l_i + 1}$, $I = \{ 1, \ldots, r \}$ and
$b \colon \mathbb{R}^d \to \mathbb{R}^s$ are continuously differentiable functions, and
$$
  K_{l_i + 1} = \big\{ y = (y^0, \overline{y}) \in \mathbb{R} \times \mathbb{R}^{l_i} \bigm|
  y^0 \ge |\overline{y}| \big\}
$$

\begin{figure}[t]
\centering
\includegraphics[width=0.4\linewidth]{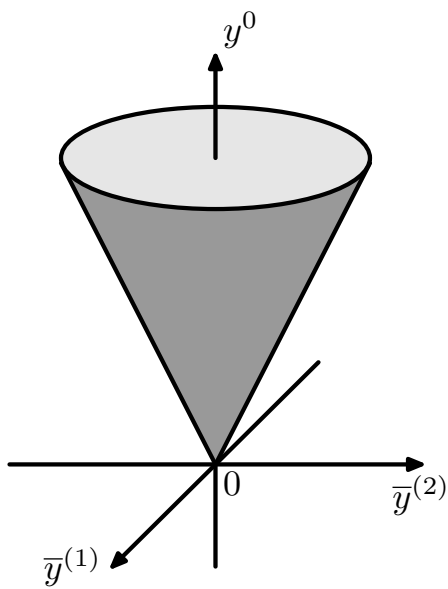}
\caption{The second order (Lorentz, ice-cream) cone of dimension $3$.}
\label{fig:Lorentz_cone}
\end{figure}

\noindent{}is the second order (Lorentz, ice-cream) cone of dimension $l_i + 1$ (see Fig.~\ref{fig:Lorentz_cone}). In
this case
$$
  Y = \mathbb{R}^{l_1 + 1} \times \ldots \times \mathbb{R}^{l_r + 1} \times \mathbb{R}^s, \quad
  K = K_{l_1 + 1} \times \ldots \times K_{l_r + 1} \times \{ 0_s \},
$$
and $G(\cdot) = (g_1(\cdot), \ldots, g_r(\cdot), b(\cdot))$. Furthermore, 
for any $\lambda = (\lambda_1, \ldots, \lambda_r, \nu) \in Y$ one has
$$
  L(x, \lambda) = f(x) + \sum_{i = 1}^r \langle \lambda_i, g_i(x) \rangle + \langle \nu, g(x) \rangle,
  \quad K^* = (- K_{l_1 + 1}) \times \ldots \times (- K_{l_r + 1}) \times \mathbb{R}^s.
$$
Finally, one can easily verify (cf.~\cite[Lemma~2.99]{BonnansShapiro}) that in the case $A = \mathbb{R}^d$ RCQ for 
problem \eqref{probl:SecondOrderConeMinimax} is satisfied at a feasible point $x$ iff the Jacobian matrix $\nabla b(x)$
has full row rank and there exists $h \in \mathbb{R}^d$ such that $\nabla b(x) h = 0$ and 
$g_i(x) + \nabla g_i(x) h \in \interior K_{l_i + 1}$ for all
$i \in I(x) = \{ i \in I \mid g_i^0(x) = |\overline{g}_i(x)| \}$, 
where $g_i(x) = (g_i^0(x), \overline{g}_i(x)) \in \mathbb{R} \times \mathbb{R}^{l_i}$ 
(here we used the obvious equality 
$\interior K_{l_i + 1} = \{ y = (y^0, \overline{y}) \in \mathbb{R} \times \mathbb{R}^{l_i} \bigm| y^0 >
|\overline{y}|$).

If we equip the space $Y$ with the norm $\| y \| = \sum_{i = 1}^r |y_i| + |z|$ for 
any $y = (y_1, \ldots, y_r, z) \in Y$, then the penalty function for problem \eqref{probl:ConstrainedMinimax}
takes the form
$$
  \Phi_c(x) = \max_{\omega \in W} f(x, \omega) + c \sum_{i = 1}^r \big| g_i(x) - P_{K_{l_i + 1}}(g_i(x)) \big|
  + c |b(x)|
$$
where
$$
  P_{K_{l_i + 1}}(y) = \begin{cases}
    \frac{\max\{ y^0 + |\overline{y}|, 0\}}{2}\left(1, \frac{\overline{y}}{|\overline{y}|}\right) &
    \text{if } y^0 \le |\overline{y}|, \\
    y, & \text{if } y^0 > |\overline{y}|
  \end{cases}
$$
is the Euclidean projection of $y = (y^0, \overline{y}) \in \mathbb{R} \times \mathbb{R}^{l_i}$ onto the second order
cone $K_{l_i + 1}$ (see \cite[Thrm.~3.3.6]{Bauschke}; an alternative expression for the projection can be found in
\cite[Prp.~3.3]{FukushimaLuoTseng}). Note also that for any feasible point $x$ one has
\begin{multline*}
  \mathcal{N}(x) = \Big\{ \sum_{i = 1}^r \nabla g_i(x)^T \lambda_i + \nabla b(x)^T \nu \Bigm| 
  \lambda_i \in - K_{l_i + 1}, \: \langle \lambda_i, g_i(x) \rangle = 0 \enspace \forall i \in I,
  \nu \in \mathbb{R}^s \Big\} \\
  = \Big\{ \sum_{i \in I_+(x)} t_i \nabla g_i(x)^T 
  \left( \begin{smallmatrix} - g_i^0(x) \\ \overline{g}_i(x) \end{smallmatrix} \right) 
  + \sum_{i \in I_0(x)} \nabla g_i(x)^T \lambda_i  
  + \nabla b(x)^T \nu \Bigm| \\
  t_i \ge 0 \: \forall i \in I_+(x), \: 
  \lambda_i \in - K_{l_i + 1} \: \forall i \in I_0(x), \: \nu \in \mathbb{R}^s \Big\},
\end{multline*}
where $I_0(x) = \{ i \in I(x) \mid g_i(x) = 0 \}$ and $I_+(x) = I(x) \setminus I_0(x)$. Here we used the following
simple auxiliary result.

\begin{lemma}
Let $y = (y^0, \overline{y}) \in K_{l + 1} \setminus \{ 0 \}$ and 
$\lambda = (\lambda^0, \overline{\lambda}) \in - K_{l + 1}$ with $l \in \mathbb{N}$ be such that 
$\langle \lambda, y \rangle = 0$. Then $\lambda = 0$, if $y^0 > \overline{y}$, and $\lambda = t (- y^0, \overline{y})$
for some $t \ge 0$, if $y^0 = |\overline{y}|$.
\end{lemma}

\begin{proof}
Indeed, by definition 
$\langle \lambda, y \rangle = \lambda^0 y^0 + \langle \overline{\lambda}, \overline{y} \rangle = 0$. Hence taking into
account the fact that $y^0 > 0$, since $y \in K_{l + 1} \setminus \{ 0 \}$, one obtains that
\begin{equation} \label{eq:LorentzCone_ZeroProd}
  \lambda^0 = - \frac{1}{y^0} \langle \overline{\lambda}, \overline{y} \rangle \ge
  - \frac{1}{y^0} |\overline{\lambda}| \cdot |\overline{y}|.
\end{equation}
Therefore, if $y^0 > |\overline{y}|$, then either (1) $\lambda = 0$ or (2) $\overline{\lambda} = 0$ and $\lambda^0 > 0$
or (3) $\lambda^0 > - |\overline{\lambda}|$. Note, however, that only the first case is possible, since 
$\lambda \in - K_{l + 1}$. Thus, $\lambda = 0$, if $y^0 > \overline{y}$. 

On the other hand, if $y^0 = |\overline{y}|$, then taking into account \eqref{eq:LorentzCone_ZeroProd} and the fact that
$\lambda \in - K_{l + 1}$, i.e. $\lambda^0 \le - |\overline{\lambda}|$, one obtains that 
$\lambda^0 = - |\overline{\lambda}|$ and 
$\langle \overline{\lambda}, \overline{y} \rangle = |\overline{\lambda}| \cdot |\overline{y}|$, that is,
$\overline{\lambda} = t \overline{y}$ for some $t \ge 0$. Thus, $\lambda = t (- y^0, \overline{y})$ for some 
$t \ge 0$, if $y^0 = |\overline{y}|$.
\end{proof}

Thus, it is natural to define
\begin{align*}
  \eta(x) 
  &= \Big\{ \nabla g_i(x)^T \left( \begin{smallmatrix} - g_i^0(x) \\ \overline{g}_i(x) \end{smallmatrix} \right)
  \Bigm| i \in I_+(x) \Big\}  \\
  &\cup \Big\{ \nabla g_i(x)^T \left( \begin{smallmatrix} - 1 \\ |v| \end{smallmatrix} \right) \Bigm|
  i \in I_0(x), \: v \in \mathbb{R}^{l_i}, \: |v| = 1 \Big\} 
  \cup \big\{ \nabla b_1(x), \ldots \nabla b_s(x) \big\}
\end{align*}
(here $b(\cdot) = (b_1(\cdot), \ldots, b_s(\cdot))$), since in the general case this is the smallest set such that
$\mathcal{N}(x) = \cone \eta(x)$.

Let us give an example demonstrating how one can verify alternance optimality conditions in the case of nonlinear
second order cone minimax problems.

\begin{example}
Consider the following second order cone minimax problem:
\begin{equation} \label{probl:2OrderConeMinimaxEx}
\begin{split}
  &\min\: F(x) = \max\{ (x^{(1)})^2 + (x^{(2)})^2 + 4 x^{(1)} - x^{(2)}, \sin x^{(1)} - x^{(2)}, \cos x^{(2)} - 1 \} \\
  &\text{s.t.} \quad g_1(x) = (- x^{(1)} + \sin x^{(2)} + 1, \sin x^{(1)} - 2 x^{(2)} - 1) \in K_2, \\
  &g_2(x) = (2 (x^{(1)})^2 + 2 (x^{(2)})^2, x^{(1)} + x^{(2)}, 2 x^{(2)}) \in K_3.
\end{split}
\end{equation}
Define $d = 2$, $f_1(x) = (x^{(1)})^2 + (x^{(2)})^2 + 4 x^{(1)} - x^{(2)}$, $f_2(x) = \sin x^{(1)} - x^{(2)}$, $f_3(x) =
\cos x^{(2)} - 1$, 
$W = \{ 1, 2, 3 \}$, $I = \{ 1, 2 \}$, and $A = \mathbb{R}^d$.

Let us check optimality conditions at the point $x_* = 0$. Observe that $W(x_*) = \{ 1, 2, 3 \}$ and
$$
  \partial F(x_*) = \co\{ \nabla f_1(x_*), \nabla f_2(x_*), \nabla f_3(x_*) \nabla \}
  = \co\left\{ \begin{pmatrix} 4 \\ -1 \end{pmatrix}, \begin{pmatrix} 1 \\ -1 \end{pmatrix},
  \begin{pmatrix} 0 \\ 0 \end{pmatrix} \right\}.
$$
Note also that $g_1(x_*) = (1, - 1) \in K_2$, 
$\nabla g_1(x_*)^T = \left( \begin{smallmatrix} -1 & 1 \\ 1 & -2 \end{smallmatrix} \right)$, 
$g_2(x_*) = 0 \in K_3$, and 
$\nabla g_2(x_*)^T = \left( \begin{smallmatrix} 0 & 1 & 0 \\ 0 & 1 & 2 \end{smallmatrix} \right)$. Therefore 
$I_+(x_*) = \{ 1 \}$, $I_0(x_*) = \{ 2 \}$, and
\begin{align*}
  \eta(x_*) &= \Big\{ 
  \nabla g_1(x_*)^T \left( \begin{smallmatrix} - g_1^0(x_*) \\ \overline{g}_i(x_*) \end{smallmatrix} \right) \Big\}
  \cup \Big\{ \nabla g_2(x_*)^T \left( \begin{smallmatrix} - 1 \\ |v| \end{smallmatrix} \right)
  \Bigm| v \in \mathbb{R}^2 \colon |v| = 1 \Big\} \\
  &= \left\{ \begin{pmatrix} 0 \\ 1 \end{pmatrix} \right\} \cup 
  \left\{ \begin{pmatrix} v^{(1)} \\ v^{(1)}  + 2 v^{(2)} \end{pmatrix} \Bigm| 
  v \in \mathbb{R}^2 \colon |v| = 1 \right\}.
\end{align*}
Let $V_1 = \nabla f_1(x_*)$, $V_2 = (0, 1)^T \in \eta(x_*)$, and 
$V_3 = (v^{(1)}, v^{(1)} + 2 v^{(2)})^T \in \eta(x_*)$ with $v = (- 1 / \sqrt{2}, - 1 / \sqrt{2})^T$. Then
$$
  \Delta_1 = \begin{vmatrix} 0 & -\frac{1}{\sqrt{2}} \\ 1 & -\frac{3}{\sqrt{2}} \end{vmatrix} = 
  \frac{1}{\sqrt{2}}, \quad 
  \Delta_2 = \begin{vmatrix} 4 & -\frac{1}{\sqrt{2}} \\ -1 & -\frac{3}{\sqrt{2}} \end{vmatrix} = 
  -\frac{13}{\sqrt{2}}, \quad 
  \Delta_3 = \begin{vmatrix} 4 & 0 \\ -1 & 1 \end{vmatrix} = 4,
$$
that is, a complete alternance exists at $x_*$. Therefore, by Theorems~\ref{thrm:SuffOptCond},
\ref{thrm:EquivOptCond_Subdiff}, and \ref{thrm:AlternanceCond} the point $x_*$ is a locally optimal solution of problem
\eqref{probl:2OrderConeMinimaxEx} at which the first order growth condition holds true. \qed
\end{example}

\subsubsection{Nonlinear semidefinite minimax problems}

Let now $(\mathcal{P})$ be a nonlinear semidefinite minimax problem of the form:
\begin{equation} \label{probl:NonlinearSemiDefProg}
  \min_x \max_{\omega \in W} f(x, \omega) \quad 
  \text{subject to} \quad G_0(x) \preceq 0, \quad b(x) = 0, \quad x \in A,
\end{equation}
where $G_0 \colon \mathbb{R}^d \to \mathbb{S}^l$ and $b \colon \mathbb{R}^d \to \mathbb{R}^s$ are continuously
differentiable functions, $\mathbb{S}^l$ denotes the set of all $l \times l$ real symmetric matrices, and  the relation
$G_0(x) \preceq 0$ means that the matrix $G_0(x)$ is negative semidefinite. In this case, 
$Y = \mathbb{S}^l \times \mathbb{R}^s$, $G(\cdot) = (G_0(\cdot), b(\cdot))$ and
$K = \mathbb{S}^l_- \times 0_s$, where $\mathbb{S}^l_-$ is the cone of $l \times l$ negative semidefinite
matrices.

We equip $Y$ with the inner product 
$\langle (B_1, z_1), (B_2, z_2) \rangle = \trace(B_1 B_2) + \langle z_1, z_2 \rangle$
for any $(B_1, z_1), (B_2, z_2) \in Y$, where $\trace(\cdot)$ is the trace of a matrix, and the corresponding norm
$\| (B, z) \|^2 = \| B \|_F^2 + |z|^2$, where $\| B \|_F = \sqrt{Tr(B^2)}$ is the Frobenius norm. Then one has
$L(x, \lambda) = F(x) + \trace(\lambda_0 \cdot G_0(x)) + \langle \nu, h(x) \rangle$ for any 
$(\lambda_0, \nu) \in \mathbb{S}^l \times \mathbb{R}^s$ and $K^* = \mathbb{S}^l_+ \times \mathbb{R}^s$, where
$\mathbb{S}^l_+ = - \mathbb{S}^l_{-}$ is the cone of positive semidefinite matrices. Note also that in the case 
$A = \mathbb{R}^d$ RCQ for problem \eqref{probl:NonlinearSemiDefProg} holds true at a feasible point $x$ iff the
Jacobian matrix $\nabla b(x)$ has full row rank and there exists $h \in \mathbb{R}^d$ such that $\nabla b(x) h = 0$ and
the matrix $G_0(x) + D G_0(x)h$ is negative definite (cf. \cite[Lemma~2.99]{BonnansShapiro}).

The penalty function for problem \eqref{probl:NonlinearSemiDefProg} has the form
$$
  \Phi_c(x) = f(x) + c \sqrt{\| G_0(x) - P_{\mathbb{S}^l_{-}}(G_0(x)) \|_F^2 + |b(x)|^2},
$$
where $P_{\mathbb{S}^l_{-}}(G_0(x))$ is the projection of $G_0(x)$ onto the cone $\mathbb{S}^l_{-}$ of negative
semidefinite matrices. One can verify that
\begin{align*}
  P_{\mathbb{S}^l_{-}}(G_0(x)) &= 0.5 (G_0(x) - \sqrt{G_0(x)^2}) \\
  &= Q \diag\Big( \min\{ 0, \sigma_1(G_0(x)) \}, \ldots, \min\{ 0, \sigma_l(G_0(x)) \} \Big) Q^T,
\end{align*}
where $G_0(x) = Q \diag(\sigma_1(G_0(x)), \ldots, \sigma_l(G_0(x))) Q^T$ is a spectral decomposition of $G_0(x)$, and
$\sigma_1(G_0(x)), \ldots, \sigma_l(G_0(x))$ are the eigenvalues of $G_0(x)$ listed in the decreasing order (see, e.g.
\cite{Higham,Malick}). Consequently, one has 
$$
  \| G_0(x) - P_{\mathbb{S}^l_{-}}(G_0(x)) \|_F = \frac{1}{2} \| G_0(x) + \sqrt{G_0(x)^2} \|_F
  = \sqrt{\sum_{i = 1}^l \max\big\{ 0, \sigma_i(G_0(x)) \big\}^2}.
$$
Observe also that for any feasible point $x$ such that $r = \rank G_0(x) < l$ one has
\begin{multline*}
  \mathcal{N}(x) 
  = \Big\{ \Big( \langle \lambda_0, D_{x_1} G_0(x) \rangle, \ldots, \langle \lambda_0, D_{x_d} G_0(x) \rangle \Big)^T
  + \nabla b(x)^T \nu \Bigm| \\
  (\lambda_0, \nu) \in \mathbb{S}^l_+ \times \mathbb{R}^s, \: \langle \lambda_0, G_0(x) \rangle = 0 \Big\}
\end{multline*}
or, equivalently,
$$
  \mathcal{N}(x) = \Big\{ 
  \Big(\langle Q_0 \Gamma Q_0^T, D_{x_1} G_0(x)  \rangle, \ldots, \langle Q_0 \Gamma Q_0^T, D_{x_d} G_0(x) \rangle
  \Big)^T + \nabla b(x)^T \nu \Bigm| (\Gamma, \nu) \in \mathbb{S}^{l - r}_+ \times \mathbb{R}^s \Big\},
$$
where $D_{x_i} = \partial / \partial x_i$, $r = \rank G_0(x)$, and $Q_0$ is an $l \times (l - r)$ matrix whose columns
are an orthonormal basis $q_1, \ldots, q_{l - r}$ of the null space of the matrix $G_0(x)$. 
In the case $r = \rank G_0(x) = l$ one has $\mathcal{N}(x_*) = \{ \nabla b(x)^T \nu \mid \nu \in \mathbb{R}^s \}$.
Here we used the following simple auxiliary result.

\begin{lemma}
Let $\lambda_0 \in \mathbb{S}^l_+$ be a given matrix. Then the following statements are equivalent:
\begin{enumerate}
\item{$\langle \lambda_0, G_0(x) \rangle = \trace(\lambda_0 G_0(x)) = 0$;
\label{stat:ZeroFrobeniusProd}
}

\item{$\lambda_0 = Q_0 \Gamma Q_0^T$ for some $\Gamma \in \mathbb{S}^{l - r}_+$ in the case $r < l$ and 
$\lambda_0 = 0$ otherwise;
\label{stat:SemidefFormNullVectors}}

\item{$\lambda_0 \in \cone\{ q q^T \mid q \in \mathbb{R}^l \colon G_0(x) q = 0 \}$.
\label{stat:ExtremeMatricesViaNullSpace}}
\end{enumerate}
\end{lemma}

\begin{proof}
Let, as above, $\sigma_1(G_0(x)), \ldots, \sigma_l(G_0(x))$ be the eigenvalues of $G_0(x)$ listed in the decreasing
order. Recall that $x$ is feasible point of problem \eqref{probl:NonlinearSemiDefProg}, i.e. $G_0(x) \preceq 0$.
Therefore
\begin{equation} \label{eq:EigenvaluesNegativeSemiDef}
  \sigma_i(G_0(x)) = 0 \quad \forall i \in \{ 1, \ldots, l - r \}, \quad 
  \sigma_i(G_0(x)) < 0 \quad \forall i \in \{ l - r + 1, \ldots, l \}. 
\end{equation}
Let also $G_0(x) = Q \diag(\sigma_1(G_0(x)), \ldots, \sigma_l(G_0(x))) Q^T$ be a spectral decomposition of $G_0(x)$ such
that the first $l - r$ columns of $Q$ coincide with $Q_0$. 

\ref{stat:ZeroFrobeniusProd} ${}\implies{}$ \ref{stat:ExtremeMatricesViaNullSpace}. 
Suppose that $\langle \lambda_0, G_0(x) \rangle = 0$. Bearing in mind the fact that the trace operator is invariant
under cyclic permutations one obtains that
\begin{equation} \label{eq:FrobeniusNorm_Eigenvalues}
\begin{split}
  0 = \trace\big( \lambda_0 G_0(x) \big) 
  &= \trace\Big( Q^T \lambda_0 Q \diag(\sigma_1(G_0(x)), \ldots, \sigma_l(G_0(x))) \Big) \\
  &= \sum_{i = 1}^l \sigma_i(G_0(x)) q_i^T \lambda_0 q_i,
\end{split}
\end{equation}
where $q_i$ are the columns of the matrix $Q$. Hence with the use of \eqref{eq:EigenvaluesNegativeSemiDef} and the fact
that $\lambda_0 \in \mathbb{S}^l_+$ one obtains that $q_i^T \lambda_0 q_i = 0$ for any 
$i \in \{ l - r + 1, \ldots, l \}$.

Since the matrix $\lambda_0$ is positive semidefinite, there exists orthogonal vectors 
$z_1, \ldots, z_k \in \mathbb{R}^l$ such that $\lambda_0 = z_1 z_1^T + \ldots + z_k z_k^T$
(see, e.g. \cite[Thrm.~7.5.2]{HornJohnson}). Consequently, one has
$$
  0 = q_i^T \lambda_0 q_i = \sum_{j = 1}^k q_i^T z_j z_j^T q_i = \sum_{j = 1}^k |z_j q_i|^2
  \quad \forall i \in \{ l - r + 1, \ldots, l \}.
$$
Therefore, the vectors $z_1, \ldots, z_k$ belong to the orthogonal complement of the linear span of eigenvectors $q_i$
of $G_0(x)$ corresponding to nonzero eigenvalues, which coincides with the null space of $G_0(x)$. Thus, 
$G_0(x) z_i = 0$ for all $i \in \{ 1, \ldots, k \}$, that is, 
$\lambda_0 \in \cone\{ q q^T \mid q \in \mathbb{R}^l \colon G_0(x) q = 0 \}$.

\ref{stat:ExtremeMatricesViaNullSpace} ${}\implies{}$ \ref{stat:SemidefFormNullVectors}. If $r = \rank G_0(x) = l$,
then $G_0(x) q = 0$ iff $q = 0$, which implies that $\lambda_0 = 0$. Thus, one can suppose that $r < l$. Then
$\lambda_0 = \sum_{i = 1}^k \alpha_i z_i z_i^T$ for some $\alpha_i \ge 0$ and $z_i \in \mathbb{R}^l$ such that
$G_0(x) z_i = 0$. Since $z_i$ belongs to the null space of $G_0(x)$ and the columns of the matrix $Q_0$ are an
orthonormal basis of this space, there exists vectors $u_i \in \mathbb{R}^{l - r}$ such that $z_i = Q_0 u_i$ for all
$i$. Therefore
$$
  \lambda_0 = \sum_{i = 1}^k \alpha_i z_i z_i^T = \sum_{i = 1}^k \alpha_i Q_0 u_i u_i^T Q_0^T
  = Q_0 \Big( \sum_{i = 1}^k \alpha_i u_i u_i^T \Big) Q_0^T.
$$
Define $\Gamma = \sum_{i = 1}^k \alpha_i u_i u_i^T$. Then $\lambda_0 = Q_0 \Gamma Q_0^T$ and, as is easily seen,
$\Gamma \in \mathbb{S}^{l - r}_+$. 

\ref{stat:SemidefFormNullVectors} ${}\implies{}$ \ref{stat:ZeroFrobeniusProd}. Suppose now that 
$\lambda_0 = Q_0 \Gamma Q_0^T$ for some $\Gamma \in \mathbb{S}^{l - r}_+$ in the case $r < l$ and $\lambda_0 = 0$
otherwise. If $\lambda_0 = 0$, then obviously $\langle \lambda_0, G_0(x) \rangle = 0$. Thus, one can suppose that 
$r < l$. Observe that
$$
  Q^T \lambda_0 Q = Q^T Q_0 \Gamma Q_0^T Q
  = Q^T Q \begin{pmatrix} \Gamma & 0 \\ 0 & 0 \end{pmatrix} Q^T Q 
  = \begin{pmatrix} \Gamma & 0 \\ 0 & 0 \end{pmatrix},
$$
where $0$ are zero matrices of corresponding dimensions. Hence taking into account the fact that 
$[Q^T \lambda_0 Q]_{ij} = q_i^T \lambda_0 q_j$ one obtains that $q_i^T \lambda_0 q_i = 0$ for 
any $i \in \{ l - r + 1, \ldots, l \}$, which with the use of the last two equalities in 
\eqref{eq:FrobeniusNorm_Eigenvalues} implies that $\langle \lambda_0, G_0(x) \rangle = 0$.
\end{proof}

Taking into account the equality $\trace(q q^T D_{x_i} G_0(x)) = q^T D_{x_i} G_0(x) q$ and the previous lemma one can
define
\begin{multline*}
  \eta(x) = \big\{ \nabla b_1(x), \ldots \nabla b_s(x) \big\} \\
  \cup  \Big\{ \Big( q^T D_{x_1} G_0(x) q, \ldots, q^T D_{x_d} G_0(x) q \rangle \Big)^T \in \mathbb{R}^d \Bigm|
  q \in \mathbb{R}^l \colon |q| = 1, \: G_0(x) q = 0 \Big\}
\end{multline*}
in the case $\rank G_0(x) < l$, and $\eta(x) = \{ \nabla b_1(x), \ldots \nabla b_s(x) \}$, if $\rank G_0(x) = l$. Let us
give a simple example illustrating alternance optimality conditions in the case of nonlinear semidefinite minimax
problems.

\begin{example}
Let $d = 3$, $W = \{ 1, 2, 3 \}$, $l = 3$, and $A = \mathbb{R}^d$. Consider the following nonlinear semidefinite minimax
problem: 
\begin{equation} \label{probl:SemiDefMinimaxEx}
  \min \: F(x) = \max\big\{ f_1(x), f_2(x), f_3(x) \big\} \quad \text{subject to} \quad 
  G_0(x) \preceq 0,
\end{equation}
where $f_1(x) = - 3 x^{(1)} - 3 x^{(2)} - 2 \sin x^{(3)}$, 
$f_2(x) = - x^{(1)} + (x^{(2)})^2 + (x^{(3)})^2 - 1$, and $f_3(x) = (x^{(1)} - 1)^2 + 2 x^{(3)}$, and 
$$
  G_0(x) = 
  \begin{pmatrix} x^{(1)} - (x^{(2)})^2 & \sin x^{(3)} & x^{(1)} + x^{(2)} + x^{(3)} \\ 
  \sin x^{(3)} & x^{(2)} & x^{(1)} x^{(2)} + (x^{(3)} + 1)^2 \\
  x^{(1)} + x^{(2)} + x^{(3)} & x^{(1)} x^{(2)} + (x^{(3)} + 1)^2 & (x^{(1)})^2 + (x^{(2)})^2 - x^{(3)} - 2
  \end{pmatrix}.
$$
Let us check optimality conditions at the point $x_* = (1, -1, 0)^T$. One has $W(x_*) = \{ 1, 3 \}$ and
$$
  \partial F(x_*) = \co\{ \nabla f_1(x_*), \nabla f_3(x_*) \}
  = \co\left\{ \left( \begin{smallmatrix} -3 \\ -3 \\ -2 \end{smallmatrix} \right), 
  \left( \begin{smallmatrix} 0 \\ 0 \\ 2 \end{smallmatrix} \right) \right\}, \quad
  G_0(x_*) = \left( \begin{smallmatrix} 0 & 0 & 0 \\ 0 & -1 & 0 \\ 0 & 0 & 0 \end{smallmatrix} \right).
$$
Consequently, $G_0(x_*) \preceq 0$ and $\rank G_0(x_*) = 1$, which implies that $x_*$ is a feasible point of problem
\eqref{probl:SemiDefMinimaxEx} and by definition one has
$$
  \eta(x) 
  = \Big\{ \Big( q^T D_{x_1} G_0(x) q, q^T D_{x_2} G_0(x) q, q^T D_{x_3} G_0(x) q \Big)^T 
  \in \mathbb{R}^d \Bigm|
  q \in \mathbb{R}^3 \colon |q| = 1, \: G_0(x) q = 0 \Big\}. 
$$
Let $V_1 = \nabla f_1(x_*)$ and $V_2 = \nabla f_3(x_*)$. For $q_1 = (1, 0, 0)^T$ and
$q_2 = (0, 0, 1)^T$ one has $G_0(x_*) q_1 = 0$, $G_0(x_*) q_2 = 0$, and
$$
  V_3 = \begin{pmatrix} q_1^T D_{x_1} G_0(x_*) q_1 \\ q_1^T D_{x_2} G_0(x_*) q_1 \\
  q_1^T D_{x_3} G_0(x_*) q_1 \end{pmatrix} = \begin{pmatrix} 1 \\ 2 \\ 0 \end{pmatrix}, \quad
  V_4 = \begin{pmatrix} q_2^T D_{x_1} G_0(x_*) q_2 \\ q_2^T D_{x_2} G_0(x_*) q_2 \\
  q_2^T D_{x_3} G_0(x_*) q_2 \end{pmatrix} = \begin{pmatrix} 2 \\ -2 \\ -1 \end{pmatrix}.
$$
By definition $V_3, V_4 \in \eta(x_*)$. For the chosen vectors $V_1, V_2, V_3$, and $V_4$ one has
\begin{align*}
  \Delta_1 &= \left|\begin{smallmatrix} 0 & 1 & 2  \\ 0 & 2 & -2  \\ 2 & 0 & -1 \end{smallmatrix}\right| = -12, \quad 
  \Delta_2 = \left|\begin{smallmatrix} -3 & 1 & 2  \\ -3 & 2 & -2  \\ -2 & 0 & -1 \end{smallmatrix}\right| = 15, \\
  \Delta_3 &= \left|\begin{smallmatrix} -3 & 0 & 2  \\ -3 & 0 & -2  \\ -2 & 2 & -1 \end{smallmatrix}\right| = -24, \quad
  \Delta_4 = \left|\begin{smallmatrix} -3 & 0 & 1  \\ -3 & 0 & 2  \\ -2 & 2 & 0 \end{smallmatrix}\right| = 6.
\end{align*}
Thus, a complete alternance exists at $x_*$, which by Theorems~\ref{thrm:SuffOptCond},
\ref{thrm:EquivOptCond_Subdiff}, and \ref{thrm:AlternanceCond} implies that the point $x_*$ is a locally optimal
solution of problem \eqref{probl:SemiDefMinimaxEx} at which the first order growth condition holds true. \qed
\end{example}

\subsubsection{Semi-infinite minimax problems}

Let finally $(\mathcal{P})$ be a nonlinear semi-infinite minimax problem of the form:
\begin{equation} \label{probl:SemiInfProg}
  \min_x \max_{\omega \in W} f(x, \omega) \quad 
  \text{s.t.} \quad g_i(x, t) \le 0, \quad t \in T, \quad i \in I, \quad b(x) = 0, \quad x \in A,
\end{equation}
where the mapping $b \colon \mathbb{R}^d \to \mathbb{R}^s$ is continuously differentiable, $T$ is a compact metric
space, and the functions $g_i \colon \mathbb{R}^d \times T \to \mathbb{R}$, $g_i = g_i(x, t)$, are continuous jointly in
$x$ and $t$, differentiable in $x$ for any $t \in T$, and the functions $\nabla_x g_i$ are continuous, 
$i \in I = \{ 1, \ldots, l \}$.

Let $C(T)$ be the space of all real-valued continuous functions defined on $T$ equipped with the uniform norm 
$\| \cdot \|_{\infty}$, and $C_{-}(T)$ be the closed convex cone consisting of all nonpositive functions from $C(T)$. As
is well-known (see, e.g. \cite[Thrm.~IV.6.3]{DunfordSchwartz}), the topological dual space of $C(T)$ is isometrically
isomorphic to the space of signed (i.e. real-valued) regular Borel measures on $T$, denoted by $rca(T)$, while the set
of regular Borel measures (which constitute a closed convex cone in $rca(T)$) is denoted by $rca_+(T)$. Define 
$Y = (C(T))^l \times \mathbb{R}$, $K = (C_{-}(T))^l \times \{ 0_s \}$, and introduce the mapping 
$G \colon \mathbb{R}^d \to Y$ by setting $G(x) = (g_1(x, \cdot), \ldots, g_l(x, \cdot), b(x))$. Then problem
\eqref{probl:SemiInfProg} is equivalent to problem $(\mathcal{P})$. We endow the space $Y$ with the norm
$\| y \| = \sum_{i = 1}^n \| y_i \|_{\infty} + |z|$ for all $y = (y_1, \ldots, y_l, z) \in Y$.

Observe that the dual space $Y^*$ is isometrically isomorphic (and thus can be identified with) 
$rca(T)^l \times \mathbb{R}^s$, while the polar cone $K^*$ can be identified with the cone 
$(rca_+(T))^l \times \mathbb{R}^s$. Then for any $\lambda = (\mu_1, \ldots, \mu_l, \nu) \in Y^*$ one has
$L(x, \lambda) = F(x) + \sum_{i = 1}^l \int_T g(x, t) d \mu_i(t) + \langle \nu, b(x) \rangle$. Note also that in the
case $A = \mathbb{R}^d$ RCQ for problem \eqref{probl:SemiInfProg} is satisfied at a feasible point $x$ iff the
Jacobian matrix $\nabla b(x)$ has full row rank and there exists $h \in \mathbb{R}^d$ such that $\nabla b(x) h = 0$ and
$\langle \nabla_x g_i(x, t), h \rangle < 0$ for all $t \in T$ and $i \in I$ such that $g_i(x, t) = 0$.

The penalty function for problem \eqref{probl:NonlinearSemiDefProg} has the form
$$
  \Phi_c(x) = f(x) + c \Big( \sum_{i = 1}^l \max_{t \in T} \{ g_i(x, t), 0 \} + |h(x)| \Big).
$$
For any feasible point $x$ one has
\begin{multline*}
  \mathcal{N}(x) = \Big\{ \sum_{i = 1}^l \int_T \nabla_x g_i(x, t) d \mu_i(t) + \nabla b(x)^T \nu \Bigm| 
  \mu_i \in rca_+(T), \\ 
  \support(\mu_i) \subseteq \{ t \in T \mid g_i(x, t) = 0 \} \enspace \forall i \in I,
  \nu \in \mathbb{R}^s \Big\},
\end{multline*}
where $\support(\mu)$ is the support of a measure $\mu$. We define $\eta(x) = \mathcal{N}(x)$, since it does
not seem possible to somehow reduce the set $\mathcal{N}(x)$ due to the infinite dimensional nature of the problem.

When it comes to numerical methods, it is very difficult to deal with measures $\mu_i \in rca_+(T)$ directly (especially
in the case when the sets $\{ t \in T \mid g_i(x, t) = 0 \}$ have infinite cardinality). Apparently, the general theory
of optimality conditions for cone constrained minimax problems developed in the previous sections cannot overcome this
obstacle for semi-infinite minimax problems. That is why such problems require a special treatment. Our aim is to show
that necessary optimality conditions for semi-infinite minimax problems, including such conditions in terms of cadre and
alternance, can be completely rewritten in terms of discrete measures whose supports consist of at most $d + 1$ points,
which allows one to avoid the use of Radon measures.

To this end, suppose that $x_*$ is a feasible point of problem \eqref{probl:SemiInfProg}, and let there exist
a Lagrange multiplier $\lambda = (\mu_1, \ldots, \mu_l, \nu) \in K^*$ of problem \eqref{probl:SemiInfProg} at $x_*$. We
say that $\lambda$ is \textit{a discrete Lagrange multiplier}, if for any $i \in I$ the measure $\mu_i$ is discrete and
its support consists of at most $d + 1$ points, i.e. $\mu_i = \sum_{j = 1}^{m_i} \lambda_{ij} \delta(t_{ij})$ for some
$t_{ij} \in T$, $\lambda_{ij} \ge 0$, and $m_i \le d + 1$. Here $\delta(t)$ is the Dirac measure of mass one at the
point $t \in T$. If $\lambda$ is a discrete Lagrange multiplier, then one has
$L(x, \lambda) = F(x) + \sum_{i = 1}^l \sum_{j = 1}^{m_i} \lambda_{ij} g_i(x, t_{ij}) 
+ \langle \nu, b(x) \rangle$ for all $x \in \mathbb{R}^d$.

Let us check that necessary optimality conditions for problem \eqref{probl:SemiInfProg} can be expressed in terms of
discrete Lagrange multipliers. Denote $I(x) = \{ i \in I \mid \max_{t \in T} g_i(x, t) = 0 \}$, and let
$T_i(x) = \{ t \in T \mid g_i(x, t) = 0 \}$.

\begin{theorem} \label{thrm:DiscreteLagrangeMultipliers}
Let $x_*$ be a feasible point of problem \eqref{probl:SemiInfProg}. Then the following statements hold true:
\begin{enumerate}
\item{if $x_*$ is a locally optimal solution of problem \eqref{probl:SemiInfProg} at which RCQ holds true, then
there exists a discrete Lagrange multiplier of this problem at $x_*$;
\label{stat:ExistenceDiscreteLagrMult}}

\item{a discrete Lagrange multiplier exists at $x_*$ if and only if for any $i \in I(x_*)$ one can find
$m_i \in \{ 1, \ldots, d + 1 \}$ and $t_{ij} \in T_i(x_*)$, $j \in \{ 1, \ldots, m_i \}$, such that
there exists at $x_*$ a Lagrange multiplier of the discretised problem
\begin{equation} \label{probl:DiscretisedSemiInfProb}
\begin{split}
  &\min_x \: \max_{\omega \in W} f(x, \omega) \\
  &\text{s.t.} \enspace g_i(x, t_{ij}) \le 0, \enspace j \in \{ 1, \ldots, m_i \}, \enspace i \in I(x_*), 
  \enspace b(x) = 0, \enspace x \in A;
\end{split}
\end{equation}
\label{stat:LagrMultViaDiscreteProb}}

\vspace{-5mm}

\item{if $b(\cdot) \equiv 0$, the function $f(\cdot, \omega)$ is convex for any $\omega \in W$, the functions 
$g_i(\cdot, t)$ are convex for any $t \in T$ and $i \in I$, and there exists $x_0 \in A$ such that $g_i(x_0, t) < 0$ for
all $t \in T$ and $i \in I$, then a discrete Lagrange multiplier exists at $x_*$ iff $x_*$ is a globally optimal
solution of problem \eqref{probl:SemiInfProg} iff for any $i \in I(x_*)$ there exist $m_i \in \{ 1, \ldots, d + 1 \}$
and $t_{ij} \in T_i(x_*)$, $j \in \{ 1, \ldots, m_i \}$, such that $x_*$ is a globally optimal solution of problem
\eqref{probl:DiscretisedSemiInfProb}.
\label{stat:DiscreteLagrMult_ConvexCase}}
\end{enumerate}
\end{theorem}

\begin{proof}
\textbf{Part~\ref{stat:ExistenceDiscreteLagrMult}.} Introduce the function
$$
  z(x) = \max\big\{ F(x) - F(x_*), \max_{t \in T} g_1(x, t), \ldots, \max_{t \in T} g_l(x, t) \}.
$$
Observe that $z(x_*) = 0$, and if $z(x) < 0$ for some $x \in A$ such that $b(x) = 0$, then $x$ is a feasible point of
problem \eqref{probl:SemiInfProg} for which $F(x) < F(x_*)$. Hence taking into account the fact that $x_*$ is a locally
optimal solution of problem \eqref{probl:SemiInfProg} one obtains that $x_*$ is a locally optimal solution of 
the problem
\begin{equation} \label{probl:ReducedSemiInfProblem}
  \min\: z(x) \quad \text{subject to} \quad b(x) = 0, \quad x \in A
\end{equation}
as well. Note that this is a constrained minimax problem, since the function $z$ can be written as 
$z(x) = \max_{\omega \in \widetilde{W}} \widetilde{f}(x, \omega)$, where 
$\widetilde{W} = W \cup (T \times \{ 1 \}) \cup \ldots \cup (T \times \{ l \})$,
$\widetilde{f}(x, \omega) = f(x, \omega) - F(x_*)$, if $\omega \in W$, and $\widetilde{f}(x, \omega) = g_i(x, t)$, if
$\omega = (t, i) \in T \times \{ i \}$ for some $i \in I$. 

Recall that by our assumption Robinson's constraint qualification for problem \eqref{probl:SemiInfProg} holds true at
$x_*$, i.e. $0 \in \interior\{ G(x_*) + D G(x_*)( A - x_* ) - K \}$ or, equivalently,
\begin{equation} \label{eq:RCQ_SemiInfProblems}
  0 \in \interior\left\{ \begin{pmatrix} g(x_*, \cdot) + \nabla_x g(x_*, \cdot )h \\ \nabla b(x_*) h \end{pmatrix}
  + \begin{pmatrix} (C_+(T))^l \\ 0_s \end{pmatrix} \Biggm| h \in A - x_* \right\}
\end{equation}
where $g = (g_1, \ldots, g_l)^T$ and $C_+(T) = - C_{-}(T)$ is the cone of nonnegative continuous functions defined 
on $T$. Hence, in particular, one gets that $0 \in \interior\{ \nabla b(x_*) (A - x_*) \}$, that is, RCQ for problem
\eqref{probl:ReducedSemiInfProblem} is satisfied at $x_*$. Consequently, by Theorem~\ref{thrm:NessOptCond} there exists
a Lagrange multiplier of problem \eqref{probl:ReducedSemiInfProblem} at $x_*$, which by
Remark~\ref{rmrk:LagrangeMultViaSubdiff} implies that 
$(\partial z(x_*) + \nabla b(x_*)^T \nu) \cap (- N_A(x_*)) \ne \emptyset$ for some $\nu \in \mathbb{R}^s$, where
$$
  \partial z(x_*) = \co\Big\{\nabla_x f(x_*, \omega),  \nabla_x g_i(x_*, t) \Bigm| 
  \omega \in W(x_*), t \in T_i(x_*), i \in I(x_*) \Big\}.
$$
Hence there exist $v_1 \in \partial F(x_*)$, 
$v_2 \in \co\{ \nabla_x g_i(x_*, t) \mid t \in T_i(x_*), i \in I(x_*) \}$, and $\alpha \in [0, 1]$ such that
$\alpha v_1 + (1 - \alpha) v_2 + \nabla b(x_*)^T \nu \in - N_A(x_*)$. By Carath\'{e}odory's theorem for any 
$i \in I(x_*)$ there exist $m_i \le d + 1$, $t_{ij} \in T_i(x_*)$, and $\alpha_{ij} \ge 0$, 
$j \in \{ 1, \ldots, m_i \}$, such that
$$
  v_2 = \sum_{i \in I(x_*)} \sum_{j = 1}^{m_i} \alpha_{ij} \nabla_x g_i(x_*, t_{ij}), \quad
  \sum_{i \in I(x_*)} \sum_{j = 1}^m \alpha_{ij} = 1.
$$
Let us check that $\alpha \ne 0$. Then putting 
$\mu_i = \sum_{j = 1}^{m_i} (1 - \alpha) (\alpha_{ij} / \alpha) \delta(t_{ij})$ for all $i \in I(x_*)$, $\mu_i = 0$
for $i \in I \setminus I(x_*)$, and $\lambda = (\mu_1, \ldots, \mu_l, \nu / \alpha) \in K^*$ one obtains that
$\langle \lambda, G(x_*) \rangle = 0$,
$$
  \frac{1 - \alpha}{\alpha} v_2  + \frac{1}{\alpha} \nabla b(x_*)^T \nu = 
  \sum_{i = 1}^l \int_T \nabla_x g_i(x, t) d \mu_i(t) + \frac{1}{\alpha} \nabla b(x_*)^T \nu = [D G(x_*)]^* \lambda,
$$
and $(\partial F(x_*) + [D G(x_*)]^* \lambda) \cap (- N_A(x_*)) \ne \emptyset$, which by
Remark~\ref{rmrk:LagrangeMultViaSubdiff} implies that $\lambda$ is a discrete Lagrange multiplier at $x_*$.

Thus, it remains to check that $\alpha \ne 0$. Arguing by reductio ad absrudum suppose that $\alpha = 0$. Then
$v_2 + \nabla b(x_*)^T \nu \in - N_A(x_*)$. Note that from \eqref{eq:RCQ_SemiInfProblems} it follows that 
there exists $h \in A - x_* \subset T_A(x_*)$ such that $\nabla b(x_*) h = 0$ and 
$\langle \nabla_x g_i(x_*, t), h \rangle < 0$ for all $t \in T_i(x_*)$ and $i \in I(x_*)$. Hence by the definition of
$v_2$ one has $\langle v_2 + \nabla b(x_*)^T \nu, h \rangle < 0$, which is impossible, since by our assumption 
$v_2 + \nabla b(x_*)^T \nu \in - N_A(x_*)$. Thus, $\alpha \ne 0$ and the proof of the first part of the theorem is
complete.

\textbf{Part~\ref{stat:LagrMultViaDiscreteProb}.} The validity of this statement follows directly from the definitions
of a discrete Lagrange multiplier and a Lagrange multiplier for problem \eqref{probl:DiscretisedSemiInfProb}.

\textbf{Part~\ref{stat:DiscreteLagrMult_ConvexCase}.} Observe that the assumptions on the functions $b(\cdot)$ and
$g_i(\cdot, t)$ imply that the mapping $G(\cdot)$ is $(-K)$-convex, while the existence of $x_0 \in A$ such that
$g_i(x_0, t) < 0$ for all $t \in T$ and $i \in I$ is equivalent to Slater's condition 
$0 \in \interior\{ G(A) - K \}$ and implies the validity of Slater's conditions for the discritised problem
\eqref{probl:DiscretisedSemiInfProb}.

Suppose that there exists a discrete Lagrange multiplier at $x_*$. Then by the second part of the theorem for any 
$i \in I(x_*)$ one can find $m_i \in \{ 1, \ldots, d + 1 \}$ and $t_{ij} \in T_i(x_*)$, $j \in \{ 1, \ldots, m_i \}$,
such that there exists a Lagrange multiplier of the discretised problem \eqref{probl:DiscretisedSemiInfProb}. Hence
by Theorem~\ref{thrm:OptCond_ConvexCase} the point $x_*$ is a globally optimal solution of problem
\eqref{probl:DiscretisedSemiInfProb}. 

If $x_*$ is a globally optimal solution of the discretised problem \eqref{probl:DiscretisedSemiInfProb}, then $x_*$ is
obviously a globally optimal solution of problem \eqref{probl:SemiInfProg} as well, since the feasible region of
problem \eqref{probl:SemiInfProg} is contained in the feasible region of problem \eqref{probl:DiscretisedSemiInfProb}.

Finally, if $x_*$ is a globally optimal solution of problem \eqref{probl:SemiInfProg}, then taking into account the
fact that in the convex case by \cite[Prp.~2.104]{BonnansShapiro} Slater's condition $0 \in \interior\{ G(A) - K \}$ is
equivalent to RCQ and applying the first part of the theorem one obtains that there exists a discrete Lagrange
multiplier at $x_*$.
\end{proof}

\begin{remark}
Note that from the proof of the theorem above it follows that in the definition of discrete Lagrange multiplier one can
suppose that \textit{the union} of the supports of all measures $\mu_i$ consists of at most $d + 1$ points. Furthermore,
dividing the inclusion $\alpha v_1 + (1 - \alpha) v_2 + \nabla b(x_*)^T \nu \in - N_A(x_*)$ by $\alpha$ one obtains
that 
$$
  \Big( v_1 + \cone\big\{ \nabla_x g_i(x_*, t) \bigm| t \in T_i(x_*), i \in I(x_*) \big\} 
  + \frac{1}{\alpha} \nabla b(x_*)^T \nu \Big) \cap (- N_A(x_*)) \ne \emptyset.
$$
Hence taking into account the fact that any point from the convex conic hull can be expressed as a non-negative linear
combination of $d$ or fewer linearly independent vectors (see, e.g. \cite[Corollary~17.1.2]{Rockafellar}) one can check
that in the definition of discrete Lagrange multiplier it is sufficient to suppose that the union of the supports of 
the measures $\mu_i$ consists of at most $d$ points. \qed
\end{remark}

With the use of the theorem above one can easily obtain convenient necessary optimality conditions for problem
\eqref{probl:SemiInfProg} in terms of cadre and alternance. Let $Z \subset \mathbb{R}^d$ be a set consisting of $d$
linearly independent vectors and let $n_A(x)$ be a nonempty set such that $N_A(x) = \cone n_A(x)$ for any 
$x \in \mathbb{R}^d$.

\begin{definition}
Let $x_*$ be a feasible point of problem \eqref{probl:SemiInfProg} and $p \in \{ 1, \ldots, d + 1 \}$ be fixed. One
says that \textit{a discrete $p$-point alternance} exists at $x_*$, if there exist $k_0 \in \{ 1, \ldots, p \}$, 
$i_0 \in \{ k_0 + 1, \ldots, p \}$, vectors
\begin{gather} \label{eq:DiscreteAlternanceDef}
  V_1, \ldots, V_{k_0} \in \Big\{ \nabla_x f(x_*, \omega) \Bigm| \omega \in W(x_*) \Big\}, \\
  V_{k_0 + 1}, \ldots, V_{i_0} \in \Big\{ \nabla_x g_i(x_*, t) \Bigm| i \in I(x_*), t \in T_i(x_*) \Big\}, \:
  V_{i_0 + 1}, \ldots, V_p \in n_A(x_*), \label{eq:DiscreteAlternanceDef2}
\end{gather}
and vectors $V_{p + 1}, \ldots, V_{d + 1} \in Z$ such that the $d$th-order determinants $\Delta_s$ of the matrices
composed of the columns $V_1, \ldots, V_{s - 1}, V_{s + 1}, \ldots V_{d + 1}$ satisfy the following conditions:
\begin{gather*}
  \Delta_s \ne 0, \quad s \in \{ 1, \ldots, p \}, \quad
  \sign \Delta_s = - \sign \Delta_{s + 1}, \quad s \in \{ 1, \ldots, p - 1 \}, \\
  \Delta_s = 0, \quad s \in \{ p + 1, \ldots d + 1 \}.
\end{gather*}
Any such collection of vectors $\{ V_1, \ldots, V_p \}$ is called \textit{a discrete $p$-point alternance} at $x_*$. Any
discrete $(d + 1)$-point alternance is called \textit{complete}
\end{definition}

Bearing in mind Theorem~\ref{thrm:DiscreteLagrangeMultipliers} and applying Proposition~\ref{prp:AlternanceVsCadre} and
Theorem~\ref{thrm:AlternanceCond} to the discretised problem \eqref{probl:DiscretisedSemiInfProb} one obtains that the
following result holds true. 

\begin{corollary}
Let $x_*$ be a feasible point of problem \eqref{probl:SemiInfProg}. Then the following statements are equivalent:
\begin{enumerate}
\item{a discrete Lagrange multiplier exists at $x_*$;}

\item{a discrete $p$-point alternance exists at $x_*$ for some $p \in \{ 1, \ldots, d + 1 \}$;}

\item{a discrete $p$-point cadre with positive cadre multipliers exists at $x_*$ for some 
$p \in \{ 1, \ldots, d + 1 \}$, that is, there exist $k_0 \in \{ 1, \ldots, p \}$, $i_0 \in \{ k_0 + 1, \ldots, p \}$,
and vectors satisfying \eqref{eq:DiscreteAlternanceDef} and \eqref{eq:DiscreteAlternanceDef2}
such that $\rank([V_1, \ldots, V_p]) = p - 1$ and $\sum_{i = 1}^p \beta_i V_i = 0$ for some $\beta_i > 0$.
}
\end{enumerate}
Furhtermore, if a complete discrete alternance exists at $x_*$, then $x_*$ is a local minimiser of problem
\eqref{probl:SemiInfProg} at which the first order growth condition holds true.
\end{corollary}

\begin{remark}
It should be noted that it is unclear whether first order sufficient optimality condition for problem
\eqref{probl:SemiInfProg} can be rewritten in an equivalent form involving discrete Lagrange multipliers. One can
consider sufficient optimality conditions for the discretised problem \eqref{probl:DiscretisedSemiInfProb}. These
conditions are obviously sufficient optimality conditions for problem \eqref{probl:SemiInfProg}, since the feasible
region of this problem is contained in the feasible region of problem \eqref{probl:DiscretisedSemiInfProb}. However, it
seems that ``abstract'' sufficient optimality conditions for problem $(\mathcal{P})$ rewritten in terms of the
semi-infinite minimax problem are not equivalent to such conditions for the discretised problem. \qed
\end{remark}

\section{Second order optimality conditions for cone constrained minimax problems}
\label{sect:SecondOrderOptCond}

First order information is often insufficient to identify whether a given point is a locally optimal solution of a
minimax problem. For instance, in the case of unconstrained problems first order sufficient optimality conditions
cannot be satisfied, if the set $W(x_*) = \{ \omega \in W \mid F(x_*) = f(x_*, \omega) \}$ consists of less than $d + 1$
points. In such cases one obviously has to use \textit{second} order optimality conditions, whose analysis is the main
goal of this section. To simplify this analysis, we will mainly utilise a standard reformulation of cone constrained
minimax problems as equivalent smooth cone constrained problems and apply well-known second order optimality conditions
for such problems from \cite{Kawasaki,Cominetti,BonComShap98,BonComShap99,BonnansShapiro} to obtain optimality
conditions for minimax problems.

Let us introduce some auxiliary definitions first. Let $(x_*, \lambda_*)$ be a KKT-pair of the problem $(\mathcal{P})$,
that is, $x_*$ is a feasible point of this problem and $\lambda_*$ is a Lagrange multiplier at $x_*$. Then
$(\partial F(x_*) + [D G(x_*)]^* \lambda_*) \cap (- N_A(x_*)) \ne \emptyset$ by
Remark~\ref{rmrk:LagrangeMultViaSubdiff}, which implies that there exists $v \in \partial F(x_*)$ such that 
$\langle v, h \rangle + \langle \lambda_*, D G(x_*) h \rangle \ge 0$ for all $h \in T_A(x_*)$. By definition there exist
$k \in \mathbb{N}$, $\omega_i \in W(x_*)$, and $\alpha_i \ge 0$, $i \in \{ 1, \ldots, k \}$, such 
that $v = \sum_{i = 1}^k \alpha_i \nabla_x f(x, \omega_i)$ and $\sum_{i = 1}^k \alpha_i = 1$. 
Let $\alpha = \sum_{i = 1}^k \alpha_i \delta(\omega_i)$ be the discrete Radon measure on $W$ corresponding to $\alpha_i$
and $\omega_i$. Then
$$
  \left\langle \int_W \nabla_x f(x, \omega) d \alpha(\omega), h \right\rangle + 
  \langle \lambda_*, D G(x_*) h \rangle \ge 0 \quad \forall h \in T_A(x_*),
  \quad \alpha(W) = 1.
$$
Denote by $\alpha(x_*, \lambda_*)$ the set of all Radon measures $\alpha \in rca_+(W)$ satisfying the conditions above
and such that $\support(\alpha) \subset W(x_*)$. It is easily seen that this set is convex, bounded and weak${}^*$
closed, i.e. $\alpha(x_*, \lambda_*)$ is a weak${}^*$ compact set. Any measure $\alpha \in \alpha(x_*, \lambda_*)$ is
called \textit{a Danskin-Demyanov multiplier} corresponding to the KKT-pair $(x_*, \lambda_*)$ (see, e.g.
\cite[Sect.~2.1.1]{Polak}). Note that in the case of discrete minimax problems, i.e. when $W = \{ 1, \ldots, m \}$ (or
in the case when the set $W(x_*)$ consists of a finite number of points), the set of Danskin-Demyanov multipliers
$\alpha(x_*, \lambda_*)$ is simply a closed convex subset of the standard (probability) simplex in $\mathbb{R}^m$.

Denote by $\mathcal{L}(x, \lambda, \alpha) = \int_W f(x, \omega) d \alpha(\omega) + \langle \lambda, G(x) \rangle$ 
\textit{the integral Lagrangian} for the problem $(\mathcal{P})$, where $x \in \mathbb{R}^d$, $\lambda \in Y^*$, and 
$\alpha \in rca_+(W)$. Note that $\alpha_*$ is a Danskin-Demyanov multiplier corresponding to $(x_*, \lambda_*)$ iff
$\langle \nabla_x \mathcal{L}(x_*, \lambda_*, \alpha_*), h \rangle \ge 0$ for all $h \in T_A(x_*)$, 
$\support(\alpha) \subset W(x_*)$, and $\alpha(W) = 1$.

Let $S \subset Y$ be a given set, and $y \in S$ and $h \in Y$ be fixed. Recall that \textit{the outer second order
tangent set} to the set $S$ at the point $y$ in the direction $h$, denoted by $T_S^2(x, h)$, consists of all those
vectors $w \in Y$ for which one can find a sequence $\{ t_n \} \subset (0, + \infty)$ such that $\lim t_n = 0$ and
$\dist(x + t_n h + 0.5 t_n^2 w, S) = o(t_n^2)$. See \cite[Sect.~3.2.1]{BonnansShapiro} for a detailed treatment of
second-order tangent sets. Here we only note that the second order tangent set $T_S^2(x, h)$ might be nonconvex even in
the case when the set $S$ is convex.

For any $\lambda \in Y^*$ denote by $\sigma(\lambda, S) = \sup_{y \in S} \langle \lambda, y \rangle$ the support
function of the set $S$. Also, for any feasible point $x_*$ of the problem $(\mathcal{P})$ denote by
$\Lambda(x_*)$ the set of all Lagrange multipliers of $(\mathcal{P})$ at $x_*$. Finally, for any feasible point 
$x_*$ of the problem $(\mathcal{P})$ denote by
$$
  C(x_*) = \Big\{ h \in T_A(x_*) \Bigm| D G(x_*) h \in T_K(G(x_*)), \enspace F'(x_*, h) \le 0 \Big\}
$$
\textit{the critical cone} at the point $x_*$. Observe that if $\Lambda(x_*) \ne \emptyset$, then by definition for any
$\lambda \in \Lambda(x_*)$ one has 
$[L(\cdot, \lambda_*)]'(x_*, h) = F'(x_*, h) + \langle \lambda_*, D G(x_*) h \rangle \ge 0$ for any 
$h \in T_A(x_*)$, which implies that
$$
  C(x_*) = \Big\{ h \in T_A(x_*) \Bigm| D G(x_*) h \in T_K(G(x_*)), \enspace F'(x_*, h) = 0 \Big\},
$$
since $\langle\lambda_*, D G(x_*) h \rangle \le 0$ for any $h$ such that $D G(x_*) h \in T_K(G(x_*))$ 
(see Remark~\ref{rmrk:SuffOptCond_Lagrangian}). Moreover, one also has
\begin{equation} \label{eq:CriticalConeViaLagrangian}
\begin{split}
  C(x_*) = \Big\{ h \in T_A(x_*) \Bigm| &D G(x_*) h \in T_K(G(x_*)), \\
  &\langle \lambda_*, D G(x_*) h \rangle = 0, \enspace [L(\cdot, \lambda_*)]'(x_*, h) = 0 \Big\}
\end{split}
\end{equation}
for any $\lambda_* \in \Lambda(x_*)$.

For the sake of simplicity, we derive second order necessary optimality conditions only in the case when
$x_* \in \interior A$ and the set $W(x_*)$ is discrete. Arguing in the same way one can derive second order conditions
in the general case. However, it should be noted that in the general case these conditions are very cumbersome, since
they involve complicated expressions depending on the second order tangent sets to $A$ and $C_{-}(W)$.

In this section we suppose that the mapping $G$ is twice continuously Fr\'{e}chet differentiable in a neighbourhood of
a given point $x_*$, the function $f(x, \omega)$ is twice differentiable in $x$ in a neighbourhood $\mathcal{O}(x_*)$ of
$x_*$ for any $\omega \in W$, and the function $\nabla^2_{xx} f(\cdot)$ is continuous on $\mathcal{O}(x_*) \times W$.

\begin{theorem} \label{thrm:2Order_NessOptCond}
Let $W = \{ 1, \ldots, m \}$, $f(x, i) = f_i(x)$ for any $i \in W$, and $x_* \in \interior A$ be a locally optimal
solution of the problem $(\mathcal{P})$ such that RCQ holds true  at $x_*$. Then for any $h \in C(x_*)$ and for any
convex set $\mathcal{T}(h) \subseteq T_K^2(G(x_*), D G(x_*) h)$ one has
$$
  \sup_{\lambda \in \Lambda(x_*)} \Big\{ 
  \sup_{\alpha \in \alpha(x_*, \lambda)} \big\langle h, \nabla^2_{xx} \mathcal{L}(x_*, \lambda, \alpha) h \big\rangle
  - \sigma(\lambda, \mathcal{T}(h)) \Big\} \ge 0.
$$
\end{theorem}

\begin{proof}
From the facts that $x_*$ is a locally optimal solution of $(\mathcal{P})$ and $x_* \in \interior A$ it
follows that $(x_*, F(x_*))$ is a locally optimal solution of the problem
$$
  \min_{(x, z)} z \quad \text{subject to} \quad
  f(x, \omega) - z \le 0 \quad \omega \in W, \quad G(x) \in K.
$$
This problem can be rewritten as the cone constrained problem
\begin{equation} \label{prob:MinMaxViaConeConstr}
  \min \: \widehat{f}(x, z) \quad \text{subject to} \quad \widehat{G}(x, z) \in \widehat{K},
\end{equation}
where $\widehat{f}(x, z) = z$, $\widehat{Y} = K \times \mathbb{R}^m$, 
$\widehat{G}(x, z) = (G(x), f_1(x) - z, \ldots, f_m(x) - z)$, and $\widehat{K} = K \times \mathbb{R}_{-}^m$, where
$\mathbb{R}_{-} = (- \infty, 0]$. Our aim is to prove the theorem by reformulating second order optimality conditions
for problem \eqref{prob:MinMaxViaConeConstr} in terms of the problem $(\mathcal{P})$.

For any $x \in \mathbb{R}^d$, $z \in \mathbb{R}$, $\lambda \in Y^*$ and $\alpha \in \mathbb{R}^m$ denote by
$$
  \mathcal{L}_0(x, z, \lambda, \alpha) = \widehat{f}(x, z) + \langle (\lambda, \alpha), \widehat{G}(x, z) \rangle
  = z + \sum_{i = 1}^m \alpha^{(i)} \big( f_i(x) - z \big) + \langle \lambda, G(x) \rangle
$$
the Lagrangian for cone constrained problem \eqref{prob:MinMaxViaConeConstr}. Observe that 
$\mathcal{L}_0(x, z, \lambda, \alpha) = \mathcal{L}(x, \lambda, \alpha) + (1 - \sum_{i = 1}^m \alpha^{(i)}) z$. One can
easily see that $(\lambda_*, \alpha_*)$ is a Lagrange multiplier of problem \eqref{prob:MinMaxViaConeConstr} at $(x_*,
F(x_*))$ iff $\lambda_*$ is a Lagrange multiplier of the problem $(\mathcal{P})$ at $x_*$ and $\alpha_*$ is a
Danskin-Demyanov multiplier corresponding to $(x_*, \lambda_*)$. 

Let $z_* = F(x_*)$. Observe that
\begin{multline*}
  \widehat{G}(x_*, z_*) + D \widehat{G}(x_*, z_*)\big( \mathbb{R}^d \times \mathbb{R} \big) - \widehat{K} \\
  = \left\{ \begin{pmatrix} G(x_*) \\ f(x_*) - z_* \mathbf{1}_m \end{pmatrix}
  + \begin{pmatrix} D G(x_*) h_x \\ \nabla f(x_*) h_x  - h_z \mathbf{1}_m \end{pmatrix}
  - \begin{pmatrix} K \\ \mathbb{R}_{-}^m \end{pmatrix} \biggm| (h_x, h_z) \in \mathbb{R}^d \times \mathbb{R}
  \right\}.
\end{multline*}
where $f(\cdot) = (f_1(\cdot), \ldots, f_m(\cdot))^T \in \mathbb{R}^m$ and 
$\mathbf{1}_m = (1, \ldots, 1)^T \in \mathbb{R}^m$. Taking into account the fact that RCQ for the problem
$(\mathcal{P})$ is satisfied at $x_*$ one can easily check that RCQ for problem \eqref{prob:MinMaxViaConeConstr} is
satisfied at $(x_*, z_*)$. Therefore, by \cite[Thrm.~3.45]{BonnansShapiro} the second order necessary optimality
conditions for problem \eqref{prob:MinMaxViaConeConstr} are satisfied at $(x_*, 0)$, that is, 
for every $\widehat{h} = (h_x, h_z) \in C(x_*, z_*)$, where
\begin{multline*}
  C(x_*, z_*) = \Big\{ (h_x, h_z) \in \mathbb{R}^d \times \mathbb{R} \Bigm| 
  D G(x_*) h_x \in T_K(G(x_*)), \\
  \nabla f(x_*) h_x  - h_z \mathbf{1}_m \in T_{\mathbb{R}^m_{-}}(f(x_*) - z_* \mathbf{1}_m), 
  \enspace h_z = 0 \Big\}
\end{multline*}
and any convex set 
$\mathcal{T}(\widehat{h}) \subseteq T^2_{\widehat{K}}(\widehat{G}(x_*, z_*), D \widehat{G}(x_*, z_*) \widehat{h})$ one
has
$$
  \sup\Big\{ \langle h_x, \nabla^2_{xx} \mathcal{L}(x_*, \lambda, \alpha) h_x \rangle
  - \sigma\big( (\lambda, \alpha), \mathcal{T}(\widehat{h}) \big) \Big\}
$$
where the supremum is taken over all Lagrange multipliers $(\lambda, \alpha)$ of problem
\eqref{prob:MinMaxViaConeConstr} at $(x_*, z_*)$. As was noted above, for any such $(\lambda, \alpha)$ one has
$\lambda \in \Lambda(x_*)$ and $\alpha \in \alpha(x_*, \lambda)$. Furthermore, note that 
\begin{align*}
  C(x_*, z_*) &= \Big\{ (h, 0) \in \mathbb{R}^d \times \mathbb{R} \Bigm| D G(x_*) h \in T_K(G(x_*)), \enspace
  \langle \nabla f_i(x_*), h \rangle \le 0 \quad \forall i \in W(x_*) \Big\} \\
  &= C(x_*) \times \{ 0 \}.
\end{align*}
Therefore for every $h \in C(x_*)$ and for any convex subset $\mathcal{T}(h, 0)$ of the second order tangent set
$T^2_{\widehat{K}}(\widehat{G}(x_*, z_*), D \widehat{G}(x_*, z_*)(h, 0))$ one has
$$
  \sup_{\lambda \in \Lambda(x_*)} \Big\{ 
  \sup_{\alpha \in \alpha(x_*, \lambda)} \big\langle h, \nabla^2_{xx} \mathcal{L}(x_*, \lambda, \alpha) h \big\rangle
  - \sigma\big( (\lambda, \alpha), \mathcal{T}(h, 0) \big) \Big\} \ge 0.
$$
It remains to note that for every $h \in C(x_*)$ and for any convex set 
$\mathcal{T}(h) \subseteq T_K^2(G(x_*), D G(x_*) h)$ one has
$\mathcal{T}(h) \times \{ 0 \} \subseteq T^2_{\widehat{K}}(\widehat{G}(x_*, z_*), D \widehat{G}(x_*, z_*)(h, 0))$, since
for all $w \in \mathcal{T}(h)$ and for any sequence $\{ t_n \} \subset (0, + \infty)$ such that $\lim t_n = 0$ and 
$\dist(G(x_*) + t_n D G(x_*) h + 0.5 t_n^2 w, K) = o(t_n^2)$ (note that at least one such sequence exists due to the
fact that $\mathcal{T}(h) \subseteq T_K^2(G(x_*), D G(x_*) h)$) one has
\begin{multline*}
  \dist\Big( \widehat{G}(x_*, z_*) + t_n D \widehat{G}(x_*)(h, 0) + \frac{1}{2} t_n^2 (w, 0), \widehat{K} \Big) \\
  \le \dist\Big( G(x_*) + t_n D G(x_*) h + \frac{1}{2} t_n^2 w, K \Big)
  + \dist\big( f(x_*) - z_* \mathbf{1}_m + t_n \nabla f(x_*) h, \mathbb{R}^m_{-} \big) = o(t_n^2).
\end{multline*}
Here we used the fact that $\dist\big( f(x_*) - z_* \mathbf{1}_m + t_n \nabla f(x_*) h, \mathbb{R}^m_{-} \big) = 0$ for
any sufficiently large $n$, since $h \in C(x_*)$ and by the definition of critical cone one has 
$\langle \nabla f_i(x_*), h \rangle \le 0$ for any $i \in W(x_*)$.
\end{proof}

Almost literally repeating the proof of \cite[Prp.~3.46]{BonnansShapiro} one can prove the following useful corollary to
the theorem above. For the sake of completeness, we outline its proof.

\begin{corollary}
Let all assumptions of the previous theorem be valid and suppose that there exists a unique Lagrange
multiplier at $x_*$, i.e. $\Lambda(x_*) = \{ \lambda_* \}$ for some $\lambda_* \in K^*$. Then for any 
$h \in C(x_*)$ one has
$$
  \sup_{\alpha \in \alpha(x_*, \lambda_*)} \big\langle h, \nabla^2_{xx} \mathcal{L}(x_*, \lambda_*, \alpha) \big\rangle
  - \sigma\big( \lambda_*, T^2_K(G(x_*), D G(x_*) h) \big) \ge 0
$$
\end{corollary}

\begin{proof}
Let $\Sigma$ be the set consisting of all sequences $\sigma = \{ t_n \} \subset (0, + \infty)$ such that 
$\lim t_n = 0$. For any $\sigma \in \Sigma$ and $h \in C(x_*)$ denote by $\mathcal{T}_{\sigma}(h)$ the set of all those
vectors $w \in Y$ for which $\dist(G(x_*) + t_n D G(x_*) h + 0.5 t_n^2 w, K) = o(t_n^2)$. Observe that the set
$\mathcal{T}_{\sigma}(h)$ is convex, since for any $n \in \mathbb{N}$ the
function $w \mapsto \dist(G(x_*) + t_n D G(x_*) h + 0.5 t_n^2 w, K)$ is convex. Furthermore, one has
$\mathcal{T}_{\sigma}(h) \subseteq T_K^2(G(x_*), D G(x_*) h)$. Hence by Theorem~\ref{thrm:2Order_NessOptCond} for any 
$h \in C(x_*)$ the following inequality holds true:
$$
  \inf_{\sigma \in \Sigma} \Big\{ \sup_{\alpha \in \alpha(x_*, \lambda_*)} 
  \big\langle h, \nabla^2_{xx} \mathcal{L}(x_*, \lambda_*, \alpha) h \big\rangle
  - \sigma(\lambda_*, \mathcal{T}_{\sigma}(h)) \Big\} \ge 0.
$$
It remains to note that
$$
  \inf_{\sigma \in \Sigma} \big( - \sigma(\lambda_*, \mathcal{T}_{\sigma}(h)) \big)
  = - \sup_{\sigma \in \Sigma} \sup_{w \in \mathcal{T}_{\sigma}(h)} \langle \lambda_*, w \rangle
  = - \sigma\big( \lambda_*, T^2_K(G(x_*), D G(x_*) h) \big), 
$$
since $T^2_K(G(x_*), D G(x_*) h) = \bigcup_{\sigma \in \Sigma} \mathcal{T}_{\sigma}(h)$ by definition.
\end{proof}

Let us briefly discuss optimality conditions from Theorem~\ref{thrm:2Order_NessOptCond}. Firstly, note that they mainly
differ from classical optimality conditions by the presence of the sigma term $\sigma(\lambda, \mathcal{T}(h))$, which,
in a sense, represents a contribution of the curvature of the cone $K$ at the point $G(x_*)$ to optimality conditions.
This term is a specific feature of second order optimality conditions for cone constrained optimisation
problems \cite{Kawasaki,Cominetti,BonComShap98,BonComShap99,BonnansShapiro}. See
\cite{BonnansShapiro,BonnansRamirez,Shapiro2009} for explicit expressions for the critical cone $C(x_*)$, the second
order tangent set $T^2_K(G(x_*), D G(x_*)h)$, and the sigma term $\sigma(\lambda, \mathcal{T}(h))$ in various particular
cases.

Secondly, it should be pointed out that $\sigma(\lambda, \mathcal{T}(h)) \le 0$ for all $\lambda \in \Lambda(x_*)$ and
$h \in C(x_*)$. Furthermore, if $0 \in T^2_K(G(x_*), D G(x_*)h)$ (in particular, if the cone $K$ is polyhedral), then 
$\sigma(\lambda, \mathcal{T}(h)) = 0$ for all $\lambda \in \Lambda(x_*)$ and $h \in C(x_*)$
(see~\cite[pp.~177--178]{BonnansShapiro}). In this case, the optimality conditions from
Theorem~\ref{thrm:2Order_NessOptCond} take the more traditional form:
$$
  \sup_{\lambda \in \Lambda(x_*)} \sup_{\alpha \in \alpha(x_*, \lambda)} 
  \big\langle h, \nabla^2_{xx} \mathcal{L}(x_*, \lambda, \alpha) h \big\rangle \ge 0
  \quad \forall h \in C(x_*).
$$
As was noted in the proof of Theorem~\ref{thrm:2Order_NessOptCond}, the set
$\{ (\lambda, \alpha) \mid \lambda \in \Lambda(x_*), \alpha \in \alpha(x_*, \lambda_*) \}$
coincides with the set of Lagrange multipliers of problem \eqref{prob:MinMaxViaConeConstr} at the point $(x_*, F(x_*))$.
Consequently, this set is convex and weak${}^*$ compact, since RCQ for problem \eqref{prob:MinMaxViaConeConstr} holds
at $(x_*, F(x_*))$. It is easily seen that the function
$(\lambda, \alpha) \mapsto \langle h, \nabla^2_{xx} \mathcal{L}(x_*, \lambda, \alpha) h \rangle$ is weak${}^*$
continuous. Therefore, if $0 \in T^2_K(G(x_*), D G(x_*)h)$ (in particular, if the cone $K$ is polyhedral), then under
the assumptions of Theorem~\ref{thrm:2Order_NessOptCond} for any $h \in C(x_*)$ one can find $\lambda \in \Lambda(x_*)$
and $\alpha \in \alpha(x_*, \lambda)$ such that 
$\langle h, \nabla^2_{xx} \mathcal{L}(x_*, \lambda, \alpha) h \rangle \ge 0$.

Now we turn to second order sufficient optimality conditions. Similar to the case of first order optimality conditions,
we study second order sufficient optimality conditions in the context of second order growth condition. Recall that
\textit{the second order growth condition} (for the problem $(\mathcal{P})$) is said to be satisfied at a
feasible point $x_*$ of $(\mathcal{P})$, if there exist $\rho > 0$ and a neighbourhood $\mathcal{O}(x_*)$ of
$x_*$ such that $F(x) \ge F(x_*) + \rho | x - x_* |^2$ for any $x \in \mathcal{O}(x_*) \cap \Omega$, where $\Omega$ is
the feasible region of $(\mathcal{P})$.

We start with simple sufficient conditions that do not involve the sigma term.

\begin{theorem}
Let $x_* \in \interior A$ be a feasible point of the problem $(\mathcal{P})$ such that $\Lambda(x_*) \ne \emptyset$ and
for any $h \in C(x_*) \setminus \{ 0 \}$ one can find $\lambda \in \Lambda(x_*)$ and $\alpha \in \alpha(x_*, \lambda)$
such that $\langle h, \nabla^2_{xx} \mathcal{L}(x_*, \lambda, \alpha) h \rangle > 0$. Then $x_*$ is a locally optimal
solution of the problem $(\mathcal{P})$ at which the second order growth condition holds true.
\end{theorem}

\begin{proof}
Consider the following smooth cone constrained optimisation problem:
\begin{equation} \label{prob:MinMaxViaConeConstr_Contin}
  \min_{(x, z)} z \quad \text{subject to} \quad f(x, \omega) - z \le 0 \quad \omega \in W, 
  \quad G(x) \in K, \quad z \in \mathbb{R}.
\end{equation}
Let us check that sufficient optimality condition for this problem hold true at the point $(x_*, F(x_*))$. Indeed, 
the Lagrangian for problem \eqref{prob:MinMaxViaConeConstr_Contin} has the form
$$
  \mathcal{L}_0(x, z, \lambda, \alpha) = z + \int_W \big( f(x, \omega) - z \big) \, d \alpha(\omega)
  + \langle \lambda, G(x) \rangle
$$
for any $\lambda \in K^*$ and $\alpha \in rca_+(W)$. As was noted in the proof of Theorem~\ref{thrm:2Order_NessOptCond},
the critical cone for problem \eqref{prob:MinMaxViaConeConstr_Contin} at $(x_*, F(x_*))$ has the form 
$C(x_*, F(x_*)) = C(x_*) \times \{ 0 \}$. Therefore by our assumptions for any 
$\widehat{h} = (h, 0) \in C(x_*, F(x_*))$, $\widehat{h} \ne 0$, one can find $\lambda \in \Lambda(x_*)$ and 
$\alpha \in \alpha(x_*, \lambda_*)$ such that
$$
  \big\langle \widehat{h}, \nabla^2_{(x, z)(x, z)} \mathcal{L}_0(x_*, F(x_*), \lambda, \alpha) \widehat{h} \big\rangle
  = \big\langle h, \nabla^2_{xx} \mathcal{L}(x_*, \lambda, \alpha) h \big\rangle > 0
$$
As was pointed out in the proof of Theorem~\ref{thrm:2Order_NessOptCond}, the pair $(\lambda, \alpha)$ is a Lagrange
multiplier of problem \eqref{prob:MinMaxViaConeConstr_Contin} at $(x_*, F(x_*))$. Thus, one can conclude that the second
order sufficient optimality condition for problem \eqref{prob:MinMaxViaConeConstr_Contin} holds true at $x_*$, which by 
\cite[Thrm.~3.63]{BonnansShapiro} implies that $(x_*, F(x_*))$ is a locally optimal solution of
\eqref{prob:MinMaxViaConeConstr_Contin} at which the second order growth condition holds true. Thus, by definition there
exist $\rho > 0$ and $\varepsilon > 0$ such that $z \ge F(x_*) + \rho (|x - x_*|^2 + |z - F(x_*)|^2)$ for all 
$x \in B(x_*, \varepsilon)$ and $z \in \mathbb{R}$ such that $|z - F(x_*)| < \varepsilon$, $F(x) \le z$, and 
$G(x) \in K$. Note that the function $F(\cdot) = \max_{\omega \in W} f(\cdot, \omega)$ is continuous, since by our
assumptions the space $W$ is compact and the function $f$ is continuous. Consequently, there exists 
$r \in (0, \varepsilon)$ such that $|F(x) - F(x_*)| < \varepsilon$ for all $x \in B(x_*, r)$. Therefore, putting 
$z = F(x)$ one obtains that $F(x) \ge F(x_*) + \rho |x - x_*|^2$ for all $x \in B(x_*, r)$ such that $G(x) \in K$, that
is, $x_*$ is a locally optimal solution of the problem $(\mathcal{P})$ at which the second order growth condition holds
true.
\end{proof}

In the case when the space $Y$ is finite dimensional and the cone $K$ is second order regular one can strengthen
the previous theorem and obtain simple sufficient optimality conditions involving the sigma term. Recall that the cone
$K$ is said to be \textit{second order regular} at a point $y \in K$, if the following two conditions are satisfied:
\begin{enumerate}
\item{for any $h \in T_K(y)$ and any sequence $\{ y_n \} \subset K$ of the form $y_n = y + t_n h + 0.5 t_n^2 w_n$ where
$t_n > 0$ for all $n \in \mathbb{N}$, $\lim t_n = 0$, and $\lim t_n w_n = 0$ one has 
$\lim \dist(w_n, T_K^2(y, h)) = 0$;}

\item{$T^2_K(y, h) = \{ w \in Y \mid \dist(x + th + 0.5 t^2 w, K) = o(t^2), t \ge 0 \}$ for any $h \in Y$.}
\end{enumerate}
We say that the cone $K$ is second order regular, if it is second order regular at every point $y \in K$. 

For more details on second order regular sets see \cite{BonComShap98,BonComShap99} and
\cite[Sect.~3.3.3]{BonnansShapiro}. Here we only mention that the cone $\mathbb{S}^l_{-}$ of negative semidefinite
matrices is second order regular (see~\cite[p.~474]{BonnansShapiro}) and the second order cone is second order regular
by \cite[Prp.~3.136]{BonnansShapiro} and \cite[Lemma~15]{BonnansRamirez}.

Below we do not assume that $x_* \in \interior A$, but avoid the usage of the second order tangent set to the set $A$
for the sake of simplicity and due to the fact that we are mainly interested in the case when the set $A$ is polyhedral.

\begin{theorem}
Let $Y$ be a finite dimensional Hilbert space, the cone $K$ be second order regular, and $(x_*, \lambda_*)$ be a
KKT-pair of the problem $(\mathcal{P})$ such that the restriction of the function 
$\sigma(\lambda_*, T^2_K(G(x_*), \cdot))$ to its effective domain is upper semicontinuous. Suppose also that
\begin{equation} \label{eq:2OrderSuffCond_SigmaTerm}
  \sup_{\alpha \in \alpha(x_*, \lambda_*)} 
  \big\langle h, \nabla^2_{xx} \mathcal{L}(x_*, \lambda_*, \alpha) h \big\rangle
  - \sigma\big( \lambda_*, T^2_K(G(x_*), D G(x_*) h) \big) > 0
\end{equation}
for all $h \in C(x_*) \setminus \{ 0 \}$. Then $x_*$ is a locally optimal solution of the problem $(\mathcal{P})$ at
which the second order growth condition holds true.
\end{theorem}

\begin{proof}
Introduce the Rockafellar-Wets augmented Lagrangian
\begin{equation} \label{eq:RockWetsAugmLagr}
  \mathscr{L}(x, \lambda, c) = F(x) + \Phi(G(x), \lambda, c), \quad
  \Phi(y, \lambda, c) = \inf_{z \in K - y} \big\{ - \langle \lambda, z \rangle + c \| z \|^2 \big\}
\end{equation}
for the problem $(\mathcal{P})$ (see~\cite{RockafellarWets,ShapiroSun,Dolgopolik_AugmLagr}), where 
$\langle \cdot, \cdot \rangle$ is the inner product in $Y$ and $c \ge 0$ is the penalty parameter. It is easily seen
that 
\begin{equation} \label{eq:AugmLagrangian}
  \Phi(y, \lambda, c) = c \big( \dist(y + (2c)^{-1} \lambda, K) \big)^2 - \frac{1}{4c} \| \lambda \|^2.
\end{equation}
Let us compute a second order expansion of the function $x \mapsto \mathscr{L}(x, \lambda, c)$. 

Denote $\delta(y) = \dist(y, K)^2$. By a generalisation of the Danskin-Demyanov theorem
\cite[Thrm.~4.13]{BonnansShapiro} the function $\delta(\cdot)$ is continuously Fr\'{e}chet differentiable and 
$D \delta(y) = 2(y - P_K(y))$, where $P_K$ is the projection of $y$ onto $K$ (note that the projection exists,
since $Y$ is finite dimensional). Hence by the chain rule the function $x \mapsto \Phi(G(x), \lambda, c)$ is
continuously Fr\'{e}chet differentiable and 
$$
  D_x \Phi(G(x), \lambda, c) h 
  = 2c \Big\langle G(x) + (2c)^{-1} \lambda - P_K(G(x) + (2c)^{-1} \lambda), D G(x) h \Big\rangle
$$
for all $h \in \mathbb{R}^d$. To simplify this expression in the case $x = x_*$ and $\lambda = \lambda_*$ note that
$$
  \left\langle z_* - G(x_*) - \frac{1}{2c} \lambda_*, z - z_* \right\rangle \ge 0 \quad \forall z \in K,
$$
if $z_* = G(x_*)$ (recall that $\langle \lambda_*, G(x_*) \rangle = 0$ and $\lambda_* \in K^*$ by the definition of
KKT-point). Thus, the point $z = G(x_*)$ satisfies the necessary and sufficient optimality conditions for the convex
problem
$$
  \min\: \| z - G(x_*) - (2c)^{-1} \lambda_* \|^2 \quad \text{subject to} \quad z \in K,
$$
that is, $P_K(G(x_*) + (2c)^{-1} \lambda_*) = G(x_*)$. Consequently, for any $c > 0$ one has
$D_x \Phi(G(x_*), \lambda_*, c) = [D G(x_*)]^* \lambda_*$.

Recall that the cone $K$ is second order regular and the space $Y$ is finite dimensional. Therefore by
\cite[Thrm.~4.133]{BonnansShapiro} (see also \cite[Thrm.~3.1]{Shapiro2016}) for all $y, v \in Y$ there
exists the second-order Hadamard directional derivative
$$
  \delta''(y; v) := \lim_{[v', t] \to [v, +0]} 
  \frac{\delta(y + t v') - \delta(y) - t D \delta(y) v'}{\frac{1}{2} t^2}
$$
and it has the form
\begin{equation} \label{eq:SecondOrderHadamardDeriv}
  \delta''(y; v) = \min_{z \in \mathscr{C}(y)}
  \Big[ 2 \| v - z \|^2 - 2 \sigma(y - P_K(y), T^2_K(P_K(y), z) \Big],
\end{equation}
where $\mathscr{C}(y) = \{ z \in T_K(P_K(y)) \mid \langle y - P_K(y), z \rangle = 0 \}$. Bearing in mind the definition
of the second-order Hadamard directional derivative one can easily check that the function $\delta''(y, \cdot)$ is
continuous and positively homogeneous of degree two. Hence taking into account the definition of this derivative one
can easily check that for any linear operator $T \colon \mathbb{R}^d \to Y$ one has
$$
  \delta(y + T h + o(|h|)) = \delta(y) + D \delta(y) \Big( T h + o(|h|) \Big) 
  + \frac{1}{2} \delta''(y; T h) + o(|h|^2 ).
$$
Consequently, putting $y = G(x_*) + (2c)^{-1} \lambda_*$ and $T h = D G(x_*) h$, taking into account the fact that
$D \delta(y) = c^{-1} \lambda_*$, and utilising the second order expansion
$$
  G(x_* + h) = G(x_*) + D G(x_*) h + \frac{1}{2} D^2 G(x_*)(h, h) + o(|h|^2)
$$
one obtains that
\begin{align*}
  \Phi(G(x_* + h), \lambda_*, c) &= \Phi(G(x_*), \lambda_*, c) + \langle \lambda_*, D G(x_*) h \rangle
  + \frac{1}{2} \langle \lambda_*, D^2 G(x_*)(h, h) \rangle \\
  &+ \frac{c}{2} \delta''\Big(G(x_*) + \frac{1}{2c} \lambda_*; D G(x_*) h \Big) + o(|h|^2).
\end{align*}
Hence with the use of the well-known second-order expansion for the max-function of the form
$$
  F(x_* + h) - F(x_*) = \max_{\omega \in W} \Big( f(x_*, \omega) - F(x_*) 
  + \langle \nabla_x f(x_*, \omega), h \rangle 
  + \frac{1}{2} \langle h, \nabla_{xx}^{2} f(x_*, \omega) h \rangle \Big) + o(|h|^2)
$$
one finally gets that for any $c \ge 0$ there exists $r_c > 0$ such that for all $h \in B(0, r_c)$ one has
\begin{multline} \label{eq:AugmLagr_2OrderExpansion}
  \Big| \mathscr{L}(x_* + h, \lambda_*, c) - \mathscr{L}(x_*, \lambda_*, c)
  - \max_{\omega \in W} \Big( f(x_*, \omega) - F(x_*) + \langle \nabla_x f(x_*, \omega), h \rangle 
  + \frac{1}{2} \langle h, \nabla_{xx}^{2} f(x_*, \omega) h \rangle \Big) \\
  - \langle \lambda_*, D G(x_*) h \rangle
  - \frac{1}{2} \langle \lambda_*, D^2 G(x_*)(h, h) \rangle - \frac{1}{2} \omega_c(h) \Big| 
  \le \frac{1}{c} |h|^2,
\end{multline}
where
\begin{align*}
  \omega_c(h) &= c \delta''\Big( G(x_*) + (2c)^{-1} \lambda_*), D G(x_*) h) \Big) \\
  &= \min_{z \in C_0(x_*, \lambda_*)}
  \Big[ 2 c \| D G(x_*) h - z \|^2 - \sigma\big( \lambda_*, T^2_K(G(x_*), z) \big) \Big]
\end{align*}
and $C_0(x_*, \lambda_*) = \{ z \in T_K(G(x_*)) \mid \langle \lambda_*, z \rangle = 0 \}$
(see \eqref{eq:SecondOrderHadamardDeriv}). By \cite[formula~(3.63)]{BonnansShapiro} one has
$T^2_K(G(x_*), z) \subseteq T_{T_K(G(x_*))}(z)$ for all $z \in Y$. Note also that the cone $T_K(G(x_*))$ is convex,
since $K$ is a convex cone. Therefore 
$$
  T_K(G(x_*)) = \cl \Big[ \bigcup_{t \ge 0} t\big( K - G(x_*) \big) \Big], \enspace
  T_{T_K(G(x_*)}(z) = \cl \Big[ \bigcup_{t \ge 0} t\big( T_K(G(x_*)) - z \big) \Big]
$$
(see, e.g. \cite[Prp.~2.55]{BonnansShapiro}). Hence bearing in mind the facts that $\lambda_* \in K^*$ and
$\langle \lambda_*, G(x_*) \rangle = 0$ one obtains that $\langle \lambda_*, y \rangle \le 0$ for all
$y \in T_K(G(x_*))$, which implies that $\langle \lambda_*, y \rangle \le 0$ for any
$y \in T^2_K(G(x_*), z) \subseteq T_{T_K(G(x_*)}(z)$ and all $z \in C_0(x_*, \lambda_*)$. Consequently,
$\sigma(\lambda_*, T^2_K(G(x_*), z)) \le 0$ for all $z \in C_0(x_*, \lambda_*)$. Recall also that 
the restriction of the function $z \mapsto \sigma(\lambda_* , T^2_K(G(x_*), z))$ to its effective domain is upper
semicontinuous by our assumption. Therefore $\lim_{c \to + \infty} \omega_c(h) = + \infty$, if
$D G(x_*) h \notin T_K(G(x_*))$ or $\langle \lambda_*, D G(x_*) h \rangle \ne 0$, and
\begin{equation} \label{eq:SecondOrderTermLimit}
  \lim_{c \to + \infty} \omega_c(h) \ge - \sigma\big( \lambda_* , T^2_K(G(x_*), D G(x_*)) \big)
\end{equation}
otherwise. Utilising this fact and the second order expansion for the augmented Lagrangian we can easily prove the
statement of the theorem.

Indeed, let us show that there exist $\rho, c > 0$, and a neighbourhood $\mathcal{O}(x_*)$ of $x_*$ such that 
$\mathscr{L}(x, \lambda_*, c) \ge \mathscr{L}(x_*, \lambda_*, c) + \rho |x - x_*|^2$ for any 
$x \in A \cap \mathcal{O}(x_*)$. Then taking into account the facts that $\Phi(y, \lambda_*, c) \le 0$ for any
$y \in K$ thanks to \eqref{eq:AugmLagrangian} and $\Phi(G(x_*), \lambda_*, c) = 0$ due to the fact that
$P_K(G(x_*) + (2c)^{-1} \lambda_*) = G(x_*)$ one obtains that
$$
  F(x) \ge \mathscr{L}(x, \lambda_*, c) \ge \mathscr{L}(x_*, \lambda_*, c) + \rho |x - x_*|^2
  = F(x_*) + \rho |x - x_*|^2
$$
for all $x \in A \cap \mathcal{O}(x_*)$ such that $G(x) \in K$, and the proof is complete.

Arguing by reductio ad absurdum suppose that for any $n \in \mathbb{N}$ there exists 
$x_n \in A$ such that $\mathscr{L}(x_n, \lambda_*, n) < \mathscr{L}(x_*, \lambda_*, n) + n^{-1} |x_n - x_*|^2$
and $x_n \in  B(x_*, \min\{ \frac{1}{n}, r_n \})$. With the use of \eqref{eq:AugmLagr_2OrderExpansion} 
for any $n \in \mathbb{N}$ one has 
\begin{multline*}
  0 > \mathscr{L}(x_n, \lambda_*, n) - \mathscr{L}(x_*, \lambda_*, n) - \frac{1}{n} |x_n - x_*|^2 \\
  \ge \max_{\omega \in W} \Big( f(x_*, \omega) - F(x_*) + \langle \nabla_x f(x_*, \omega), u_n \rangle 
  + \frac{1}{2} \langle u_n, \nabla_{xx}^{2} f(x_*, \omega) u_n \rangle \Big) \\
  + \Big\langle \lambda_*, D G(x_*) u_n + \frac{1}{2} D^2 G(x_*)(u_n, u_n) \Big\rangle + \frac{1}{2} \omega_n(u_n)
  - \frac{2}{n} |u_n|^2,
\end{multline*}
where $u_n = x_n - x_*$. Consequently, for any $\alpha \in \alpha(x_*, \lambda_*)$ one has
\begin{multline} \label{eq:AugmLagr_2OrderExp_DanskinDemyanov}
  0 > \mathscr{L}(x_n, \lambda_*, n) - \mathscr{L}(x_*, \lambda_*, n) - \frac{1}{n} |x_n - x_*|^2
 \ge \big\langle \nabla_x \mathcal{L}(x_*, \lambda_*, \alpha), u_n \big\rangle \\
  + \frac{1}{2} \big\langle u_n, \nabla^2_{xx} \mathcal{L}(x_*, \lambda_*, \alpha) u_n \big\rangle
  + \frac{1}{2} \omega_n(u_n) - \frac{2}{n} |u_n|^2,
\end{multline}
since $\alpha \in rca_+(W)$, $\support(\alpha) \subseteq W(x_*)$, and $\alpha(W) = 1$ by the definition of
Danskin-Demyanov multipliers.

Define $h_n = u_n / |u_n|$. Without loss of generality one can suppose that the sequence $\{ h_n \}$ converges to some
$h_* \in \mathbb{R}^d$ with $|h_*| = 1$. Moreover, $h_* \in T_A(x_*)$ by virtue of the facts that the set $A$ is convex
and $\{ x_n \} \subset A$. Let us show that $[L(\cdot, \lambda_*)]'(x_*, h_*) = 0$. 

Indeed, suppose that $[L(\cdot, \lambda_*)]'(x_*, h_*) \ne 0$. Note that by the definition of Lagrange multiplier one
has $[L(\cdot, \lambda_*)]'(x_*, h_*) \ge 0$. Thus, $[L(\cdot, \lambda_*)]'(x_*, h_*) > 0$, which thanks to the
equality $D_x \Phi(G(x_*), \lambda_*, c) = [D G(x_*)]^* \lambda_*$ implies that
\begin{align*}
  \lim_{n \to \infty} 
  \frac{\mathscr{L}(x_* + \beta_n h_n, \lambda_*, c) - \mathscr{L}(x_*,\lambda_*, c) - \beta_n^2}{\beta_n} 
  &= F'(x_*, h_*) + \langle \lambda_*, D G(x_*) h_* \rangle \\
  &= [L(\cdot, \lambda_*)]'(x_*, h_*) > 0,
\end{align*}
for any $c > 0$, where $\beta_n = |u_n| = |x_n - x_*|$ (note that $x_* + \beta_n h_n = x_n$). Consequently, there
exists $n_0 \in \mathbb{N}$ such that $\mathscr{L}(x_n, \lambda_*, 1) > \mathscr{L}(x_*, \lambda_*, 1) + |x_n - x_*|^2$
for all $n \ge n_0$. As was noted above, $\Phi(G(x_*), \lambda_*, c) = 0$ for any $c > 0$. Consequently,
$\mathscr{L}(x_*, \lambda_*, 1) = \mathscr{L}(x_*, \lambda_*, c) = F(x_*)$ for any $c > 0$. Hence bearing in mind the
fact that the function $c \mapsto \mathscr{L}(x, \lambda, c)$ is obviously non-decreasing (see
\eqref{eq:RockWetsAugmLagr}) one obtains that
$$
  \mathscr{L}(x_n, \lambda_*, c) \ge \mathscr{L}(x_n, \lambda_*, 1)
  > \mathscr{L}(x_*, \lambda_*, 1) + |x_n - x_*|^2 = \mathscr{L}(x_*, \lambda_*, c) + |x_n - x_*|^2
$$
for all $c \ge 1$, which contradicts the definition of $x_n$. Thus, $[L(\cdot, \lambda_*)]'(x_*, h_*) = 0$.

Note that $\langle \nabla_x \mathcal{L}(x_*, \lambda_*, \alpha_*), u_n \big\rangle \ge 0$ for all $n \in \mathbb{N}$ due
to the definition of Danskin-Demyanov multiplier and the fact that $u_n = x_n - x_* \in T_A(x_*)$, since $A$ is a convex
set. Hence with the use of \eqref{eq:AugmLagr_2OrderExp_DanskinDemyanov} one obtains that
$$
  0 > \frac{1}{2} \big\langle u_n, \nabla^2_{xx} \mathcal{L}(x_*, \lambda_*, \alpha_*) u_n \big\rangle
  + \frac{1}{2} \omega_n(u_n) - \frac{2}{n} |u_n|^2
$$
for any $n \in \mathbb{N}$. Dividing this inequality by $|u_n|^2$ (recall that $\omega_c(\cdot)$ is positively
homogeneous of degree two), passing to the limit as $n \to \infty$ with the use of \eqref{eq:SecondOrderTermLimit}, and
taking the supremum over all $\alpha \in \alpha(x_*, \lambda_*)$ one finally gets that
$$
  0 
  \ge \sup_{\alpha \in \alpha(x_*, \lambda_*)} \big\langle h_*, \nabla^2_{xx} \mathcal{L}(x_*, \lambda, \alpha) h_*
\big\rangle
  - \sigma\big( \lambda_*, T^2_K(G(x_*), D G(x_*) h_*) \big),
$$
and $D G(x_*) h_* \in T_K(G(x_*))$, $\langle \lambda_*, D G(x_*) h_* \rangle = 0$, and 
$[L(\cdot, \lambda_*)]'(x_*, h_*) = 0$, that is, $h_* \in C(x_*)$ (see \eqref{eq:CriticalConeViaLagrangian}), which 
contradicts \eqref{eq:2OrderSuffCond_SigmaTerm}.
\end{proof}

\begin{remark}
{(i)~Note that the restriction of the function $\sigma(\lambda_*, T^2_K(G(x_*), \cdot))$ to its effective domain is
continuous in the case when $K$ is the second order cone (see~\cite[Formula~(42) and Thrm.~29]{BonnansRamirez})
or the cone $\mathbb{S}^l_{-}$ (see~\cite[Sect.~5.3.5]{BonnansShapiro}).
}

\noindent{(ii)~It should be noted that one can obtain second order sufficient optimality conditions for 
the problem $(\mathcal{P})$ involving the sigma term that are equivalent to the second order growth condition without
the additional assumption that the space $Y$ is finite dimensional. However, this condition is much more cumbersome than
the one stated in the theorem above, since it involves the second order tangent sets to $A$ and $C_{-}(W)$. That is why
we leave the derivation of such second order conditions to the interested reader (see \cite[Sect.~3.3.3]{BonnansShapiro}
for more details in the smooth case). \qed
}
\end{remark}

\section{Optimality conditions for Chebyshev problems with cone constraints}
\label{sect:ChebyshevProblems}

In this section we study optimality conditions for cone constrained Chebyshev problems of the form:
$$
  \min_x\: \max_{\omega \in W} \big| f(x, \omega) - \psi(\omega) \big| \quad 
  \text{subject to} \quad G(x) \in K, \quad x \in A.
  \eqno{(\mathcal{C})}
$$
Here $\psi \colon W \to \mathbb{R}$ is a continuous function. This problem is a particular case of the problem
$(\mathcal{P})$. Indeed, define $\widehat{W} = W \times \{ 1, -1 \}$,
$\widehat{f}(x, \omega, 1) = f(x, \omega) - \psi(\omega)$ and 
$\widehat{f}(x, \omega, -1) = - f(x, \omega) + \psi(\omega)$ for any $\omega \in W$. Then the problem $(\mathcal{C})$
can be rewritten as the problem $(\mathcal{P})$ of the form:
\begin{equation} \label{prob:EquivChebyshevProblem}
  \min_x \: \max_{\widehat{\omega} \in \widehat{W}} \widehat{f}(x, \widehat{\omega}) \quad 
  \text{subject to} \quad G(x) \in K, \quad x \in A.
\end{equation}
Therefore, optimality conditions for the problem $(\mathcal{C})$ can be easily obtained as a direct corollary to
optimality conditions for the problem $(\mathcal{P})$. Nevertheless, it is worth explicitly formulating these
conditions. Furthermore, the following sections can be viewed as a convenient and concise summary of the main results
obtained in this article.

\subsection{First order optimality conditions}

Define $F(x) = \max_{\omega \in W} |f(x, \omega) - \psi(\omega)|$, and let
$W(x) = \{ \omega \in W \mid F(x) = |f(x, \omega) - \psi(\omega)| \}$ the set of points of maximal deviation. Under our
assumptions on $f$, the function $F$ is Hadamard directionally differentiable and its Hadamard directional
derivative has the form
$$
  F'(x, h) = \max_{v \in \partial F(x)} \langle v, h \rangle 
  = \max_{\omega \in W(x)} 
  \Big( \sign(f(x, \omega) - \psi(\omega)) \big\langle \nabla_x f(x, \omega), h \big\rangle \Big)
$$
for any $h \in \mathbb{R}^d$, where 
$\partial F(x) = \co\{ \sign(f(x, \omega) - \psi(\omega)) \nabla_x f(x, \omega) \mid \omega \in W(x) \}$ is the
Hadamard subdifferential of the function $F$ at the point $x$. In this section we suppose that $\sign(0) = \{-1, 1 \}$.

For any $\lambda \in Y^*$ denote by $L(x, \lambda) = F(x) + \langle \lambda, G(x) \rangle$ the Lagrangian for
the problem $(\mathcal{C})$. A vector $\lambda_* \in Y^*$ is called \textit{a Lagrange multiplier} of the problem
$(\mathcal{C})$ at a feasible point $x_*$, if $\lambda_* \in K^*$, $\langle \lambda_*, G(x_*) \rangle = 0$, and
$[L(\cdot, \lambda_*)]'(x_*, h) \ge 0$ for all $h \in T_A(x_*)$. In this case, the pair $(x_*, \lambda_*)$ is called
\textit{a KKT-pair} of the problem $(\mathcal{C})$.

Applying Theorems~\ref{thrm:NessOptCond}--\ref{thrm:OptCond_ConvexCase} to problem
\eqref{prob:EquivChebyshevProblem} one obtains that the following results hold true.

\begin{theorem} \label{thrm:NessOptCond_Chebyshev}
Let $x_*$ be a locally optimal solution of the problem $(\mathcal{C})$ such that RCQ holds at $x_*$. Then: 
\begin{enumerate}
\item{$h = 0$ is a globally optimal solution of the linearised problem
$$
  \min_{h \in \mathbb{R}^d} \max_{v \in \partial F(x_*)} \langle v, h \rangle \quad
  \text{s.t.}	\quad	D G(x_*) h \in T_K\big( G(x_*) \big), \quad h \in T_A(x_*);
$$
}
\vspace{-5mm}
\item{the set of Lagrange multipliers at $x_*$ is a nonempty, convex, bounded, and weak${}^*$ compact subset of $Y^*$.}
\end{enumerate}
\end{theorem}

\begin{theorem} \label{Thrm:ConvexCase_Chebyshev}
Let there exist continuous functions $\phi \colon W \to \mathbb{R}^d$ and $\phi_0 \colon W \to \mathbb{R}$ such that
$f(x, \omega) = \langle \phi(\omega), x \rangle + \phi_0(\omega)$ for all $x$ and $\omega$. Suppose also that the
mapping
$G$ is $(-K)$-convex and $x_*$ is a feasible point of the problem $(\mathcal{C})$. Then:
\begin{enumerate}
\item{$\lambda_*$ is a Lagrange multiplier of $(\mathcal{C})$ at $x_*$ iff $(x_*, \lambda_*)$ is a global saddle point
of the Lagrangian $L(x, \lambda) = F(x) + \langle \lambda, G(x) \rangle$, that is, for all $x \in A$ and 
$\lambda \in K^*$ one has $L(x, \lambda_*) \ge F(x_*) \ge L(x_*, \lambda)$;
}

\item{if a Lagrange multiplier of the problem $(\mathcal{C})$ at $x_*$ exists, then $x_*$ is a globally optimal solution
of $(\mathcal{C})$; conversely, if $x_*$ is a globally optimal solution of the problem $(\mathcal{C})$ and Slater's
condition $0 \in \interior\{ G(A) - K \}$ holds true, then there exists a Lagrange multiplier of $(\mathcal{C})$ 
at $x_*$.
}
\end{enumerate}
\end{theorem}

\begin{theorem} \label{thrm:SuffOptCond_Chebyshev}
Let $x_*$ be a feasible point of the problem $(\mathcal{C})$. If
\begin{equation} \label{eq:SuffOptCond_Chebyshev}
  \max_{v \in \partial F(x_*)} \langle v, h \rangle > 0
  \quad \forall h \in T_A(x_*) \setminus \{ 0 \} \colon D G(x_*) h \in T_K\big( G(x_*) \big),
\end{equation}
then the first order growth condition holds at $x_*$. Conversely, if the first order growth condition and RCQ hold at
$x_*$, then inequality \eqref{eq:SuffOptCond_Chebyshev} is valid.
\end{theorem}

Next we present several equivalent reformulations of necessary and sufficient optimality conditions for the problem
$(\mathcal{C})$ from Theorems~\ref{thrm:NessOptCond_Chebyshev} and \ref{thrm:SuffOptCond_Chebyshev}. Recall that
$$
  \mathcal{N}(x) = [D G(x)]^* (K^* \cap \linhull(G(x))^{\perp}) 
  = \{ i(\lambda \circ D G(x)) \mid \lambda \in K^*, \: \langle \lambda, G(x) \rangle = 0 \},
$$
where $i$ is the natural isomorphism between $(\mathbb{R}^d)^*$ and $\mathbb{R}^d$. For any $c \ge 0$ a penalty function
for the problem $(\mathcal{C})$ is denoted by $\Phi_c(x) = F(x) + c \dist(G(x), K)$. By
Lemma~\ref{lem:ConeConstrPenFunc_Subdiff} this function is Hadamard subdifferentiable and for any $x$ such that
$G(x) \in K$ its Hadamard subdifferential has the form
$$
  \partial \Phi_c(x) = \partial F(x)
  + c \Big\{ [D G(x)]^* y^* \in \mathbb{R}^d \Bigm| y^* \in Y^*, \: \| y^* \| \le 1, \:
  \langle y^*, y - G(x) \rangle \le 0 \enspace \forall y \in K \Big\}.
$$
Let us reformulate alternance optimality conditions in terms of the problem $(\mathcal{C})$. Let, as earlier, 
$Z \subset \mathbb{R}^d$ be any collection of $d$ linearly independent vectors, and
$\eta(x)$ and $n_A(x)$ be any sets such that $\mathcal{N}(x) = \cone \eta(x)$ and 
$N_A(x) = \cone n_A(x)$.

\begin{definition}
Let $p \in \{ 1, \ldots, d + 1 \}$ be given and $x_*$ be a feasible point of the problem $(\mathcal{C})$. One says that
\textit{a $p$-point alternance} exists at $x_*$, if there exist $k_0 \in \{ 1, \ldots, p \}$, 
$i_0 \in \{ k_0 + 1, \ldots, p \}$, vectors
\begin{gather} \label{eq:AlternanceDef_Ch}
  V_1, \ldots, V_{k_0} \in 
  \Big\{ \sign\big(f(x_*, \omega) - \psi(\omega) \big)\nabla_x f(x_*, \omega) \Bigm| \omega \in W(x_*) \Big\}, \\
  V_{k_0 + 1}, \ldots, V_{i_0} \in \eta(x_*), \quad
  V_{i_0 + 1}, \ldots, V_p \in n_A(x_*), \label{eq:AlternanceDef_Ch_2}
\end{gather}
and vectors $V_{p + 1}, \ldots, V_{d + 1} \in Z$ such that the $d$th-order determinants $\Delta_s$ of the matrices
composed of the columns $V_1, \ldots, V_{s - 1}, V_{s + 1}, \ldots V_{d + 1}$ satisfy the following conditions:
\begin{gather*}
  \Delta_s \ne 0, \quad s \in \{ 1, \ldots, p \}, \quad
  \sign \Delta_s = - \sign \Delta_{s + 1}, \quad s \in \{ 1, \ldots, p - 1 \}, \\
  \Delta_s = 0, \quad s \in \{ p + 1, \ldots d + 1 \}.
\end{gather*}
Such collection of vectors $\{ V_1, \ldots, V_p \}$ is called a $p$-point alternance at $x_*$. Any $(d + 1)$-point
alternance is called \textit{complete}. If the set in the right-hand side of \eqref{eq:AlternanceDef_Ch} is replaced by
$\partial F(x_*)$ and the sets $\eta(x_*)$ and $n_A(x_*)$ in \eqref{eq:AlternanceDef_Ch_2} are replaced by
$\mathcal{N}(x_*)$ and $N_A(x_*)$ respectively, then one says that \textit{a generalised $p$-point alternance} exists at
$x_*$, and the corresponding collection of vectors $\{ V_1, \ldots, V_p \}$ is called \textit{a generalised $p$-point
alternance} at $x_*$.
\end{definition}

Finally, if $x_*$ is a feasible point of $(\mathcal{C})$, then any collection of vectors $V_1, \ldots, V_p$
with $p \in \{ 1, \ldots, d + 1 \}$ satisfying \eqref{eq:AlternanceDef_Ch}, \eqref{eq:AlternanceDef_Ch_2}, and such
that $\rank([V_1, \ldots, V_p]) = \rank([V_1, \ldots, V_{i - 1}, V_{i + 1}, \ldots, V_p]) = p - 1$ for any 
$i \in \{ 1, \ldots, p \}$ is called a $p$-point \textit{cadre} for the problem $(\mathcal{C})$ at $x_*$. It is easily
seen that a collection $V_1, \ldots, V_p$ satisfying \eqref{eq:AlternanceDef_Ch}, \eqref{eq:AlternanceDef_Ch_2} is 
a $p$-point cadre at $x_*$ iff $\rank([V_1, \ldots, V_p]) = p - 1$ and
$\sum_{i = 1}^p \beta_i V_i = 0$ for some $\beta_i \ne 0$, $i \in \{ 1, \ldots, p \}$. Any such $\{ \beta_i \}$ are
called \textit{cadre multipliers}.

Applying the main results of Sections~\ref{subsect:Subdifferentials_ExactPenaltyFunc} and
\ref{subsect:Alternance_Cadre} to problem \eqref{prob:EquivChebyshevProblem} one obtains the following six equivalent
reformulations of necessary/sufficient optimality conditions for the cone constrained Chebyshev problem
$(\mathcal{C})$.

\begin{theorem} \label{thrm:EquivNeccOptCond_Chebyshev}
Let $x_*$ be a feasible point of the problem $(\mathcal{C})$. Then the following statements are equivalent:
\begin{enumerate}
\item{there exists a Lagrange multiplier of $(\mathcal{C})$ at $x_*$;}

\item{there exists $v \in \partial F(x_*)$ and $\lambda_* \in K^*$ such that $\langle \lambda_*, G(x_*) \rangle = 0$
and $\langle v, h \rangle + \langle \lambda_*, D G(x_*) h \rangle \ge 0$ for all $h \in T_A(x_*)$;}

\item{$0 \in \partial F(x_*) + \mathcal{N}(x_*) + N_A(x_*)$;}

\item{$0 \in \partial \Phi_c(x_*) + N_A(x_*)$ for some $c > 0$;}

\item{a $p$-point alternance exists at $x_*$ for some $p \in \{1, \ldots, d + 1 \}$;}

\item{a $p$-point cadre with positive cadre multipliers exists at $x_*$ for some $p \in \{ 1, \ldots, d + 1 \}$.}
\end{enumerate}
\end{theorem}

\begin{theorem} \label{thrm:EquivSuffOptCond_Chebyshev}
Let $x_*$ be a feasible point of the problem $(\mathcal{C})$. Then the following statements are equivalent:
\begin{enumerate}
\item{sufficient optimality condition \eqref{eq:SuffOptCond_Chebyshev} holds true at $x_*$;}

\item{$0 \in \interior(\partial F(x_*) + \mathcal{N}(x_*) + N_A(x_*))$;}

\item{$0 \in \interior(\partial \Phi_c(x_*) + N_A(x_*))$ for some $c > 0$;}

\item{$\Phi_c$ satisfies the first order growth condition on $A$ at $x_*$ for some $c \ge 0$.}
\end{enumerate}
Moreover, all these conditions are satisfied, if a complete alternance exists at $x_*$. In addition, if one of 
the following assumptions is valid
\begin{enumerate}
\item{$\interior \partial F(x_*) \ne \emptyset$,}

\item{$\mathcal{N}(x_*) + N_A(x_*) \ne \mathbb{R}^d$ and either $\interior \mathcal{N}(x_*) \ne \emptyset$ 
or $\interior N_A(x_*) \ne \emptyset$,}

\item{$N_A(x_*) = \{ 0 \}$ and there exists $w \in \relint \mathcal{N}(x_*) \setminus \{ 0 \}$ such that 
$0 \in \partial F(x_*) + w$ (in particular, it is sufficient to suppose that $0 \notin \partial F(x_*)$ or the cone
$\mathcal{N}(x_*)$ is pointed),}

\item{$\mathcal{N}(x_*) = \{ 0 \}$ and there exists $w \in \relint N_A(x_*) \setminus \{ 0 \}$ such that 
$0 \in \partial F(x_*) + w$,}
\end{enumerate}
then the four equivalent sufficient optimality conditions stated in this theorem are satisfied iff a generalised
complete alternance exists at $x_*$.
\end{theorem}

\begin{remark} \label{rmrk:ChebyshevSquared}
It should be noted that the Chebyshev problem $(\mathcal{C})$ can be reduced to the minimax problem $(\mathcal{P})$ in
a different way. Namely, define $F_2(\cdot) = \max_{\omega \in W} 0.5 (f(\cdot, \omega) - \psi(\omega))^2$ and consider
the following cone constrained problem:
\begin{equation} \label{prob:ChebyshevSquared}
  \min \: F_2(x) \quad \text{subject to} \quad G(x) \in K, \quad x \in A.
\end{equation}
Note that $W(x) = \{ \omega \in W \mid F_2(x) = 0.5 (f(x, \omega) - \psi(\omega))^2 \}$. Furthermore, one has $
F(x_1) \ge F(x_2)$ for some $x_1$ and $x_2$ if and only if 
$$
  |f(x_1, \omega_*) - \psi(\omega_*)| \ge |f(x_2, \omega) - \psi(\omega)|
  \quad \forall \omega_* \in W(x_1), \quad \forall \omega \in W,
$$
while this inequality is satisfied if and only if
$$
  \frac{1}{2} \big( f(x_1, \omega_*) - \psi(\omega_*) \big)^2 \ge \frac{1}{2} \big( f(x_2, \omega) - \psi(\omega) \big)
  \quad \forall \omega_* \in W(x_1), \quad \forall \omega \in W,
$$
or, equivalently, if and only if $F_2(x_1) \ge F_2(x_2)$. Therefore, $x_*$ is a locally/globally optimal solution of the
problem $(\mathcal{C})$ iff $x_*$ is a locally/globally optimal solution of problem \eqref{prob:ChebyshevSquared}.
Moreover, it is easily seen that the function $F_2$ is Hadamard subdifferentiable, $F_2'(x, h) = \max_{v \in \partial
F_2(x)} \langle v, h \rangle$ for all $h \in \mathbb{R}^d$, where
$$
  \partial F_2(x) = \big\{ (f(x, \omega) - \psi(\omega)) \nabla_x f(x, \omega) \bigm| \omega \in W(x) \big\},
$$
that is, $F_2'(x, \cdot) = F(x) F'(x, \cdot)$ and $\partial F_2(x) = F(x) \partial F(x)$ for all $x$. Consequently,
$\lambda_*$ is a Lagrange multiplier of the problem $(\mathcal{C})$ at a feasible point $x_*$ such that $F(x_*) \ne 0$
iff $F(x_*) \lambda_*$ is a Lagrange multiplier of problem \eqref{prob:ChebyshevSquared} at $x_*$. Therefore, replacing
$\partial F(x_*)$ with $F(x_*) \partial F(x_*)$ in 
Theorems~\ref{thrm:NessOptCond_Chebyshev}--\ref{thrm:EquivSuffOptCond_Chebyshev} one obtains equivalent
necessary/sufficient optimality conditions for the cone constrained Chebyshev problem $(\mathcal{C})$. \qed
\end{remark}

\subsection{Second order optimality conditions}

Let us finally formulate second order optimality conditions for the problem $(\mathcal{C})$. To this end, suppose that
the mapping $G$ is twice continuously Fr\'{e}chet differentiable in a neighbourhood of a given point $x_*$, the function
$f(x, \omega)$ is twice differentiable in $x$ in a neighbourhood $\mathcal{O}(x_*)$ of $x_*$ for any $\omega \in W$, and
the function $\nabla^2_{xx} f(\cdot)$ is continuous on $\mathcal{O}(x_*) \times W$. 

Firstly, note that if for a feasible point $x_*$ one has $F(x_*) = 0$, then $x_*$ is a globally optimal solution of the
problem $(\mathcal{C})$, since this function is nonnegative. Therefore, below we suppose that the optimal value of
the problem $(\mathcal{C})$ is strictly positive.

Let $(x_*, \lambda_*)$ be a KKT-pair of the problem $(\mathcal{C})$. Then by the second part of
Theorem~\ref{thrm:EquivNeccOptCond_Chebyshev} there exist $v \in \partial F(x_*)$ and
$\lambda_* \in K^*$ such that $\langle \lambda_*, G(x_*) \rangle = 0$ and 
$\langle v, h \rangle + \langle \lambda_*, D G(x_*) h \rangle \ge 0$ for all $h \in T_A(x_*)$. Then by the definition of
$\partial F(x_*)$ there exist $k \in \mathbb{N}$, $\omega_i \in W(x_*)$, and $\alpha_i \ge 0$, 
$i \in \{ 1, \ldots, k \}$, such that 
$$
  v = \sum_{i = 1}^k \alpha_i \sign(f(x, \omega_i) - \psi(\omega_i)) \nabla_x f(x, \omega_i), \quad
  \sum_{i = 1}^k \alpha_i = 1.
$$
Let $\alpha = \sum_{i = 1}^k \sign(f(x, \omega_i) - \psi(\omega_i)) \alpha_i \delta(\omega_i)$ be the discrete Radon
measure on $W$ corresponding to $\alpha_i$ and $\omega_i$. Then
$$
  \left\langle \int_W \nabla_x f(x, \omega) d \alpha(\omega), h \right\rangle + 
  \langle \lambda_*, D G(x_*) h \rangle \ge 0 \quad \forall h \in T_A(x_*),
  \quad |\alpha|(W) = 1,
$$
where $|\alpha| = \alpha^+ + \alpha^-$ is the total variation of the measure $\alpha$, while $\alpha^+$ and $\alpha^-$
are positive and negative variations of $\alpha$ respectively (see, e.g. \cite{Folland}). Denote by 
$\alpha(x_*, \lambda_*)$ the set of all Radon measures $\alpha \in rca(W)$ satisfying the conditions above and the
inclusions 
\begin{align*}
  \support(\alpha^+) \subseteq W_+(x_*) &:= \{ \omega \in W(x_*) \mid f(x_*, \omega) - \psi(\omega) > 0 \}, \\
  \support(\alpha^-) \subseteq W_-(x_*) &:= \{ \omega \in W(x_*) \mid f(x_*, \omega) - \psi(\omega) < 0 \}.
\end{align*}
One can easily verify that $\alpha(x_*, \lambda_*)$ is a convex, bounded and weak${}^*$ closed (and, therefore,
weak${}^*$ compact) set. Any measure $\alpha \in \alpha(x_*, \lambda_*)$ is called \textit{a Danskin-Demyanov
multiplier} corresponding to the KKT-pair $(x_*, \lambda_*)$.

For any $x \in \mathbb{R}^d$, $\lambda \in Y^*$, and $\alpha \in rca(W)$ denote by
$$
  \mathcal{L}(x, \lambda, \alpha) = \int_W f(x, \omega) d \alpha(\omega) + \langle \lambda, G(x) \rangle
$$
the integral Lagrangian for the problem $(\mathcal{C})$. It is easily seen that $\alpha_*$ is a Danskin-Demyanov
multiplier corresponding to $(x_*, \lambda_*)$ if and only if $|\alpha_*|(W) = 1$, 
$\support(\alpha_*^{\pm}) \subseteq W_{\pm}(x_*)$, 
and $\langle \nabla_x \mathcal{L}(x_*, \lambda_*, \alpha_*), h \rangle \ge 0$ for all $h \in T_A(x_*)$.

Applying the main results of Section~\ref{sect:SecondOrderOptCond} to problem~\eqref{prob:EquivChebyshevProblem} one
gets the following necessary/sufficient second order optimality conditions for the problem $(\mathcal{C})$.

\begin{theorem}
Let $W = \{ 1, \ldots, m \}$, $f(x, i) = f_i(x)$ for any $i \in W$, and $x_* \in \interior A$ be a locally optimal
solution of the problem $(\mathcal{C})$ such that RCQ holds true at $x_*$. Then for any vector $h$ from the critical
cone
$$
  C(x_*) = \Big\{ h \in T_A(x_*) \Bigm| D G(x_*) h \in T_K(G(x_*)), \enspace F'(x_*, h) \le 0 \Big\}
$$
and for any convex set $\mathcal{T}(h) \subseteq T_K^2(G(x_*), D G(x_*) h)$ one has
$$
  \sup_{\lambda \in \Lambda(x_*)} \Big\{ 
  \sup_{\alpha \in \alpha(x_*, \lambda)} \big\langle h, \nabla^2_{xx} \mathcal{L}(x_*, \lambda, \alpha) h \big\rangle
  - \sigma(\lambda, \mathcal{T}(h)) \Big\} \ge 0.
$$
Furthermore, if $\Lambda(x_*) = \{ \lambda_* \}$, then for any $h \in C(x_*)$ one has
$$
  \sup_{\alpha \in \alpha(x_*, \lambda_*)} \big\langle h, \nabla^2_{xx} \mathcal{L}(x_*, \lambda_*, \alpha) \big\rangle
  - \sigma\big( \lambda_*, T^2_K(G(x_*), D G(x_*) h) \big) \ge 0.
$$
\end{theorem}

\begin{theorem}
Let $x_* \in \interior A$ be a feasible point of the problem $(\mathcal{C})$ such that $\Lambda(x_*) \ne \emptyset$ and
for any $h \in C(x_*) \setminus \{ 0 \}$ one can find $\lambda \in \Lambda(x_*)$ and $\alpha \in \alpha(x_*, \lambda)$
such that $\langle h, \nabla^2_{xx} \mathcal{L}(x_*, \lambda, \alpha) h \rangle > 0$. Then $x_*$ is a locally optimal
solution of the problem $(\mathcal{C})$ at which the second order growth condition holds true.
\end{theorem}

\begin{theorem}
Let $Y$ be a finite dimensional Hilbert space, the cone $K$ be second order regular, and $(x_*, \lambda_*)$ be a
KKT-pair of the problem $(\mathcal{C})$ such that the restriction of the function 
$\sigma(\lambda_*, T^2_K(G(x_*), \cdot))$ to its effective domain is upper semicontinuous. Suppose also that
$$
  \sup_{\alpha \in \alpha(x_*, \lambda_*)} 
  \big\langle h, \nabla^2_{xx} \mathcal{L}(x_*, \lambda_*, \alpha) h \big\rangle
  - \sigma\big( \lambda_*, T^2_K(G(x_*), D G(x_*) h) \big) > 0
$$
for all $h \in C(x_*) \setminus \{ 0 \}$. Then $x_*$ is a locally optimal solution of the problem $(\mathcal{C})$ at
which the second order growth condition holds true.
\end{theorem}

\section{Conclusions}

In this article we presented a unified theory of first and second order necessary and sufficient optimality
conditions for minimax and Chebyshev optimisation problems with cone constraints, including such problems with equality
and inequality constraints, problems with second order cone constraints, problems with semidefinite constraints, as well
as problems with semi-infinite constraints. We analysed different, but equivalent forms of first order optimality
conditions and demonstrated how they can be reformulated in a more convenient way for particular classes of
cone constrained minimax problems. These results can be utilised to develop new methods for solving cone constrained
minimax and Chebyshev problems based on structural properties of optimal solutions (cf. such methods for discrete
minimax problems \cite{ConnLi92}, problems of rational $\ell_{\infty}$-approximation \cite{BarrodalePowellRoberts},
and synthesis of a rational filter \cite{MalozemovTamasyan}). A development of such methods is an interesting topic of
future research.

\section*{Acknowledgements}

The author wishes to express his sincere gratitude to prof. V.N.~Malozemov and the late prof. V.F.~Demyanov. Their
research on minimax problems and alternance optimality conditions, as well as inspiring lectures, were the main source
of inspiration for writing this article. In particular, the main results of Section~\ref{subsect:Alternance_Cadre} are
a natural continuation of their research on alternance optimality conditions
\cite{DemyanovMalozemov_Alternance,DemyanovMalozemov_Collect}.

%%%
%	Bibliography
%%%

\bibliographystyle{abbrv}  
\bibliography{AlternanceConditions_bibl}

\end{document}